\documentclass[11pt,a4paper]{amsart}
\usepackage{latexsym,amsfonts,amsthm,amsmath,mathrsfs,amssymb}
\usepackage[dvips]{color}
\usepackage{pdfsync}
\usepackage[dvips,final]{graphics}
\usepackage{epstopdf}
\usepackage{amscd}
\usepackage{pstricks, pst-node, pst-text, pst-3d}

\theoremstyle{plain}
\newtheorem{theorem}{Theorem}[section]
\newtheorem{lem}[theorem]{Lemma}
\newtheorem{prop}[theorem]{Proposition}

\theoremstyle{definition}
\newtheorem{defi}{Definition}[section]

\newtheorem{exa}{Example}[section]
\theoremstyle{remark}
\newtheorem{rem}[theorem]{Remark}

\newcommand{\R}{\mathbb{R}}
\newcommand{\N}{\mathbb{N}}
\newcommand{\M}{\mathbb{M}}

\newcommand{\Q}{\mathcal{Q}}
\newcommand{\W}{\mathbb{W}}
\newcommand{\V}{\mathbb{V}}

\newcommand{\C}{\mathbb{C}}
\newcommand{\E}{\mathcal{E}}
\newcommand{\HH}{\mathbb{H}}
\newcommand{\B}{\mathcal{B}}
\newcommand{\G}{\mathbb{G}}
\newcommand{\eps}{\epsilon}

\newcommand{\mcal}{\mathcal}

\newcommand{\vp}{\varphi}

\newcommand{\mfrak}{\mathfrak}

\newcommand{\graph}[1]{\mathrm{graph}\,(#1)}

\newcommand{\average}{{\mathchoice {\kern1ex\vcenter{\hrule height.4pt
width 6pt
depth0pt} \kern-9.7pt} {\kern1ex\vcenter{\hrule height.4pt width 4.3pt
depth0pt}
\kern-7pt} {} {} }}
\newcommand{\ave}{\average\int}
\newcommand{\res}{\mathop{\hbox{\vrule height 7pt width .5pt depth 0pt
\vrule height .5pt width 6pt depth 0pt}}\nolimits}

\title{Intrinsic Lipschitz graphs in Carnot groups of step 2}

\date{\today}
\author[D. Di Donato]{Daniela Di Donato}
\address{Daniela Di Donato: Dipartimento di Matematica\\Universit\`a di Trento\\ Via Sommarive 14\\ 38123, Povo (Trento) - Italy\\} 
\email{daniela.didonato@unitn.it}
    \thanks{D.D.D. is supported by University of Trento, Italy.}

\date{\today}


\begin{document}

\begin{abstract}
We focus our attention on the notion of intrinsic Lipschitz graphs, inside a special class of metric spaces i.e. the Carnot groups. More precisely, we provide a characterization of locally intrinsic Lipschitz functions in Carnot groups of step $2$ in terms of their intrinsic distributional gradients. 

\end{abstract}

\maketitle

\section{Introduction}
The notion of rectifiable set is a key one in calculus of variations and in geometric measure theory. To develop a satisfactory theory of rectifiable
sets inside Carnot groups has been the object of much research in the last years. For a general theory of rectifiable sets in euclidean spaces one can see  \cite{biblio5, biblio4, biblio16} while a general theory in metric spaces can be found in \cite{biblioAMBkirc}. 

Rectifiable sets are classically defined as contained in  the countable union of $\C^1$ submanifolds. From Rademacher type theorem proved in a special subclass of Carnot groups which including Carnot groups of step $2$ (see Theorem 4.3.5 in \cite{biblio21}), it follows (as Euclidean spaces) the equivalence of defining one codimentional rectifiable sets with intrinsic $\C^1$ submanifolds or with intrinsic Lipschitz graphs (see Proposition 4.4.4 in \cite{biblio21}). 
Hence, understanding these objects inside Carnot groups, that naturally take the role of $\C^1$ submanifolds or Lipschitz graphs, is a preliminary task in developing a satisfactory theory of rectifiable sets inside Carnot groups.

In this paper we focus our attention on the notion of intrinsic Lipschitz graphs inside the Carnot groups. A Carnot group $\G$ is a connected, simply connected and stratified Lie group and has a sufficiently rich compatible underlying structure, due to the existence of intrinsic
families of left translations and dilations and depending on the horizontal vector fields which generate the horizontal layer. We call intrinsic any notion depending directly by the structure and geometry of $\G$. For a complete description of Carnot groups \cite{biblio3, biblioLeDonne, biblioLIBROSC} are recommended.

Euclidean spaces are commutative Carnot groups and are  the only commutative ones. The simplest but, at the same time, non-trivial instances of  non-Abelian Carnot groups are provided by the Heisenberg groups $\mathbb{H}^n$ (see for instance \cite{biblio3, biblioLIBROSC}).

We begin recalling that a intrinsic regular hypersurface (i.e. a topological codimension 1 surface) $S\subset \G$ is locally defined as a non critical level set of a $\C^1$ intrinsic function. More precisely, there exists a continuous function $f:\G \to \R$ such that locally $S=\{ p\in \G : f(p)=0\}$ and the intrinsic gradient $\nabla _\G f=(X_1f,\dots , X_m f)$ exists in the sense of distributions and it is continuous and non vanishing on $S$. In a similar way, a $k$-codimensional regular surface $S\subset \G$ is locally defined as a non critical level set of a $\C^1$ intrinsic vector function $f=(f_1,\dots , f_k):\G\to \R^k$.


On the other hand, the intrinsic graphs came out naturally in \cite{biblio6}, while studying level sets of Pansu differentiable functions from $\mathbb{H}^n$ to $\R$.  The simple idea of intrinsic graph is the following one: let $\mathbb{M}$ and $\W$ be complementary subgroups of $\G$, i.e. homogeneous subgroups such that  $\W \cap \mathbb{M}= \{ 0 \}$ and $\G=\W\cdot \mathbb{M}$ (here $\cdot $ indicates the group operation in $\G$ and $0$ is the unit element), then the intrinsic left graph of $\phi :\W\to \mathbb{M}$ is the set
\begin{equation*}
\graph {\phi}:=\{ a\cdot \phi (a) \, |\, a\in \W \}.
\end{equation*}
Hence the existence of intrinsic graphs depends on the possibility of splitting $\G$ as a product of complementary subgroups hence it depends on the structure of the Lie algebra associated to $\G$.

Specifically the notion of intrinsic Lipschitz graphs has been introduced with different degrees of generality in \cite{biblio2, biblio21, biblio6, bibliofsscINTRINSICLIP}. In \cite{biblio22} the authors provide a comprehensive presentation of this theory. A function $\phi : \mathcal O \subset \W\to \mathbb{M}$ is said to be intrinsic Lipschitz if it is possible to put, at each point $p\in \graph {\phi}$, an intrinsic cone (see Definition \ref{FMSdefi2.3.1part}) with vertex $p$, axis $\M$ and fixed opening, intersecting $\graph {\phi}$ only at $p$. 
We call a set $S\subset \G$ an intrinsic Lipschitz graph if there exists an intrinsic Lipschitz function $\phi : \mathcal O \subset \W \to \M$ such that $S= \graph{\phi}$ for suitable complementary subgroups $\W$ and $\M$ of $\G$.

Through the implicit function theorem (see \cite{biblioMAGNANI, biblio7} ), it is known $k$-codimensional regular surface can be locally represented as intrinsic graph by a function $\phi : \mathcal O \subset \W \to \M$. In \cite{biblioDDD} we establish that an ap\-pro\-priate notion of differentiability, 
denoted intrinsic differentiability (see Definition \ref{d3.2.1}), 
 is the additional assumption on $\phi$ to obtain the opposite implication. More precisely, a function is intrinsic differentiable if it is well approximated by appropriate linear type functions, denoted intrinsic linear functions.  In particular, when $\W$ and $\M$ are both normal subgroup, this notion corresponds to that of Pansu differentiability.

When $\M$ is one dimensional we can identify $\phi:\mathcal O \subset \W \to \M$ with  a real valued continuous function defined on a one codimensional homogeneous subgroup of $\G$ (see Remark \ref{remIMPORT}). In this case in Heisenberg groups  the parameterization $\phi$ is weak solution of a system of non-linear first order PDEs $D^\phi \phi =w$, where $w $ is a continuous map and $D^\phi$ are suitable \lq derivatives\rq\,  $D^\phi _j\phi$ of $\phi$ (see \cite{biblio1, biblio27, biblio12}). The non linear first order differential operators $D^\phi _j$ were introduced by Serra Cassano et al. in the context of Heisenberg groups $\HH^n$ (see \cite{biblioLIBROSC} and the references therein).  Following the notations in \cite{biblioLIBROSC} the operators $D^\phi _j$ are denoted as \emph{intrinsic derivatives} of $\phi$ and $D^\phi \phi$, the vector of the intrinsic derivatives of $\phi$, is the  intrinsic gradient of $\phi$.  
 In the first Heisenberg group $\mathbb{H}^1$  the  operator $D^\phi $ reduces to the classical Burgers' operator.


In \cite{biblioDDD, biblioDDD1} we study real valued functions defined on a one codimensional homogeneous subgroup of Carnot groups of step 2 (see Section 5.1,  \cite{biblio3}) which including  the Heisenberg groups.   We define an appropriate notion of intrinsic derivative in this setting and we extend Theorem 1.3 and Theorem 5.7 in \cite{biblio1} proved in $\HH^n$.
 Precisely, we prove that the intrinsic graph of continuous map $\phi$ is  a regular hypersurface if and only if $\phi $ and its intrinsic gradient can be uniformly approximated by  $\C^1$ functions and $\phi$ is distributional solution of PDEs' system $D^\phi \phi =w$, with $w$ is a continuous map.  Moreover, we also show that these assumptions are equivalent to the fact that $\phi $ is broad* solution of $D^\phi \phi =w$ (see Definition \ref{defbroad*}) and it is little 1/2-H\"older continuous (see Proposition \ref{propCharacterizationRegularSurfaces} (1)). In $\HH^1$ the notion of broad* solution extends the classical notion of broad solution for Burger's equation through characteristic curves provided  $\phi$ and $w$ are locally Lipschitz continuous. In our case  $\phi$ and $w$  are supposed to be only continuous then the classical theory breaks down. On the other hand broad* solutions of the system $D^\phi \phi =w$ can be constructed with a continuous datum $w$.

In this paper we generalize Theorem 1.1 in \cite{biblio17} valid in Heisenberg groups to Carnot groups of step 2.  The main difference between 2 step Carnot groups and Heisenberg groups is that in $\HH^n$ there is only one vertical (i.e. non-horizontal) coordinate, whereas for 2 step Carnot groups there can be many.

Precisely we show that a locally intrinsic Lipschitz map $\phi$ is a continuous distributional solution of the non-linear first order PDEs' system $D^\phi \phi =w$, where $w$ is only a measurable function (see Theorem \ref{lemma5.6}). This is a direct consequence of the fact that $\phi $ can be approximated by a sequence of smooth maps, with pointwise convergent intrinsic gradient (see Theorem \ref{cmpssteo17}).

Moreover the opposite implication is true when $\phi$ is also locally $1/2 $-H\"older continuous map along the vertical components (see Theorem \ref{lemma5.4bcsc}). In order to establish this statement 
we need a preliminary result: Lemma \ref{lemma5.1bcsc}. Here we prove that a locally $1/2 $-H\"older continuous map along the vertical components, which is also a continuous distributional solution of the system $D^\phi \phi =w$, is a Lipschitz map in classical sense along any characteristic curve of the vector fields $D^\phi$.

Finally in the last section, we consider the system $D^\phi \phi =w$, where $w$ is supposed to be a continuous map and \textit{not} measurable function as opposed to the previous sections. We give a characterization of $\G$-regular hypersurafces in terms of distributional solution of this system. Our strategy will be to prove that each continuous distributional solution of the system $D^\phi \phi =w$ is a broad* solution and vice versa (see Proposition \ref{propCharacterizationRegularSurfaces}).

The plan of the work is the following:\\
$\mathbf{Section \, 2.}$ we provide the definitions and some properties about the differential calculus within Carnot groups. In particular we give the general definitions and some basic tools of intrinsic differentiable and intrinsic Lipschitz maps.\\
$\mathbf{Section \, 3.}$ is specialized to Carnot groups of step 2 (see Chapter 3, \cite{biblio3}).\\
$\mathbf{Section \, 4.}$ is dedicated to Caccioppoli sets in Carnot groups of step 2. We show that the boundary of sets with finite $\G$-perimeter and having a bound on the measure theoretic normal is the intrinsic graph of an intrinsic Lipschitz function (see Theorem \ref{teo11montivittoneTO}).\\
$\mathbf{Section \, 5. \, and \, 6.}$  contains  an area formula beside the spherical Hausdorff measure for the graph of an intrinsic Lipschitz function (see Theorem \ref{citteo16OK}) and we characterize intrinsic Lipschitz functions as maps which can be approximated by a sequence of smooth maps.\\
$\mathbf{Section \, 7.}$ contains the main results of this paper, i.e. Theorem \ref{lemma5.6} and Theorem \ref{lemma5.4bcsc}.\\
$\mathbf{Section \, 8.}$ is dedicated to $\G$-regular hypersurafces. We show that if we consider a locally little $1/2 $-H\"older continuous map along the vertical components, then it is a  distributional solution of the system $D^\phi \phi =w$, where $w$ is a given continuous function if and only if its graph is a $\G$-regular hypersurface (see Theorem \ref{CharacterizationRegularSurfaces}).
 \\
 
$\mathbf{Acknowledgements.}$ We wish to express our gratitude to R.Serapioni and F.Serra Cassano, for having signaled us this problem and for many invaluable discussions during our PhD at University of Trento. We thank B.Franchi, A.Pinamonti and D.Vittone for useful discussions and important suggestions on the subject. 

\section{Notations and preliminary results}

\subsection{Carnot groups}\label{Carnotinizio} We begin recalling briefly the definition of Carnot groups. For a general account see e.g. \cite{biblio3, biblio22, biblioLeDonne, biblioLIBROSC}. 

A Carnot group $\G=(\G, \cdot, \delta_\lambda)$ of step $\kappa$  is a connected and simply connected Lie group whose Lie algebra $\mathfrak g$ admits a stratification, i.e. a direct sum decomposition $\mathfrak g=V_1\,  \oplus \, V_2 \, \oplus \dots \oplus \, V_\kappa$. The stratification has the further property  that the entire Lie algebra $\mathfrak g$ is generated by its first layer $V_1$, the so called horizontal layer, that is 
\[
\left\{
\begin{array}{l}
[V_1, V_{l-1}]= V_l \quad \mbox{if } \, 2\leq l \leq \kappa \\ \hspace{0,05 cm}  [V_1,V_\kappa]=\{0 \}
\end{array}
\right.
\]
We denote by $N$ the dimension of $\mathfrak g$ and by $n_s$ the dimension of $V_s$.  

The exponential map $\exp :\mathfrak g \to \G$ is a global diffeomorphism from $\mathfrak g$ to $\G$.
Hence, if we choose a basis $\{X_1,\dots , X_N\}$ of $\mathfrak g$,  any $p\in \G$ can be written in a unique way as $p=\exp (p_1X_1+\dots +p_NX_N)$ and we can identify $p$ with the $N$-tuple $(p_1,\dots , p_N)\in \R^N$ and $\G$ with $(\R^N,\cdot, \delta_\lambda)$. The identity of $\G$ is the origin of $\R^N$.
 


For any $\lambda >0$, the (non isotropic) dilation $\delta _\lambda :\G\to \G$ are automorhisms of $\G$ and are  defined as 
\[
\delta _\lambda (p_1,\dots , p_N)=(\lambda ^{\alpha _1} p_1,\dots , \lambda ^{\alpha _N}p_N)
\]
where $\alpha _l \in \N$ is called homogeneity of the variable $p_l$ in $\G$ and is given by $\alpha _l=j$ whenever $m_{j-1}<l\leq m_j$ with  $m_s-m_{s-1}=n_s$. Hence $1=\alpha _1=\dots =\alpha _{m_1}<\alpha _{m_1+1}=2\leq \dots \leq \alpha _N=\kappa$.

For any $p\in \G$ the intrinsic left translation $\tau _p:\G \to \G $ are  defined as
\begin{equation*}
q \mapsto \tau _p q := p\cdot q=pq.
\end{equation*} 

The explicit expression of the group operation $\cdot $ is determined by the Campbell-Hausdorff formula. It has the form 
\begin{equation*}
p\cdot q= p+q+\Q(p,q) \quad \mbox{for all }\, p,q \in \G\equiv \R^N,
\end{equation*} 
where $\Q=(\Q_1,\dots , \Q_N):\R^N\times \R^N \to \R^N$. Here $\Q_l (p,q)=0$ for each $l=1,\dots , m_1$ and each $1<j\leq k$ and $m_{j-1}+1 \leq l \leq m_j$ we have $\Q_l(p,q)=\Q_l((p_1,\dots , p_{m_{j-1}}),(q_1,\dots , q_{m_{j-1}}))$. Moreover every $\Q_l$ is a homogeneous polynomial of degree $\alpha _l$ with respect to the intrinsic dilations of $\G$. 
It is useful to think $\G=\G^1\,  \oplus \, \G^2 \, \oplus \dots \oplus \, \G^\kappa $ where $\G^l=\exp (V_l)=\R^{n_i}$ is the $l^{th}$ layer of $\G$ and to write $p\in \G$ as $(p^1,\dots , p^\kappa)$ with $p^l\in \G^l$. According to this
\begin{equation}\label{opgr}
p\cdot q= (p^1+q^1,p^2+q^2+\Q^2(p^1,q^1),\dots ,p^\kappa +q^\kappa+\Q^\kappa ((p^1,\dots , p^{\kappa-1} ),(q^1,\dots ,q^{\kappa-1}) ) 
\end{equation} 
for every $p=(p^1,\dots , p^\kappa )$, $q=(q^1,\dots ,q^\kappa ) \in \G$. 
In particular $p^{-1}=(-p^1,\dots , -p^\kappa )$.

The norm of $\R^{n_s}$  is denoted with  the symbol $|\cdot |_{\R^{n_s}}$. 
\\

An absolutely continuous curve $\gamma : [0,T]\to \G$ is a \textit{sub-unit curve} with respect to $X_1,\dots ,X_{m_1}$ if it is an \textit{horizontal curve}, that is if there are real measurable functions $h_1(t),\dots ,h_{m_1}(t)$, $t\in [0,T]$ such that 
\[
\dot \gamma (s)= \sum_{l=1}^{m_1}h_l(t)X_l(\gamma (t)), \hspace{0,5 cm} \mbox{ for a.e. } t\in [0,T],
\]
and if $ \sum_{l=1}^{m_1}h_l^2\leq 1$.
\begin{defi} (Carnot--Carath\'eodory distance).
If $p,q\in \G$, their cc-distance $d_{cc}(p,q)$ is 
\[
d_{cc}(p,q)=\inf \{T>0 : \mbox{there is a subunit curve }  \gamma   \mbox{ with  }  \gamma (0)=p, \gamma (T)=q \}
\]
\end{defi}

The set of subunit curves joining $p$ and $q$ is not empty, by Chow's theorem, since  the rank of the Lie algebra generated by $X_1,\dots ,X_{m_1}$ is $N$; hence $d_{cc}$ is a distance on $\G$ inducing the same topology as the standard Euclidean distance. 

A \emph{homogeneous norm} on $\G$ is a nonnegative function $p\mapsto \|p\|$ such that for all $p,q\in \G$ and for all $\lambda \geq 0$
\begin{equation*}\label{defihomogeneous norm}
\begin{split}
\|p\|=0\quad &\text{if and only if }  p=0\\
\|\delta _\lambda p\|= \lambda \|p\|,& \qquad 
\|p \cdot q\|\leq \|p\|+ \|q\|.
\end{split}
\end{equation*}
Given any homogeneous norm $\|\cdot \|$, it is possible to introduce a distance in $\G$ given by
\[
d(p,q)=d(p^{-1} q,0)= \|p^{-1} q\| \qquad \text{for all $p,q\in G$}.   
\]
We observe that any  distance $d $ obtained in this way is always equivalent with the $cc$-distance $d_{cc}$ of the group.


Both $d$ and $d_{cc}$ are well behaved with respect to left translations and dilations, i.e. for all $p,q,q' \in \G$ and $\lambda >0$,
\begin{equation*}
\begin{aligned}
d (p\cdot q,p\cdot q')=d(q,q'), \qquad d (\delta_\lambda q,\delta_\lambda q')=\lambda d(q,q')
\end{aligned}
\end{equation*}
Moreover, for any bounded subset $\Omega \subset \G$ there exist positive constants $c_1=c_1(\Omega),c_2=c_2(\Omega)$ such that for all $p,q \in \Omega$
\[
c_1|p-q|_{\R^N} \leq d(p,q) \leq c_2 |p-q|_{\R^N}^{1/\kappa }
\] 
and, in particular, the topology induced on $\G$ by $d$ is the Euclidean topology. For $p\in \G$ and $r > 0$,  $\mathcal U (p,r)$ will be the open ball associated with the distance $d$. Intrinsic $t$-dimensional spherical Hausdorff measure $\mathcal{S}^{t}$ on $\G$, $t \geq 0$, is obtained from $d$, following Carath\'eodory construction (see for instance \cite{biblio16}).

The Hausdorff dimension of $(\G, d )$ as a metric space is  denoted \textit{homogeneous dimension} of $\G$ and it can be proved to be   the integer $\mathfrak q:=\sum_{j=1}^{N}\alpha _j=\sum_{l=1}^{\kappa }i$ dim$V_l \geq N$ (see \cite{biblioMitchell}). 

The subbundle of the tangent bundle $T\G$, spanned by the vector fields $X_1,\dots, X_{m_1}$ plays a particularly important role in the theory, and is called the \textit{horizontal bundle} $H\G$; the fibers of $H\G$ are
\[
H\G_P=\mbox{span} \{ X_1(p),\dots ,X_{m_1}(p)\}, \hspace{0,5 cm} p\in \G.  
\]
A sub Riemannian structure is defined on $\G$, endowing each fiber of $H\G$ with a scalar product $\langle\cdot ,\cdot \rangle_p$ and a norm $|\cdot |_p$ making the basis $X_1(p),\dots ,X_{m_1}(p)$ an orthonormal basis. Hence, if $v=\sum_{l=1}^{m_1}v_l X_l(p)=v_1$ and $w=\sum_{l=1}^{m_1}w_l X_l(p)=w_1$ are in  $H\G$, then $\langle v,w \rangle_p:=\sum_{l=1}^{m_1} v_l w_l$ and $|v|_p^2:=\langle v,v \rangle_p$. We will write, with abuse of notation, $\langle\cdot ,\cdot \rangle$  meaning $\langle\cdot ,\cdot \rangle_p$ and $|\cdot |$ meaning $|\cdot |_p$.

The sections of $H\G$ are called \textit{horizontal sections}, a vector of $H\G_p$ is an \textit{horizontal vector} while any vector in $T\G _p$ that is not horizontal is a vertical vector. 

The Haar measure  of the group $\G=\R^N$ is the Lebesgue measure $d\mathcal{L}^N$. It is left (and right) invariant.
Various Lebesgue spaces on $\G$ are meant always with respect to the measure  $d\mathcal{L}^N$ and are denoted as $\mathcal L^p(\G)$.

\subsection{$\C^1_\G$ functions, $\G$-regular surfaces, Caccioppoli sets} (See \cite{biblioLIBROSC}). 
In \cite {biblio11} Pansu introduced an appropriate notion of differentiability for functions acting between Carnot groups. We recall this definition in the particular instance that is relevant  here. 

Let $\mathcal U$ be an open subset of a Carnot group $\G$.  A function $f:\mathcal U\to \R^k$ is Pansu differentiable or more simply P-differentiable in $a \in \mathcal U$ if there is a homogeneous homomorphism 
\[
d_\mathbf Pf(a): \G\to \R^k,
\]
the Pansu differential of $f$ in $a$, such that, for $b\in \mathcal U$, 
\[
\lim_{r\to 0^+}\sup_{0<\Vert a^{-1}b\Vert<r}\frac{|f(b)-f(a)- d_\mathbf Pf(a)(a^{-1}b)|_{\R^k}}{\Vert a^{-1}b\Vert}= 0.
\]
Saying that $d_\mathbf Pf(a)$ is a  homogeneous homomorphism we mean that $d_\mathbf Pf(a): \G\to \R^k$ is a group homomorphism and also that $d_\mathbf Pf(a)(\delta_\lambda b)=\lambda d_\mathbf Pf(a)(b)$ for all $b\in \G$  and $\lambda \geq 0$.

Observe that, later on in Definition \ref{d3.2.1}, we give a different notion of differentiability for functions acting between subgroups of a Carnot group and we reserve the notation $df$ or $df(a)$ for that differential. 

We denote  $\C^1_\G (\mathcal U ,\R^k )$  the set of functions $f:\mathcal U\to \R^k$ that are P-differentiable in each $a\in \mathcal U$ and such that $d_\mathbf Pf(a)$ depends continuously on $a$. 

It can be proved that $f=(f_1,\dots , f_k)\in \C^1_\G (\mathcal U ,\R^k )$ if and only if the distributional horizontal derivatives  
$X_lf_j $,  for $l=1\dots, m_1$, $j=1,\dots, k$,
are continuous in $\mathcal U $.  Remember that $\C^1(\mcal U, \R ) \subset \C^1_\G (\mcal U, \R)$ with strict inclusion whenever $\G$ is not abelian (see Remark 6 in \cite{biblio6}).

The \emph{horizontal Jacobian} (or the \emph{horizontal gradient} if $k=1$) of $f:\mcal U \to \R^k$ in $a\in \mathcal U$ is the matrix
\[
 \nabla_\G f(a):=\left[X_lf_j(a)\right] _{l=1\dots m_1, j=1\dots k}
\]
when the partial derivatives $X_if_j$ exist. Hence $f=(f_1,\dots , f_k)\in \mathbb C^1_\G (\mcal U ,\R^k )$ if and only if its horizontal Jacobian exists and is continuous in $\mathcal U$. 
 The \emph{horizontal divergence} of $\phi:=(\phi _1,\dots , \phi _{m_1}):\mcal U\to \R^{m_1}$  is defined as 
\begin{equation*}
 \mbox{div}_\G \phi := \sum_{j=1}^{m_1} X_j\phi _j
\end{equation*}
if $X_j\phi _j$ exist for  $j=1,\dots ,m_1$.

Now we use the notion of P-differentiability do introduce introduce the $\G$-regular surfaces. 
\begin{defi}\label{Gregularsurfaces}
$S\subset \G$ is a \emph{$k$-codimensional $\G$-regular surface} if for every $p\in S$ there are a neighbourhood $\mcal U$ of $p$ and a function $f=(f_1,\dots , f_k)\in \mathbb{C}^1_\G(\mcal U,\R^k)$ such that
\[
S\cap \mcal U=\{ q\in \mcal U : f(q)=0 \}
\]
and $d_\mathbf Pf(q)$ is surjective, or equivalently if the $(k\times m_1)$ matrix $\nabla_\G f(q)$ has rank $k$, for all $q\in \mcal U$.
\end{defi}

The class of $\G$-regular surfaces is different from the class of Euclidean regular surfaces. In \cite{biblio19}, the authors give an example of $\HH^1$-regular surfaces, in $\mathbb{H}^1$ identified with $\R^3$, that are (Euclidean) fractal sets. Conversely, there are continuously differentiable 2-submanifolds in $\R^3$ that are not $\HH^1$-regular surfaces (see \cite{biblio6} Remark 6.2 and  \cite{biblio1} Corollary 5.11).
\\

In the setting of Carnot groups, there is a natural definition of bounded variation functions and of finite
perimeter sets (see \cite{biblioGARN} or \cite{biblioLIBROSC} and the bibliography therein).

We say that $f:\mcal U \to \R $ is of bounded $\G$-variation in an open set $\mcal U \subset \G$ and we write $f\in BV_\G(\mcal U )$, if $f\in \mathcal L^1(\mcal U )$ and
\[
\| \nabla _\G f \| (\mcal U ):= \sup \Bigl\{ \int _{\mcal U} f \, \mbox{div}_\G\phi \, d\mathcal{L}^N : \phi \in \C^1_c (\mcal U , H\G), |\phi (p)| \leq 1 \Bigl\} <+\infty .
\] 
The space $BV_{\G , loc }(\mcal U )$ is defined in the usual way.

In the setting of Carnot groups, the structure theorem for $BV_\G$ functions reads as follows.
\begin{theorem}\label{structure theorem BV}
If $f\in BV_{\G , loc }(\Omega )$ then $\| \nabla _\G f \|$ is a Radon measure on $\Omega$. Moreover, there is a $\| \nabla _\G f \|$ measurable horizontal section   $\sigma _f : \Omega \to H\G$ such that $|\sigma _f (P)|=1$ for $\| \nabla _\G f \|$-a.e. $P\in \Omega $ and
\begin{equation*}
\int_{\Omega } f \mbox{div}_\G \xi \, d \mathcal{L}^{N}  =   \int_{ \Omega } \langle \xi ,\sigma _f\rangle\, d\| \nabla _\G f \|,
\end{equation*}
for every $\xi \in \C^1_c(\Omega , H\G)$. Finally the notion of gradient $\nabla _\G$ can be extended from regular functions to functions $f\in BV_\G$ defining $\nabla _\G f$ as the vector valued measure 
\begin{equation*}
\nabla _\G f:=-\sigma _f \res \|\nabla _\G f\| =(-(\sigma _f)_1 \res \|\nabla _\G f\|  , \dots , -(\sigma _f)_{m_1} \res \|\nabla _\G f\| ),
\end{equation*}
where $(\sigma _f)_i$ are the components of $\sigma _f$ with respect to the base $X_i$.
\end{theorem}

A set $\mcal E\subset \G$ has locally finite $\G$-perimeter, or is a $\G$-Caccioppoli set, if $\chi_{\mcal E} \in BV_{\G , loc }(\G)$, where $\chi_{\mcal E}$ is the characteristic function of the set $\mcal E$. In this case the measure $\| \nabla _\G \chi_{\mcal E}\|$ is called the $\G$-perimeter measure of $\mcal E$ and is denoted by $|\partial \mcal E|_\G$. Moreover we call generalized intrinsic normal of $\partial \E$ in $\Omega$ the vector $$\nu _\E (p):= -\sigma _{\chi_\E } (p).$$


Fundamental estimates in geometric measure theory are the so-called relative and global isoperimetric inequalities for Caccioppoli sets. The proof is established in \cite{biblioGARN}, Theorem 1.18.
\begin{theorem}\label{isoperimetric inequalitieC}
 There exists a constant $C>0$ such that for any $\G$-Caccippoli set $E\subset \G$, for every $P\in \G$ and $r>0$  
 \begin{equation*}
\min \{ \mathcal{L}^N\left(E\cap \mcal U (p,r)\right),  \mathcal{L}^N\left(\mcal U (p,r) - E\right)\}^{ (\mathfrak q  -1)/ \mathfrak q } \leq C  |\partial  E|_\G (\mcal U(p,r))
\end{equation*}
 and
  \begin{equation*}
\min \{ \mathcal{L}^N\left(E\right) ,  \mathcal{L}^N\left(\G - E\right)\}^{ (\mathfrak q  -1)/ \mathfrak q } \leq C  |\partial  E|_\G (\G)
\end{equation*}
 where $\mathfrak q$ is homogeneous dimension of $\G$ defined in Section $\ref{Carnotinizio}$.
\end{theorem}

The perimeter measure is concentrated in a subset of topological boundary of $E$, the so-called reduced boundary $\partial ^*_\G E$.
\begin{defi}[Reduced boundary]\label{defiReduced boundary}
 Let $E\subset \G$ be a $\G$-Caccioppoli set. We say that $p\in \partial ^*_\G E$ if
\begin{enumerate}
\item $ |\partial E|_\G (\mcal U (p,r))>0, $ for all $r>0$
\item there exists $\lim_{r\to 0} \ave _{\mcal U (p,r)} \nu_E \, d  |\partial  E|_\G$
\item $\left| \lim_{r\to 0} \ave_{\mcal U (p,r)} \nu_E \, d  |\partial  E|_\G \right| =1$ 
\end{enumerate}
\end{defi}  
The  reduce boundary of a set $E\subset \G$ is invariant under group translations, i.e.
\begin{equation*}
q\in  \partial ^*_\G E \qquad \mbox{if and only if}\qquad \tau_pq\in  \partial ^*_\G (\tau_pE)
\end{equation*}
and also $\nu_E (q)=\nu_{\tau _pE} (\tau_p q)$.

\begin{lem}
[Differentiation Lemma, \cite{biblioAMBROS18}] If $E\subset \G$ is a $\G$-Caccioppoli set, then
\end{lem} 
\begin{equation*}
\lim_{r\to 0} \ave _{ \mcal U (p,r)} \nu_E \, d  |\partial  E|_\G  =\nu _E (p), \quad \mbox{for } |\partial  E|_\G \mbox{-a.e.}\, p
\end{equation*}
hence $|\partial  E|_\G $ is concentrated on the reduced boundary  $\partial ^*_\G E$.


The perimeter measure equals a constant times the spherical ($\mathfrak q-1$)-dimensional Hausdorff measure restricted to the reduced boundary, indeed
\begin{theorem}[\cite{bibliofsscAREA}, Theorem 4.18]\label{Theorem 4.18fssc} 
Let $\G $ be a 	Carnot group of step $2$, endowed with the invariant distance $d$. If $E \subset \G$ is a $\G$-Caccioppoli set, then 
$$|\partial E|_\G = c \, \mathcal{S}^{\mathfrak q-1} \,  \res \, \partial ^*_\G E$$
 where $\mathfrak q$ is homogeneous dimension of $\G$ defined in Section $\ref{Carnotinizio}$.

\end{theorem}

At each point of the reduced boundary of a $\G$-Caccioppoli set there is a (generalized) tangent group: 
\begin{theorem}[Blow-up Theorem, \cite{biblio8}]\label{teo433} 
Let $\G$ be a Carnot group of step $2$ and let $E\subset \G$ be a set with locally finite $\G$-perimeter. If $p\in \partial ^*_\G E$  
then
\[
\lim_{r\to 0 } \chi _{ E_{r,p} } = \chi _{ S^+_\G (\nu _E(p)) }  \quad \mbox{in } \mathcal{L}^1_{loc} (\G)
\]
where $E_{r,p}:= \delta _{1/r} (\tau_{p^{-1}} E) =\{ q\, :\, \tau_p(\delta_r (q)) \in  E\}$ and
\begin{equation*}
S^+_\G (\nu _E(p)) :=\{ q=(q^1,q^2) \in \G \, |\,   \langle \nu _E (p), q^1  \rangle \geq 0 \}.
\end{equation*}
Moreover for all $\delta>0$
\begin{equation*}
\lim_{r\to 0 } |\partial E_{r,p} |_\G (\mcal U(0,\delta ))=|\partial S^+_\G (\nu_E (p))|_\G (\mcal U(0,\delta ))
\end{equation*}
and
\begin{equation*}
|\partial S^+_\G (\nu_E (p))|_\G (\mcal U(0,\delta )) = \mathcal{H}^{N-1} (T_\G^g (\nu_E(0))\cap \mcal U(0,\delta ))
\end{equation*}
where $T_\G^g  (\nu_E(0)):= \{ q=(q^1,q^2) \in \G \, |\,   \langle \nu _E (0), Q^1  \rangle = 0 \}$ is the topological boundary of $S^+_\G (\nu_E (0))$.

\end{theorem}

Finally, as it is usual in the literature, we can also define the measure theoretic boundary $\partial _{*,\G} E$: 
\begin{defi}
 Let $E\subset \G$ be a measurable set. We say that $p$ belongs to measure theoretic boundary $\partial _{*,\G} E$ of $E$ if
\begin{equation*}
\limsup_{r\to 0 ^+} \frac{\mathcal L^N (E\cap \mcal U (p,r) ) }{ \mathcal{L}^{N} ( \mcal U (p,r) )}  >0 \qquad  \mbox{and}\qquad \limsup_{r\to 0 ^+} \frac{\mathcal{L}^{N} (E^c\cap \mcal U(p,r) ) }{ \mathcal{L}^{N} ( \mcal U(p,r) )}  >0.
\end{equation*}
\end{defi}  

If $E\subset \G$ is $\G$-Caccioppoli set, then
\[
\partial ^*_{\G} E \subset \partial _{*,\G} E \subset \partial E.
\]
Moreover,  $\mathcal{S}^{\mathfrak q-1} ( \partial _{*,\G} E - \partial ^*_{\G} E)=0.$

%


\subsection{Complementary subgroups and graphs} 
\begin{defi} We say that $\W$ and $\M$ are \emph{complementary subgroups in $\G$} if 
$\W$ and $\M$ are homogeneous subgroups of $\G$ such that  $\W \cap \M= \{ 0 \}$ and  $$\G=\W\cdot \M.$$  By this we mean that for every $p\in \G$ there are $p_\W\in \W$ and $p_\M \in \M$ such that $p=p_\W  p_\M$.
\end{defi}


The elements $p_\W \in \W$ and $p_\M \in \M$ such that $p=p_\W \cdot p_\M$ are unique because of $\W \cap \M= \{ 0 \}$ and are denoted  components of $p$ along $\W$ and $\M$ or  projections of $p$ on $\W$ and $\M$.
The projection maps $\mathbf{P}_\W :\G \to \W$ and $\mathbf{P}_\M:\G \to \M$ defined
\[
\mathbf{P}_\W (p)=p_\W, \qquad \mathbf{P}_\M (p)=p_\M, \qquad \text{for all $p\in \G$}
\]
are polynomial functions (see Proposition 2.2.14 in \cite{biblio22}) if we identify $\G$ with $\R^N$, hence are $\C^\infty$. Nevertheless in general they are not  Lipschitz maps, when $\W$ and $\mathbb{M}$ are endowed with the restriction of the left invariant distance $d$ of $\G$ (see Example 2.2.15 in \cite{biblio22}). 

\begin{rem}\label{rem2.2.1}
The stratification of $\G$ induces a stratifications on the complementary subgroups $\W$ and $\M$. If $\G=\G^1\oplus \dots \oplus \G^\kappa$ then also $\W= \W^1\oplus \dots \oplus \W^\kappa$, $\M=\M^1\oplus \dots \oplus \M^\kappa$ and $\G^i =\W^i \oplus \M^i$. A subgroup is \emph{horizontal} if it is contained in the first layer $\G^1$. If $\M$ is horizontal then the complementary subgroup $\W$  is normal.
\end{rem}

\begin{prop}[see \cite{biblio2}, Proposition 3.2]
If $\W$ and $\M$ are complementary subgroups in $\G$ there is $c_0=c_0(\W , \M)\in (0,1)$ such that for each  $p_\W \in \W$ and $p_\M \in \M$
\begin{equation}\label{c_0}
c_0(\| p_\W \|+\|p_\M \|)\leq \| p_\W  p_\M \| \leq \| p_\W \|+\|p_\M \|
\end{equation}
\end{prop}

\begin{defi}
 We say that $S\subset \G$ is a \emph{left intrinsic graph} or more simply a \emph{intrinsic graph} if there are complementary subgroups $\W$ and $\M$  in $\G$ and  $\phi: \mathcal O \subset \W \to \M$ such that
\[
S=\graph {\phi} :=\{ a \phi (a):\, a\in \mathcal O \}.
\]
\end{defi}
Observe that, by uniqueness of the components along $\W$ and $\M$, if $S=\graph {\phi}$ then $\phi $ is uniquely determined among all functions from $\W$ to $\M$.  

We call graph map of $\phi $, the function $\Phi :\mathcal O \to \G$ defined as
\begin{equation}\label{Phi}
\Phi (a):= a \cdot \phi (a) \quad \mbox{for all } a\in \mathcal O. 
\end{equation}
Hence $S=\Phi (\mathcal O )$ is equivalent to $S=\graph{\phi}$. 

The concept of intrinsic graph is preserved by translation and dilation, i.e.
\begin{prop}[see Proposition 2.2.18 in \cite{biblio22}]\label{P2.2.18} 
If $S$ is a intrinsic graph then, for all $\lambda >0$ and for all $q\in \G$, $q \cdot S$ and $\delta _\lambda S$ are intrinsic graphs. In particular, if $S=\graph {\phi}$ with $\phi :\mathcal O \subset \W \to \M$, then
\begin{enumerate}
\item
For all $\lambda >0$, \[\delta _\lambda \left(\graph {\phi}\right) =\graph {\phi _\lambda}\]  where
$\phi _\lambda :\delta _\lambda \mathcal O \subset \W \to \M $ and 
$ \phi _\lambda (a):= \delta _\lambda \phi (\delta _{1/\lambda }a)$,  for $a \in \delta _\lambda \mathcal O$.  
\item
For any $q\in \G$, \[q \cdot \graph {\phi} = \graph {\phi _q }\] where
$\phi _q : \mathcal O _q \subset \W \to \M$ is defined as 
$\phi _q (a):= (\mathbf P_\M (q^{-1}a))^{-1} \phi( \mathbf P_\W (q^{-1}a))$, for all $a \in \mathcal O_q:=\{ a\, :\, \mathbf P_\W (q^{-1}a)\in \mathcal O  \}$.  
\end{enumerate}
\end{prop}

\subsection{Intrinsic differentiability}


\begin{defi}
Let $\W$ and $\M$ be complementary subgroups in $\G$. Then $\ell:\W\to \M$ is  \emph{intrinsic linear} if $\ell$ is defined on all of $\W$ and if $\graph {\ell} $ is a homogeneous subgroup of $\G$.
\end{defi}

We use intrinsic linear functions to define intrinsic differentiability as in the usual definition of differentiability.
\begin{defi}\label{d3.2.1}
Let $\W$ and $\M$  be complementary subgroups in $\G$ and let $\phi :\mathcal O \subset \W \to \M$ with $\mathcal O$ open in $\W$. For $a\in \mathcal O$, let $p:=a\cdot \phi (a)$ and $\phi _{p^{-1}}: \mathcal O _{p^{-1}} \subset \W \to \M$ be the shifted function defined in Proposition $\ref{P2.2.18}$.
\begin{enumerate}
\item We say that $\phi$ is \emph{intrinsic differentiable in $a$} if the shifted function $\phi_{p^{-1}}$ is intrinsic dif\-fe\-ren\-tia\-ble in $0$, i.e. if there is a intrinsic linear $d\phi_a:\W\to \M$ such that
 \begin{equation*}\label{3.0}
\lim_{r\to 0^+}\sup_{0<\|b\|<r}\frac{\| d\phi_{a} (b)^{-1} \phi _{p^{-1}} (b) \|}{\|b\|} =0.
\end{equation*}
The function $d\phi_a$ is the \emph{intrinsic differential of $\phi $ at $a$}.

\item We say that $\phi$ is \emph{uniformly intrinsic differentiable in $a_0\in \mathcal O$} or $\phi$ is \emph{u.i.d. in $a_0$} if  there exist a intrinsic linear function $d\phi_{a_0}: \W \to \M$ such that
\begin{equation}\label{3.0.1}
\lim_{r\to 0^+}\sup_{\|a_0^{-1}a\|<r}\sup_{0<\|b\|<r}\frac{\| d\phi_{a_0} (b)^{-1} \phi _{p^{-1}} (b) \|}{\|b\|} =0.
\end{equation}
Analogously, $\phi$ is u.i.d. in $\mathcal O$ if it is u.i.d. in every point of $\mathcal O$. 
\end{enumerate}
\end{defi}


\begin{rem} Definition \ref{d3.2.1} is a natural one because of the following observations.

\emph{(i)} If $\phi$ is intrinsic differentiable in $a\in \mathcal O$, there is a unique  intrinsic linear function $d\phi_a$ satisfying $\eqref{3.0}$.  Moreover $\phi$ is continuous at $a$. (See Theorem 3.2.8 and Proposition 3.2.3 in \cite{biblio21}).

\emph{(ii)} The notion of intrinsic differentiability is invariant under group translations. Precisely, let $p:=a\phi (a), q:=b\phi (b)$, then $\phi $ is intrinsic differentiable in $a$ if and only if $\phi _{qp^{-1}} := (\phi _{p^{-1}})_{q}$ is intrinsic differentiable in $b$.

\emph{(iii)}  It is clear, taking $a=a_0$ in \eqref{3.0.1}, that if $\phi$ is uniformly intrinsic differentiable in $a_0$ then it is intrinsic differentiable in $a_0$ and $d\phi_{a_0}$ is the \emph{intrinsic differential of $\phi $ at $a_0$}.
\end{rem}

 The analytic definition of intrinsic differentiability of Definition $\ref{d3.2.1}$ has an equivalent geometric formulation. Indeed intrinsic differentiability in one point is equivalent to the existence of a tangent subgroup to the graph, i.e.
 \begin{theorem}[Theorem 3.2.8. in \cite{biblio21}]\label{teo3.2.8}
Let $\W, \M$ be complementary subgroups in $\G$ and let $\phi :\mathcal O \to \M $ with $\mathcal O $ relatively open in $\W$. If $\phi$ is intrinsic differentiable in $a\in \mathcal O $, set $\mathbb{T} :=\graph{d\phi _a}$. Then
\begin{enumerate}
\item $\mathbb{T}$ is an homogeneous subgroup of $\G$;
\item $\mathbb{T}$ and $\M$ are complementary subgroups in $\G$;
\item$ p\cdot \mathbb{T}$ is the tangent coset to $\graph{\phi }$ in $p:=a\phi (a)$.
\end{enumerate}
Conversely, if $p:=a\phi (a)\in \graph{\phi}$ and if there is $\mathbb{T}$ such that $(1)$, $(2)$, $(3)$ hold, then $\phi$ is intrinsic differentiable in $a$ and the differential $d\phi _a : \W \to \M$ is the unique intrinsic linear function such that $\mathbb{T} :=\graph{d\phi _a}$.
\end{theorem}
\medskip


From now on we restrict our setting  studying the notions of intrinsic differentiability and of uniform intrinsic differentiability for  functions $\phi: \W\to \HH$ when $\HH$ is a horizontal subgroup.  When $\HH$ is horizontal, $\W$ is always a normal subgroup since, as observed in Remark $\ref{rem2.2.1}$, it contains the whole strata $\G^2, \dots , \G^\kappa$.
In this case,   the more explicit form of the shifted function $\phi_{P^{-1}} $ allows  a more explicit form of equations \eqref{3.0} and \eqref{3.0.1}.

\begin{prop}[Theorem 3.5. in \cite{biblioDDD}]\label{prop1.1.1}
 Let $\HH$ and $\W$ be complementary subgroups of $\G$, $\mathcal O$ open in $\W$ and $\HH$ horizontal. Then $\phi :\mathcal O \subset \W\to \HH$ is intrinsic differentiable in $a_0\in \mathcal O$ if and only if there is a intrinsic linear  $d\phi_{a_0}:\W\to \HH$ such that
\[
\lim_{r\to 0^+}\sup_{0<\|a_0^{-1}b\|<r} \frac{\|     \phi (b) -\phi (a_0)- d\phi_{a_0}( a_0^{-1}b )\|} {\|  \phi (a_0)^{-1} a_0^{-1}b \phi (a_0)\|} =0.
\]
Analogously,  $\phi$ is uniformly intrinsic differentiable in $a_0\in \mathcal O$, or $\phi$ is u.i.d. in $a_0\in \mathcal O$,  if  there is a intrinsic linear  $d\phi_{a_0}:\W\to \HH$ such that
\[
\lim_{r \to 0^+}\sup_{a,b }    \frac { \|  \phi ( b) - \phi (a)  - d\phi_{a_0}(a^{-1} b) \|}{\|\phi(a)^{-1}a^{-1}b\phi(a)  \|}  =0
\]
where $r$ is small enough so that $\mcal U(a_0,2r)\subset \mcal O$ and the supremum is for $\Vert{a_0^{-1} a}\Vert<r,\, 0<\Vert{a^{-1} b}\Vert<r.$

Finally, if $k<m_1$ is the dimension of $\HH$, and if, without loss of generality, we assume that
\[
\HH=\{p: p_{k+1}=\dots =p_N=0\}\qquad \W=\{p: p_{1}=\dots =p_k=0\}\
\]  then there is a $k\times (m_1-k)$ matrix, here denoted as   $\nabla^\phi\phi(a_0)$, such that 
\begin{equation*}\label{DISSUdifferential}
d\phi_{a_0} (b)= \left(\nabla^\phi\phi(a_0) (b_{k+1},\dots,b_{m_1})^T,0,\dots ,0\right),
\end{equation*}
for all $b=(b_1,\dots,b_N)\in \W$. The matrix $\nabla^\phi\phi(a_0)$ is called the \emph{intrinsic horizontal Jacobian} of $\phi$ in $a_0$ or the \emph{intrinsic horizontal gradient} or even the \emph{intrinsic gradient} if $k=1$.
\end{prop}
 
 Observe that u.i.d. functions do exist. In particular, when $\HH$ is a horizontal subgroup,  $\HH$ valued euclidean $\C^1$ functions are u.i.d.

\begin{theorem}[Theorem 4.7. in \cite{biblioDDD}]\label{propC1implicauid}
If $\W$ and $\HH$ are complementary subgroups of a Carnot group $\G$ with $\HH$  horizontal and $k$ dimensional. If
 $\mathcal O$ is open in $\W$ and  $\phi :\mathcal O \subset \W \to \HH $ is such that $\phi  \in \C^1( \mathcal O, \HH)$ then $\phi$ is u.i.d. in $ \mathcal O$. 
\end{theorem}

In \cite{biblioDDD}, the author gets a comparison between $\G$-regular surfaces (see Definition \ref{Gregularsurfaces}) and the uniformly intrinsic differentiable maps:
\begin{theorem}\label{teo4.1}
Let $\W$ and $\HH$ be complementary subgroups of a Carnot group $\G$ with $\HH$  horizontal and $k$ dimensional. 
Let $\mathcal O$ be open in $\W$, $\phi :\mathcal O \subset \W \to \HH $ and $S:= \graph{\phi}$. Then the following are equivalent:
\begin{enumerate}
\item there are $ \mathcal U$ open in $\G$ and $f=(f_1,\dots, f_k)\in \C_\G^1(  \mathcal U; \R^k)$ such that 
\begin{equation*}
\begin{split}
& S=\{p\in  \mathcal  U: f(p)=0\}\\
& d_{\bf P}f(q)_{\vert \HH}:\HH\to \R^k\quad \text{is bijective for all $q\in \mathcal U$}
\end{split}
\end{equation*}
and $q\mapsto \left(d_{\bf P}f(q)_{\vert \HH}\right)^{-1}$ is continuous.
\item $\phi $ is u.i.d. in $\mathcal O$. 
\end{enumerate}
Moreover, if {\rm(1)} or equivalently {\rm (2)}, hold then, for all $a\in \mathcal O$ the intrinsic differential $d\phi_a$ is
\[
d\phi_a= - \left(d_{\bf P}f(a\phi(a))_{\vert \HH} \right)^{-1}\circ d_{\bf P}f(a\phi(a))_{\vert \W}.
\]
Finally, if, without loss of generality, we choose a base $X_1,\dots, X_N$ of $\mfrak g$ such that $X_1,\dots, X_k$ are  horizontal vector fields, $\HH=\exp(\text{\rm span} \{X_1,\dots, X_k\})$ and $\W=\exp(\text{\rm span} \{X_{k+1},\dots, X_N\})$ then
\[
\HH=\{p: p_{k+1}=\dots =p_N=0\}\qquad \W=\{p: p_{1}=\dots =p_k=0\},
\] 
$
 \nabla _\G f= \left(
\, \mathcal{M}_1 \, \, | \,\, \mathcal{M}_2 \,
\right)
$
where 
 \begin{equation*}
\mathcal{M}_1 := \begin{pmatrix}
X_1f_1 \dots  X_kf_1 \\
\vdots \qquad \ddots \qquad \vdots \\
X_1f_k \dots  X_kf_k
\end{pmatrix},\qquad\mathcal{M}_2 := \begin{pmatrix}
X_{k+1}f_1\dots  X_{m_1}f_1 \\
\vdots \qquad \ddots \qquad \vdots \\
X_{k+1}f_k \dots  X_{m_1}f_k
\end{pmatrix}.
\end{equation*}
Finally, for all $q\in \mathcal U$, for all $a\in \mathcal O$ and for all $p\in \G$ $$\left(d_{\bf P}f(q)\right)(p)= \left(\nabla _\G f(q)\right)p^1$$ and the intrinsic differential is 
\begin{equation*}\label{teo4.1.1}
\begin{split}
d\phi_a(b)&= \left( \left(\nabla^\phi \phi(a)\right)(b_{k+1},\dots,b_{m_1})^T,0 ,\dots , 0 \right)\\
&= \left( \left( - \mathcal M_1(a\phi(a))^{-1}\mathcal M_2(a\phi(a))\right)(b_{k+1},\dots,b_{m_1})^T,0 ,\dots , 0 \right),
\end{split}
\end{equation*}
for all $b=(b_1,\dots,b_N)\in \W$.

\end{theorem}

\subsection{Intrinsic Lipschitz Function}
The following notion of intrinsic Lipschitz function appeared for the first time in \cite{biblio6} and was studied, more diffusely, in \cite{biblio17, biblio27, biblio21, biblio22,  biblio24, biblioPROF}. Intrinsic Lipschitz functions play the same role as Lipschitz functions in Euclidean context. 
\begin{defi}\label{FMSdefi2.3.1part}
If $\W, \HH$ are complementary subgroups in $\G$, $q\in \G$ and $\beta \geq 0$. We can define the cones $C_{\W,\HH} (q,\beta) $ with base $\W$ and axis $\HH$, vertex $q$, opening $\beta $ are given by
\begin{equation*}
C_{\W,\HH} (q,\beta)=q\cdot C_{\W,\HH} (0,\beta)
\end{equation*}
where $C_{\W,\HH} (0,\beta)=\{ p\, :\, \|p_\W\| \leq \beta \|p_\HH\| \}$.
\end{defi}

For all $\lambda >0$ we have that $\delta _\lambda (C_{\W,\HH} (0,\beta _1))=C_{\W,\HH} (0,\beta _1)$ and if $0<\beta _1<\beta _2$, then 
\begin{equation*}
C_{\W,\HH} (q,\beta _1) \subset C_{\W,\HH} (q,\beta _2). 
\end{equation*}

Now we introduce the basic definitions of this paragraph.
\begin{defi}\label{FMSdefi2.3.3}
Let $\W, \HH$ are complementary subgroups in $\G$. We say that $\phi : \mathcal O \subset \W \to \HH$ is intrinsic $C_L$-Lipschitz in $\mathcal O$ for some $C_L\geq 0$ if for all $C_1>C_L$
 \begin{equation*}\label{coniFMS2.3.3}
C_{\W,\HH} (p, 1/C_1)\cap \graph{\phi} =\{p\} \quad \mbox{ for all } p\in  \graph{\phi}.
 \end{equation*}
 The Lipschitz constant of $\phi$ in $\mathcal O$ is the infimum of the $C_1>0$ such that \eqref{coniFMS2.3.3} holds.
 
We will call a set $S\subset \G$ an intrinsic Lipschitz graph if there exists an intrinsic Lipschitz function $\phi : \mathcal O \subset \W \to \HH$ such that $S= \graph{\phi}$ for suitable complementary subgroups $\W$ and $\HH$.
\end{defi}

We observe that the geometric definition of intrinsic Lipschitz graphs has equivalent analytic forms (see Proposition 3.1.3. in  \cite{biblio22}):

\begin{prop}\label{prop4.56SerraC} 
Let $\W, \HH$ be complementary subgroups in $\G$, $\phi : \mathcal O \subset \W \to \HH$ and $C_L > 0$. Then
 the following statements are equivalent:
\begin{enumerate}
\item $\phi$ is intrinsic $C_L$-Lipschitz in $\mathcal O$.
\item $\| \mathbf{P}_\HH ( q^{-1} q') \| \leq C_L \| \mathbf{P}_\W (q^{-1} q') \| \, \, $ for all $q, q' \in  \graph{\phi}.$
\item $\| \phi_{Q^{-1}} (a) \| \leq C_L \| a \| \, \, $ for all $ q \in \graph{\phi}$ and $a\in \mathcal O_{ q^{-1}}$.
\end{enumerate}
\end{prop}

If $\phi : \mathcal O \subset \W \to \HH$ is intrinsic $C_L$-Lipschitz in $\mathcal O$ then it is continuous. Indeed if $\phi (0)=0$ then by the condition 3. of Proposition \ref{prop4.56SerraC} $\phi$ is continuous in $0$. To prove the continuity in $a\in \mathcal O$, observe that $\phi _{q^{-1}}$ is continuous in $0$, where $q=a\phi (a)$.

\begin{rem} \label{lip0}
In this paper we are interested mainly in the special case  when $\HH$ is a horizontal subgroup and consequently  $\W$ is a normal subgroup. 
Under these assumptions, for all $p= a\phi (a),q=b \phi (b) \in \graph {\phi} $ we have 
\[\mathbf{P}_\HH (p^{-1} q)= \phi (a)^{-1}\phi(b),\quad  \mathbf{P}_\W (p^{-1} q)=\phi (a)^{-1} a^{-1}b\phi (a).
\]
 Hence, if $\HH$ is a horizontal subgroup, $\phi : \mathcal O \subset \W \to \HH$ is intrinsic Lipschitz if 
\[
\|\phi (a)^{-1}\phi(b)\| \leq  C_L \|\phi (a)^{-1} a^{-1}b\phi (a) \| \qquad \text{for all  $a,b \in \mathcal O$.}
\]
Moreover,  if $\phi$ is intrinsic Lipschitz then  $\|\phi (a)^{-1} a^{-1}b\phi (a) \|$ is comparable with $\Vert{p^{-1} q}\Vert$. Indeed from \eqref{c_0}
\begin{equation*}
\begin{split}
c_0\|\phi (a)^{-1} a^{-1}b\phi (a) \| &\leq \Vert{p^{-1} q}\Vert\\ &\leq \|\phi (a)^{-1} a^{-1}b\phi (a) \|+\|\phi (a)^{-1}\phi(b)\|\\
&\leq (1+C_L) \|\phi (a)^{-1} a^{-1}b\phi (a) \|.
\end{split}
\end{equation*}
The quantity $\|\phi (a)^{-1} a^{-1}b\phi (a) \|$, or better a symmetrized version of it, can play the role of a $\phi$ dependent, quasi distance on $\mathcal O$.  See e.g. \cite{biblio1}. 
\end{rem}

\begin{rem}
A map $\phi$ is intrinsic $C_L$-Lipschitz if and only if the distance of two points $q, q'\in $ graph$(\phi )$ is bounded by the norm of the projection of $q^{-1} q'$ on the domain $\mathcal O$. Precisely $\phi :\mathcal O \subset \W \to \HH$ is intrinsic $C_L$-Lipschitz in $\mathcal O$ if and only if there exists a constant $C_1>0$ satisfying
\begin{equation*}\label{rellip} 
\| q^{-1} q' \|\leq C_1\| \mathbf{P}_\W (q^{-1} q') \|, 
\end{equation*}
for all $q,q' \in \graph{\phi}$. Moreover the relations between $C_1$ and the Lipschitz constant $C_L$ of $\phi $ follow from $\eqref{c_0}$. In fact if $\phi $ is intrinsic $C_L$-Lipschitz in $\mathcal O $ then
\begin{equation*}
\| q^{-1} q' \|\leq \| \mathbf{P}_\W (q^{-1} q') \| +\| \mathbf{P}_\HH (q^{-1} q') \| \leq (1+C_L)\| \mathbf{P}_\W (q^{-1} q') \|
\end{equation*}
for all $q,q' \in \graph{\phi}$. Conversely if $\| q^{-1} q' \|\leq c_0 (1+C_L)\| \mathbf{P}_\W (q^{-1} q') \|$ then
\begin{equation*}
\| \mathbf{P}_\HH (q^{-1} q' )\|\leq C_L\| \mathbf{P}_\W (q^{-1} q') \|
\end{equation*}
for all $q,q' \in \graph{\phi}$, i.e. the condition 2. of Proposition \ref{prop4.56SerraC}  holds.
\end{rem}

We observe that in Euclidean spaces intrinsic Lipschitz maps are the same as Lipschitz maps. The converse is not true (see Example 2.3.9 in \cite{biblio21}) and if $\phi :\W \to \HH$ is intrinsic Lipschitz  then this does not yield the existence of a constant $C$ such that
\[
\|\phi (a)^{-1}\phi (b)\|\leq  C \|a^{-1}b\| \quad \mbox{for } a,b\in \W
\]
not even locally.  In Proposition 3.1.8 in \cite{biblio22} the authors proved that the intrinsic Lipschitz functions, even if non metric Lipschitz, nevertheless are H\"older continuous.
\begin{prop}\label{lip84} 
Let $\W ,\HH$ be complementary subgroups in $\G$ and $\phi :\mathcal O \subset \W \to \HH$ be an intrinsic $C_L$-Lipschitz function. Then, for all $r>0$,
\begin{enumerate}
\item there is $C_1= C_1(\phi, r)>0$ such that 
\[
\|\phi (a)\| \leq C_1 \quad \text{ for all $a\in \mathcal O$ with $\|a\|\leq r$}
\]
\item there is $C_2= C_2(C_L, r)>0$ such that  $\phi$ is locally $1/\kappa $-H\"older continuous i.e. 
\begin{equation*}
\|\phi (a)^{-1}\phi (b)\|\leq C_2 \|a^{-1}b\|^{1/\kappa }\quad \text{for all $a, b$ with $\|a\| ,  \|b\| \leq r$}
\end{equation*}
where $\kappa$ is the step of $\G$.
\end{enumerate}
\end{prop}

Now we present some results which we will use later:

\begin{prop}[\cite{biblio21}, Theorem 4.2.9]\label{Theorem 4.2.9fms} 
If $\hat \phi :\hat{ \mathcal O} \to \V$ is intrinsic Lipschitz then the subgraph $\E$ of $\hat \phi$  is a set with locally finite $\G$-perimeter.
\end{prop}

\begin{prop}[Proposition 3.6 \cite{biblioDDD}]\label{lip1512DDD} 
Let $\HH$, $\W$ be complementary subgroups of $\G$ with $\HH$ hori\-zon\-tal. Let  $\mathcal O$ be open in $\W$ and $\phi :\mathcal O\to \HH$ be u.i.d. in $\mathcal O$. Then
\begin{enumerate}
\item $\phi$ is intrinsic Lipschitz continuous in every relatively compact subset of $\mathcal O$. 
\item the function $a \mapsto d\phi_{a}$ is continuous in $\mathcal O$.
\end{enumerate}
\end{prop}

Finally, we recall the following Rademacher type theorem, proved in \cite{biblio21}, Theorem 4.3.5.
\begin{theorem}\label{Theorem 4.3.5fms}
Let $\G =\W\cdot \HH$ be a Carnot group of step $2$ with $\HH$  horizontal and one dimensional and let $\hat \phi :\mathcal O \to \HH $ be intrinsic Lipschitz function, where $\mathcal O$ is a relatively open subset of $\W$. 
Then $\hat \phi$ is intrinsic differentiable $( \mathcal{L}^{m+n-1}\res \W )$-a.e. in $\mathcal O$. Notice that $ \mathcal{L}^{m+n-1} \res \W $ is the Haar measure of $\W$.
\end{theorem}

\section{Carnot groups of step 2}
In this section we present Carnot groups of step 2 as in Chapter 3 of \cite{biblio3}.  According to \eqref{opgr}, the group operation of a Carnot group $\G := (\R^{m+n}, \cdot )$ of step 2 is
\begin{equation*}
p\cdot q= (p^1+q^1,p^2+q^2+\Q^2(p^1,q^1) ) 
\end{equation*} 
for every $p=(p^1, p^2 )$, $q=(q^1,q^2 ) \in \G$, where $\Q^2=(\Q_1,\dots , \Q_n):\R^m\times \R^m \to \R^n$ and by Proposition 2.2.22 $(4)$ in \cite{biblio3} we have
\begin{equation*}
\Q_s (p^1,q^1)=\sum_{j,l=1}^m c^s_{jl} (p_jq_l-p_lq_j), \quad \mbox{ for } s=1,\dots , n
\end{equation*} 
with $  c_{jl}^{(s)}\in \R$ for each $j,l=1,\dots , m$, $s=1,\dots n$. Hence if we consider $n$ skew-symmetric $m\times m$ real matrices $\mathcal{B}^{(1)}, \dots , \mathcal{B}^{(n)}$, it follows 
\begin{equation*}\label{qui4}
\Q_s (p^1,q^1) =  \langle \mathcal{B}^{(s)}p^1, q^1 \rangle , \quad \mbox{ for } s=1,\dots , n
\end{equation*} 
where $\langle \cdot , \cdot \rangle$ is the inner product in $\R^m$.

From Proposition 3.2.1 in \cite{biblio3}, the authors explicitly remark that the linear independence of these matrices is also necessary for $\G$ to be a Carnot group. Consequently by also Theorem 3.2.2 in \cite{biblio3}, we say that $\G := (\R^{m+n}, \cdot  , \delta_\lambda )$ is a Carnot group of step 2 if there are $n$ linearly independent, skew-symmetric $m\times m$ real matrices $\mathcal{B}^{(1)}, \dots , \mathcal{B}^{(n)}$ such that  
 for all $p=(p^1,p^2)$ and $q= (q^1,q^2)\in \R^{m} \times \R^{n} $ and for all $\lambda >0$
\begin{equation}\label{1} 
p\cdot q= (p^1+q^1 , p^2+q^2+ \frac{1}{2}  \langle \mathcal{B}p^1,q^1 \rangle )
\end{equation}
where $\langle \mathcal{B}p^1,q^1 \rangle := (\langle \mathcal{B}^{(1)}p^1,q^1 \rangle, \dots , \langle \mathcal{B}^{(n)}p^1,q^1 \rangle)$ and $\langle \cdot , \cdot \rangle$ is the inner product in $\R^m$ and 
\begin{equation*}
\delta_\lambda p  := (\lambda p^1 , \lambda^2 p^2).
\end{equation*}

We make the following choice of the homogeneous norm in $\G$:
\begin{equation*}
\Vert(p^1,p^2)\Vert:= \max \left\{\vert p^1\vert _{\R^m}, \epsilon \vert p^2\vert _{\R^n}^{1/2} \right\}
\end{equation*}
for a suitable $\epsilon \in (0,1]$ (for the existence of such an $\epsilon >0$ see Theorem 5.1 in \cite{biblio8}). We recall also that there is $c_1>1$ such that for all $p=(p^1,p^2)\in \G$
\begin{equation}\label{deps}
c_1^{-1}\left(\vert p^1\vert_{\R^m}+\vert p^2\vert_{\R^n}^{1/2} \right)\leq\Vert p\Vert\leq c_1\left(\vert p^1\vert_{\R^m}+\vert p^2\vert_{\R^n}^{1/2} \right)
\end{equation}

From now on we will depart slightly from the notations of the previous sections. Precisely, instead of writing $p=(p_1,\dots, p_{m+n})$ we will write 
\[
p=(x_1,\dots,x_m,y_1,\dots, y_n).
\]
With this notation, when $\mathcal{B}^{(s)}:=(b_{jl}^{(s)})_{ j,l=1}^m$,  a basis of  the Lie algebra $\mathfrak g$ of $\G$, is given by the $m+n$ left invariant vector fields
\begin{equation}\label{5.1.0}
X_j (p) = \partial _{x_j }  +\frac{1}{2 } \sum_ {s=1 }^{n} \sum_ {l=1 }^{m} b_ {jl}^{(s)} x_l  \partial _{y_s },  \qquad\qquad
Y_s(p)  = \partial _{y_s } ,
\end{equation}
where $j=1,\dots ,m, $ and $s=1,\dots , n.$

\begin{rem}\label{rem3.1free} 
The space of skew-symmetric $m\times m$ matrices has dimension $\frac{m(m-1)}{2}$. Hence in any Carnot group $\G$ of step 2 the dimensions $m$ of the horizontal layer and $n$ of the vertical layer are related by the inequality 
\[
n \leq \frac{m(m-1)}{2}.
\]
\end{rem}

\begin{rem}
Heisenberg groups $\mathbb H^k= \R^{2k}\times \R$ are Carnot groups of step 2 and 
the group law  is of the form $\eqref{1}$ with  
 \[
\mathcal{B}^{(1)}=  \begin{pmatrix}
0 &  \mathcal{I}_k\\
-\mathcal{I}_k & 0
\end{pmatrix}
\]
where $\mathcal I_k$ is the $k\times k$ identity matrix.

More generally, H-type groups are examples of Carnot groups of step 2 (see Definition 3.6.1 and Remark 3.6.7 in \cite{biblio3}).  The composition law is of the form $\eqref{1}$ where the matrices $\mathcal{B}^{(1)},\dots ,  \mathcal{B}^{(n)}$ have the following additional properties:
\begin{enumerate}
\item $ \mathcal{B}^{(s)}$ is an $m\times m $ orthogonal matrix for all $s=1,\dots n$
\item  $ \mathcal{B}^{(s)}\mathcal{B}^{(l)}=- \mathcal{B}^{(l)} \mathcal{B}^{(s)}$ for every $s,l=1,\dots ,n$ with $s\ne l$.
\end{enumerate}

Any H-type group is a H-group in the sense of M\'etivier, or a HM-group in short (see Section 3.7 in \cite{biblio3}), which is also example of Carnot groups of step 2. Here the composition law is of the form $\eqref{1}$ with the following additional condition: every non-vanishing linear combination of the matrices $ \mathcal{B}^{(s)}$'s is non-singular.

 Another example of Carnot groups of step 2 is provided by the class $\mathbb F _{m,2}$ of free groups of step-2  (see Section 3.3 in \cite{biblio3}). Here $\mathbb F _{m,2}= \R^{m}\times \R^{\frac{m(m-1)}{2}}$ and  the composition law $\eqref{1}$ is defined by the matrices  $\mathcal B^{(s)}\equiv \mathcal{B}^{(l,j)}$ where $1\leq j< l\leq m$ and $\mathcal{B}^{(l,j)}$ has entries $-1$ in  position $(l,j)$, $1$ in  position $(j,l)$ and $0$ everywhere else. 

Notice that Heisenberg groups are H-type groups while $\mathbb H^1$ is also a free step-$2$ group.
\end{rem}

\subsection{The intrinsic gradient}

 
Let $\G = (\R ^{m+n}, \cdot  , \delta_\lambda )$ be a Carnot group of step 2 as above and $\W$, $\V$ be complementary subgroups in $\G$ with $\V$ horizontal and one dimensional.
 
%

\begin{rem}\label{remIMPORT} To keep notations simpler,
through all this section we assume, without loss of ge\-ne\-ra\-li\-ty, that  the complementary subgroups $\W$, $\V$ are
\begin{equation}\label{5.2.0}
\V:=\{ (x_1,0\dots , 0) \}, \qquad \W:=\{ (0,x_2,\dots , x_{m+n}) \}.
\end{equation}
This amounts simply to a linear change of variables  in the first layer of the algebra $\mathfrak g$.   If we denote $\mathcal{M}$ a non singular $m\times m$ matrix, the linear change of coordinates associated to $\mathcal M$ is 
\begin{equation*}
p=(p^1,p^2)\mapsto (\mathcal M p^1, p^2).
\end{equation*}
The new composition law $\star$ in $\R^{m+n}$,   obtained by writing $\cdot $ in the new coordinates, is
\begin{equation*}
(\mathcal M p^1, p^2) \star (\mathcal M q^1, q^2) := (\mathcal M p^1 +\mathcal M q^1, p^2 + q^2 + \frac{1}{2} \langle  \mathcal{ \tilde B} \mathcal M p^1 ,\mathcal M q^1 \rangle),
\end{equation*}
where $\mathcal{\tilde B} := (\mathcal{\tilde B}^{(1)}, \dots , \mathcal{\tilde B}^{(n)})$ 
and 
$ \mathcal{ \tilde B} ^{(s)}= (\mathcal{M}^{-1})^T \mathcal{B}^{(s)}  \mathcal{M}^{-1}$ for $s=1,\dots ,n.$
It is easy to check that the matrices $\mathcal{ \tilde B}^{(1)} ,\dots ,\mathcal{ \tilde B}^{(n)}$ are skew-symmetric and that $(\R^{m+n}, \star, \delta_\lambda)$ is a Carnot groups of step 2 isomorphic to $\G = (\R^{m+n}, \cdot, \delta_\lambda)$. 
\end{rem}

When $\V$ and $\W$ are defined as in \eqref{5.2.0} there is a natural inclusion $i: \R^{m+n-1}\to \W$ such that, for all $(x_2,\dots x_m,y_1,\dots,y_n)\in \R^{m+n-1}$,  
\[
i((x_2,\dots x_m,y_1,\dots,y_n)):=(0, x_2,\dots x_m,y_1,\dots,y_n)\in \W.
\]
If $\mathcal O$ and $\phi$ are respectively an open set in $\R^{m+n-1}$ and a function $\phi:\mathcal O\to \R$ 
we denote  $\hat {\mathcal O}:= i(\mathcal O)\subset \W$ and $\hat \phi :\hat{\mathcal O}\to \V$ the function defined as
\begin{equation}\label{phipsi}
\hat\phi (i(a)) :=(\phi(a), 0,\dots ,0) 
\end{equation}
for all $a\in \mathcal O$.

From \eqref{teo4.1.1} in Theorem $\ref{teo4.1}$, if $\hat \phi :\hat{\mathcal O} \subset \W \to \V$ is such that $\graph {\hat \phi} $ is 
locally a non critical level set of  $f\in \C^1_\G (\G, \R)$ with $X_1f\neq 0$, then $\hat\phi$ is u.i.d. in $\hat{\mathcal O}$ and  the following representation of the  intrinsic gradient $\nabla^{\hat \phi} \hat\phi$ holds
\begin{equation}\label{DPHI2}
\nabla^{\hat \phi} \hat\phi (p)=-\left (\frac{X_2f}{X_1f} ,\dots , \frac{X_{m}f}{X_1f}\right)(p\cdot \hat\phi (p))
\end{equation}
for all $p\in \hat{\mathcal O}$. 

In Proposition \ref{prop2.22} we prove  a different  explicit expression of  $\nabla^{\hat \phi} \hat\phi $, not involving  $f$, but only derivatives of the real valued function  $\phi$.  

\begin{prop}[Proposition 5.4, \cite{biblioDDD}]\label{prop2.22}
Let $\G := (\R^{m+n}, \cdot  , \delta_\lambda )$ be a Carnot group of step $2$ and $\V$, $\W$ the complementary subgroups defined in \eqref{5.2.0}. Let $\mcal U$ be  open in $\G$,  $f\in \mathbb{C}^1_\G(\mcal U, \R )$ with $X_1f>0$ and assume that $S:=\{ p\in \mcal U : f(p)=0\}$ is non empty. Then
\begin{enumerate}
\item [(i)]
there are $\hat {\mcal O}$ open in $\W$ and  $\hat{\phi} :\hat{\mcal O}  \to \V$ such that $S=\graph {\hat \phi} $.  Moreover $\hat \phi$ is u.i.d. in $\hat{\mcal O}$ and the intrinsic gradient $\nabla^{\hat \phi} \hat\phi $ is the vector 
\[
\nabla^{\hat \phi} \hat\phi (i(a))=\left(D_2^{\phi} {\phi}(a),\dots, D_m^{\phi} {\phi}(a) \right)\]
for all $a\in \mathcal O$ where, for $j=2,\dots, m$,
\begin{equation}\label{pde5}
D_j^{\phi} {\phi} = X_j \phi+\phi  \sum_ {s=1 }^{n} b_ {j1 }^{(s)} Y_s\phi
\end{equation}
in distributional sense in $\mathcal O$, where $X_j, Y_s$ are defined as \eqref{5.1.0}.
\item [(ii)]
The subgraph $ \E:=\{ p\in \mcal U : f(p)<0\}$ has locally finite $\G$-perimeter in $\mcal U$ and its $\G$-perimeter measure $|\partial \mcal E|_\G$ has the integral representation
\[
|\partial \mcal E|_\G (\mathcal{F}) =\int _{\Phi ^{-1}(\mathcal{F})} \sqrt{1+  |\nabla^{\hat \phi} \hat\phi|_{\R^{m-1}}^2} \, \, d\mathcal{L}^{m+n-1}
\]
for every Borel set $\mathcal{F} \subset \mathcal U $ where $\Phi:\mathcal O\to \G$ is defined as $\Phi(a):=i(a)\cdot\hat \phi(i(a))$ for all $a\in\mathcal O$. 
\end{enumerate}
\end{prop} 

From Proposition \ref{prop2.22},   if $\graph{\phi}$ is a $\G$-regular hypersurface, the intrinsic gradient of $\phi $ 
takes the explicit form given in \eqref{pde5}. 
This motivates the definitions of the operators \emph{intrinsic horizontal  gradient} and  \emph{intrinsic derivatives}.
\begin{defi}\label{defintder}
Let $\mcal O $ be  open in $\R^{m+n-1}$,  $\phi :\mcal O \to \R$ be continuous in $\mathcal O$. The \emph{intrinsic derivatives}  $D^\phi _j $, for $j=2,\dots ,m$, are the differential operators with continuous coefficients
\begin{equation*}
\begin{split}
D^\phi _j  &:=\partial _{x_j }+\sum_ {s=1 }^{n} \left(\phi b_ {j 1}^{(s)} +\frac{1}{2 } \sum_ {i=2 }^{m}  x_i  b_ {ji }^{(s)} \right) \partial _{y_s } \\
&= {X_j}_{\vert \W}+ \phi\,\sum_ {s=1 }^{n}b_ {j 1}^{(s)} {Y_s}_{\vert \W}
\end{split}
\end{equation*}
where, in the second line with abuse of notation, we denote with the same symbols $X_j$ and $Y_s$ the vector fields acting on functions defined in $\mathcal O$.

If  $\hat\phi:=(\phi, 0,\dots, 0):\hat{\mathcal O}\to \V$, we denote
\emph{intrinsic horizontal gradient} $\nabla^{\hat\phi}$ the differential operator 
\[
\nabla^{\hat\phi} :=(D^\phi _2,\dots, D^\phi _m).
\]
\end{defi}

\begin{defi} 
$\mathbf{(Distributional \, \, solution)}$. Let $\mcal O \subset  \R^{m+n-1}$ be open and $w=\left(w_2,\dots,w_m\right)\in \mathcal{L}^\infty _{loc} (\mathcal O , \R^{m-1})$ . We say that $\phi\in \C(\mcal O,\R)$ is a \emph{distributional solution} in $\mcal O$ of the non-linear first order PDEs' system
\begin{equation}\label{sistema}
\left(D^{\phi}_2\phi,\dots, D^{\phi}_m\phi\right)= w 
\end{equation}
if for every $\zeta \in \C^1 _c (\mathcal O,\R)$
\begin{equation*}\label{distrib}
 \int _\mathcal O \phi \left( \,X_j \zeta +\phi  \sum_{s=1}^n b_ {j1 }^{(s)} Y_s \zeta \right)  \, d\mathcal{L}^{m+n-1} = -\int_\mathcal O w_j \zeta \, d \mathcal L^{m+n-1} 
\end{equation*}
for $j=2,\dots , m$.
\end{defi} 

\begin{rem}
If the vector fields $D^{\phi}_2,\dots, D^{\phi}_m $ are smooth we know that it is possible to connect each couple of points $a$ and $b$ in $\mcal O$ with a piecewise continuous integral curve of horizontal vector fields. This means that there is an absolutely continuous curve $\gamma _h :[t_1,t_2]\to \mathcal O  $ from $a$ to $b$ such that $-\infty <t_1<t_2<+\infty $ and
\begin{equation}\label{gamma}
\dot \gamma _h (t) = \sum_{j=2}^{m} h_j(t) \, D_j^\phi (\gamma _h (t)) \hspace{0,5 cm } \mbox{a.e.}  \, \, t\in (t_1,t_2)
\end{equation}
with $h=(h_2,\dots, h_m) :[t_1,t_2]\to \R^{m-1}$ a piecewise continuous function. In our case the vector fields $D_j^\phi $ are only continuous, and consequently it isn't sure the existence of $\gamma _h$.
\end{rem}

\begin{prop} \label{rem2.14citti}
Let $\G := (\R^{m+n}, \cdot  , \delta_\lambda )$ be a Carnot group of step $2$ and $\V$, $\W$ the complementary subgroups defined in \eqref{5.2.0}. Let $\hat \phi :\hat {\mathcal O} \to \V $  be intrinsic $C_L$-Lipschitz function, where $\mathcal O$ is open in $\W$ and $\phi :\mathcal O  \to \R$ is the map associated to $\hat \phi$ as in $\eqref{phipsi}$.

If $\gamma :[t_1,t_2]\to \mathcal O $ is an integral curve of the vector fields $D^{\phi}_2,\dots, D^{\phi}_m $, then $[t_1,t_2]\ni t  \mapsto \phi (\gamma  (t))$ is Lipschitz continuous.  
\end{prop}

\begin{proof}
For simplicity we define $\hat \gamma (t) := i(\gamma (t)) \in \hat {\mathcal O}$. We would like to show that there exists $C_1>0$ such that
\begin{equation}\label{52cit}
\|\hat \phi(\hat \gamma (t))^{-1}\hat \gamma (t)^{-1}\hat \gamma (t_1)\hat \phi(\hat \gamma (t))  \|  \leq  C_1 (t-t_1)
\end{equation}
for all $t\in [t_1,t_2]$. In fact since $\hat \phi$ is an intrinsic $C_L$-Lipschitz function and recall Remark \ref{lip0} we have
\begin{equation*}
|\phi (\gamma (t))- \phi (\gamma (t_1))|\leq C_L \|\hat \phi(\hat \gamma (t))^{-1}\hat \gamma (t)^{-1}\hat \gamma (t_1)\hat \phi(\hat \gamma (t))  \|  \leq   C_L C_1 (t-t_1).
\end{equation*}

By hypothesis $\gamma = \gamma _h:[t_1,t_2]\to \mathcal O $ is an absolutely continuous curve satisfying $\eqref{gamma}$ with a piecewise continuous function $h=(h_2,\dots , h_m) \in \mathcal{L}^\infty ((t_1,t_2),\R^{m-1})$ and $\gamma _h(t_1)=a=(x,y)$, $\gamma _h(t_2)=b=(x',y')$. 

More precisely we have $\gamma _h(t)=(x_2(t),\dots , x_m(t),y_1(t),\dots ,y_n(t))=(x_\gamma (t),y_1(t),\dots ,y_n(t))$ such that
\begin{equation}\label{x_i}
x_l(t)-x_l=\int_{t_1}^t h_l(r)\, dr  \hspace{0,4 cm } \mbox{ for all }  \,  t\in [t_1,t_2], \, l\in {2,\dots , m}
\end{equation}
\begin{equation}\label{y_s}
y_s(t)-y_s= \sum_{j=2}^m \biggl(b_ {j1 }^{(s)}\int_{t_1}^t h_j(r)\phi (\gamma _h (r))\, dr + \frac{1}{2}  \sum_{l=2}^m b_{jl}^{(s)} \biggl(x_l+ \int_{t_1}^t h_l(r) \, dr\biggl)\int_{t_1}^t h_j(r)\, dr \biggl)
\end{equation}
for all $ t\in (t_1,t_2)$, $s=1,\dots ,n$.

Now we consider
\begin{equation*}
\sigma _\phi (b,a)  := \sum_{s=1}^{n} \Big|y_s-y'_s+\phi(b) \Big(\sum_{l=2}^{m}(x_l -x'_l)b_ {1l}^{(s)}\Big) -\frac{1}{2} \langle \mathcal{B}^{(s)}  x', x\rangle \Big|^{1/2} 
\end{equation*}
If we put $\mathcal{B}_{M} = \max \{ b_{ij}^{(s)} \, | \,  i,j=1,\dots , m \, , s=1,\dots , n \}$, then it easy to see that 
\begin{equation}\label{symmetrya}
\sigma _\phi (b,a) \leq \sigma _\phi (a,b) + n\sqrt{\mathcal{B}_M} |\phi (a)-\phi (b)|^{1/2} |x-x'|^{1/2}_{\R^{m-1}}
\end{equation}
and, recalling \eqref{deps}, if $b=\gamma _h(t)$ then
\begin{equation*}\label{stimasigma}
\begin{aligned}
\|\hat \phi(\hat \gamma (t))^{-1}\hat \gamma (t)^{-1}\hat \gamma (t_1)\hat \phi(\hat \gamma (t))  \|  & \leq  c_1\Big(|x-x_\gamma (t)|_{\R^{m-1}} +\sigma _\phi ( a, \gamma _h(t)) \\
&  + \, n\sqrt{\mathcal{B}_M}|x-x_\gamma (t)|_{\R^{m-1}}^{1/2}|\phi (\gamma _h(t))-\phi (a)|^{1/2} \Big) .
\end{aligned}
\end{equation*}
By $\hat \phi$ is  intrinsic $C_L$-Lipschitz, it follows
\begin{equation*}
\begin{aligned}
n\sqrt{\mathcal{B}_M}|x-x_\gamma (t)|_{\R^{m-1}}^{1/2}  |\phi (\gamma _h(t))-\phi (a)|^{1/2} & \leq   n\sqrt{C_L\mathcal{B}_M}|x-x_\gamma (t)|_{\R^{m-1}}^{1/2}  \|\hat \phi(\hat \gamma (t))^{-1}\hat \gamma (t)^{-1}\hat \gamma (t_1)\hat \phi(\hat \gamma (t))  \| ^{1/2}  \\
& \leq  \frac{ n^2C_L\mathcal{B}_M}2 |x-x_\gamma (t)|_{\R^{m-1}}+ \frac{1}{2} \|\hat \phi(\hat \gamma (t))^{-1}\hat \gamma (t)^{-1}\hat \gamma (t_1)\hat \phi(\hat \gamma (t))  \| \\
\end{aligned}
\end{equation*}
for all $ t\in [t_1,t_2]$. From $\eqref{x_i}$
\begin{equation}\label{hstima8}
|x-x_\gamma (t)|_{\R^{m-1}} \leq (m-1)(t-t_1)\| h\|_{\mathcal{L}^\infty ((t_1,t_2),\R^{m-1})}
\end{equation}
and, consequently, if we put $M_1:=  c_1 (m-1)(2+C_L\mathcal{B}_Mn^2),$ then
\begin{equation}\label{stima1}
\|\hat \phi(\hat \gamma (t))^{-1}\hat \gamma (t)^{-1}\hat \gamma (t_1)\hat \phi(\hat \gamma (t))  \|  \leq  2c_1\sigma _\phi (a,\gamma _h(t)) +M_1(t-t_1)\| h\|_{\mathcal{L}^\infty ((t_1,t_2),\R^{m-1})}.
\end{equation}
Hence it remains  to estimate $\sigma _\phi (a,\gamma _h(t))$.  
By $\eqref{x_i}$ and $\eqref{y_s}$ we observe that
\begin{equation*}
\begin{aligned}
\sigma _\phi (a,\gamma _h(t)) & =   \sum_{s=1}^n \biggl|  \sum_{j=2}^m \left(b_ {j1 }^{(s)}\int_{t_1}^t h_j(r)\phi (\gamma _h (r))\, dr + \frac{1}{2}  \sum_{l=2}^m b_{jl}^{(s)} \biggl(x_l+ \int_{t_1}^t h_l(r) \, dr\biggl) \int_{t_1}^t h_j(r)\, dr \right)+ \\
& \hspace{0,4 cm} +\phi (a)\sum_{l=2}^m b_{1l}^{(s)}\int_{t_1}^t h_l(r)\, dr -\frac{1}{2}\langle \mathcal{B}^{(s)}  x, x_\gamma (t) \rangle \biggl|^{1/2}
\end{aligned}
\end{equation*}
with
\[
\sum_{l,j=2}^m x_l b_{jl}^{(s)}\int_{t_1}^t h_j(r)\, dr - \langle \mathcal{B}^{(s)} x, x_\gamma (t) \rangle=  \sum_{l,j=2}^m x_l b_{jl}^{(s)}(x_j(t)-x_j) - \langle \mathcal{B}^{(s)} x, x_\gamma (t)-  x  \rangle =0
\]
and
\[
 \biggl| \frac{1}{2}  \sum_{l,j=2}^m b_{jl}^{(s)} \, \int_{t_1}^t h_l(r) \, dr \, \int_{t_1}^t h_j(r)\, dr  \biggl| \leq \frac{1}{2} \mathcal{B}_M (m-1)^2 (t-t_1)^2\| h\|^2_{\mathcal{L}^\infty ((t_1,t_2),\R^{m-1})} 
\]
Hence, remembering that $\phi $ is intrinsic $C_L$-Lipschitz function, it follows
\begin{equation*}
\begin{aligned}
\sigma _\phi (a,\gamma _h(t)) & \leq \sum_{s=1}^n \biggl|  \sum_{j=2}^m b_{j1}^{(s)}\int_{t_1}^t h_j(r)\left(\phi(a)- \phi (\gamma _h (r))\right) \, dr  \biggl|^{1/2} + \frac{\sqrt{2\mathcal{B}_M}}{2}  n(m-1)(t-t_1)\| h\|_{\mathcal{L}^\infty ((t_1,t_2),\R^{m-1})} \\
& \leq n \sqrt{\mathcal{B}_MC_L(m-1) (t-t_1)} \| h\|^{1/2}_{\mathcal{L}^\infty ((t_1,t_2),\R^{m-1})} \max_{ r\in [a,b]} \|\hat \phi(\hat \gamma (t))^{-1}\hat \gamma (t)^{-1}\hat \gamma (t_1)\hat \phi(\hat \gamma (t))  \|  ^{1/2}\\ & \quad + \frac{\sqrt{2\mathcal{B}_M}}{2}  n(m-1)(t-t_1)\| h\|_{\mathcal{L}^\infty ((t_1,t_2),\R^{m-1})} 
\end{aligned}
\end{equation*}
Since $\eqref{stima1}$ we conclude that
\begin{equation*}
\begin{aligned}
\max_{ t\in [t_1,t_2]} & \|\hat \phi(\hat \gamma (t))^{-1}\hat \gamma (t)^{-1}\hat \gamma (t_1)\hat \phi(\hat \gamma (t))  \|    \leq M_1 (t-t_1) \| h\|_{\mathcal{L}^\infty ((t_1,t_2),\R^{m-1})} \\ & \quad + 2c_1 n\sqrt{\mathcal{B}_M C_L(m-1)  (t-t_1)} \| h\|^{1/2}_{\mathcal{L}^\infty ((t_1,t_2),\R^{m-1})} \max_{ t\in [t_1,t_2]}\|\hat \phi(\hat \gamma (t))^{-1}\hat \gamma (t)^{-1}\hat \gamma (t_1)\hat \phi(\hat \gamma (t))  \| ^{1/2}  \\
& \quad +\sqrt{2\mathcal{B}_M}  c_1n(m-1) (t-t_1) \| h\|_{\mathcal{L}^\infty ((t_1,t_2),\R^{m-1})}  \\
&  \leq \biggl( M_1+ 8 c_1^2n^2(m-1) \mathcal{B}_M C_L +   \sqrt{2\mathcal{B}_M}c_1n(m-1) \biggl)  (t-a) \| h\|_{\mathcal{L}^\infty ((t_1,t_2),\R^{m-1})} \\
& \quad + \frac{1}{2} \max_{ t\in [t_1,t_2]} \|\hat \phi(\hat \gamma (t))^{-1}\hat \gamma (t)^{-1}\hat \gamma (t_1)\hat \phi(\hat \gamma (t))  \| \\
& =: \frac{M_2}{2} (t-t_1) \| h\|_{\mathcal{L}^\infty ((t_1,t_2),\R^{m-1})} + \frac{1}{2} \max_{ t\in [t_1,t_2]} \|\hat \phi(\hat \gamma (t))^{-1}\hat \gamma (t)^{-1}\hat \gamma (a)\hat \phi(\hat \gamma (t))  \| .
\end{aligned}
\end{equation*}
Consequently for all $t\in [t_1,t_2]$
\begin{equation*}
\|\hat \phi(\hat \gamma (t))^{-1}\hat \gamma (t)^{-1}\hat \gamma (t_1)\hat \phi(\hat \gamma (t))  \|  \leq M_2 (t-t_1) \| h\|_{\mathcal{L}^\infty ((t_1,t_2),\R^{m-1})} =: C_1 (t-t_1), 
\end{equation*}
i.e. $\eqref{52cit}$ holds. 
\end{proof}

\section{Caccioppoli sets}

Let $\G$ be a Carnot group of step 2. Let $\mathbb{S}^{m-1}$ be the unit sphere of $\R^m$ and $\nu \in \mathbb{S}^{m-1}$, i.e. $\nu \in \R^m$ and $|\nu |_{\R^m}=1$. By abuse of notation we identify $\nu =(\nu_1,\dots , \nu_m) \in \R^m$ and $\nu =(\nu , 0 ,\dots, 0)\in \G$. 

Fix $p=(p^1,p^2)\in \G$ with $p^1\in \R^{m}$, $p^2\in \R^n$. Let $\nu (p)=\langle p,\nu \rangle \nu \in \G$ and we define $\nu ^\perp (p)\in \G$ as the unique point such that
\[
p=\nu ^\perp (p) \cdot \nu (p)
\]
More precisely, 
\[
\nu ^\perp (p) = \biggl( p^1-\langle p^1, \nu \rangle \nu  , p^2- \frac{1}{2} \langle p^1, \nu \rangle \langle \mathcal{B}p^1, \nu \rangle \biggl).
\]
We denote by $\nu ^\perp =  \{ p=(p^1,p^2)\in \G \, :\, \langle p^1, \nu \rangle =0 \}$ the orthogonal complement of $\nu $ in $\G$. It is clear that $\nu ^\perp (p)\in \nu ^\perp$ for every $p\in \G$. It is clear that if $\nu =(1,0,\dots ,0)$ then $$ \nu ^\perp (p)= p_\W \qquad \nu (p)= p_\V$$
where $\W$, $\V$ are the complementary subgroups defined as \eqref{5.2.0}. Moreover according to Definition \ref{FMSdefi2.3.3}, the set
\[
\{ q\in \G \, :\, \| \nu ^\perp (p^{-1} q) \| <\beta  \| \nu  (p^{-1} q) \|  \}
\]
is an intrinsic cone with vertex $p$, opening $\beta >0$ and axis specified by $\nu $.

 The main results of this section are Theorem \ref{teo11montivittone} and Theorem \ref{teo11montivittoneTO}. We show that the boundary of set with finite $\G$-perimeter and having a bound on the measure theoretic normal is an intrinsic graph of an intrinsic Lipschitz function. This is explained in Theorem 1.1 in \cite{biblioMV} in the context of Heisenberg groups.
\begin{theorem}\label{teo11montivittone}
Let $E\subset \G$ be a set with finite $\G$-perimeter in $\mathcal U(0,r)$, $\nu _E$ be the measure theoretic inward normal of $E$ and $\nu \in \mathbb{S}^{m-1}$. Assume that there exists $k\in (0,1]$ such that $\langle  \nu _E (p), \nu \rangle \leq -k $ for $|\partial E|_\G$-a.e. $p\in \mathcal U(0,r)$. Then there exists $\beta >0$ such that possibly modifying $E$ in a negligible set, we get for every $p\in \partial E \cap \mathcal U(0,r)$
\begin{equation}\label{mvt14k}
\{ q\in \mathcal U(0,r) \, :\, \| \nu ^\perp (p^{-1} q) \| <-\beta \langle p^{-1} q, \nu \rangle  \} \subset E
\end{equation} 
\begin{equation}\label{mvt15k}
\,\quad \{ q\in \mathcal U(0,r) \, :\, \| \nu ^\perp (p^{-1} q) \| <\beta  \langle p^{-1} q, \nu \rangle  \} \subset \G -E
\end{equation} 
\end{theorem}

\begin{theorem}\label{teo11montivittoneTO}
Under the same assumptions of Theorem $\ref{teo11montivittone}$, possibly modifying $E$ on a negligible set, the set $\partial E \cap \mathcal U(0,r)$ is the intrinsic graph of an intrinsic Lipschitz function.
\end{theorem}

Firstly we prove some preliminary results. 

\begin{lem}\label{mvlemma2.1}
Let $\G $ be a  Carnot group of step $2$ and let $E\subset \G$ be  a locally finite $\G$-perimeter set. Then for all $\mathcal U(p,r)$ with $p\in \G$ and $r>0$,
if we consider a horizontal left invariant vector field $Z$ satisfying
 \begin{equation} \label{mv217a}
\int _E Z\xi \, d \mathcal{L}^{m+n } \leq 0 \quad \forall \xi \in \C^1_c (\mathcal U(p,r),\R),\, \xi \geq 0
\end{equation}  
for each $\mathcal{L}^{m+n}$-measurable set $F\subset \mathcal U(p,r)$, we obtain $$\mathcal{L}^{m+n}(E\cap F)\leq \mathcal{L}^{m+n}(E\cap \exp tZ(F))$$ for all $t\geq 0$ such that $\exp t Z(F)\subset \mathcal U(p,r)$.
\end{lem}

\begin{proof}
Let $\mathcal U(0,r)$ with $r>0$. Notice that because the  homogeneous norm is invariant we can assume $p=0$. Moreover thanks to Remark \ref{remIMPORT}, without loss of generality, we also assume $Z=X_1$.

We consider the map $\Theta :\G \to \G$ given by
 \begin{equation*} 
\Theta (p)=\exp p_1X_1(0,p_2,\dots , p_{m+n }).
\end{equation*} 
It is a global diffeomorphism and it satisfies 
 \begin{equation}\label{mv218b} 
\det d\Theta (p)=1 \quad \mbox{ and } \quad \Theta _*\left( \partial _{ p_1} \right) =X_1.
\end{equation} 
where $d\Theta$ denotes the differential of $\Theta$. If we put $E_1:= \Theta ^{-1} (E)$ and $F_1:= \Theta ^{-1} (F)$, then
 \begin{equation} \label{mv219c}
\Theta (te_1+ F_1)= \exp (tX_1(F)), \quad t\in \R
\end{equation} 
with $e_1=(1,0,\dots, 0) \in \R^{m+n }$. Moreover for all $\theta \in \C^1_c (\Theta ^{-1}(\mathcal U(0,r)) , \R)$ with $\theta \geq 0$ we define  $\xi (q):=\theta (\Theta ^{-1} (q))$ and consequently by \eqref{mv217a} and \eqref{mv218b}
 \begin{equation*} 
\int _{E_1}  \partial _{ p_1} \theta  (p)\, d  \mathcal{L}^{m+n}(p)= \int_E X_1\xi (q)\, d  \mathcal{L}^{m+n}(q) \leq 0.
\end{equation*} 
 Hence by Fubini-Tonelli Theorem and by a standard approximation argument we know that the function $t \mapsto \chi _{E_1} (p+te_1) $ is increasing for $\mathcal{L}^{m+n}$-a.e. $p\in \Theta ^{-1}(\mathcal U(0,r))$ as long as $p+te_1 \in  \Theta ^{-1}(\mathcal U(0,r))$. Then for a certain $t\geq 0$, using again Fubini-Tonelli Theorem we obtain 
 \begin{equation*} 
 \begin{aligned} 
\mathcal{L}^{m+n }(E_1\cap F_1) & =\int _{\R^{m+n-1} } \int_\R \chi_{E_1}(p) \chi_{F_1}(p) \, d  \mathcal{L}d  \mathcal{L}^{m+n -1}\\
&\leq \int _{\R^{m+n-1} } \int_\R \chi_{E_1}(te_1+p) \chi_{F_1}(p) \, d  \mathcal{L}d  \mathcal{L}^{m+n -1}\\
&= \int _{\R^{m+n-1} } \int_\R \chi_{E_1}(p) \chi_{te_1+F_1}(p) \, d \mathcal{L}d \mathcal{L}^{m+n -1}\\
&= \mathcal{L}^{m+n }(E_1\cap (te_1+F_1)).
\end{aligned} 
\end{equation*} 
Finally by the last inequality, $\eqref{mv218b}$ and $\eqref{mv219c}$ we obtain
 \begin{equation*} 
\mathcal{L}^{m+n}(E\cap F)\leq \mathcal{L}^{m+n}(E\cap \exp (tX_1(F))).
\end{equation*} 
Consequently the proof is complete.
\end{proof}

\begin{lem}\label{mvprop2.2}
Let $\G$ be a  Carnot group of step $2$. If $k\in (0,1]$, then there exists $\beta =\beta (k)>0 $ such that for all $\nu \in \mathbb{S}^{m-1}$ and  $p=(p^1,p^2)\in \G$ satisfying 
\begin{equation}\label{mv2.6p22}
\| \nu ^\perp (p)\|=\max \biggl\{ |p^1 -\langle p^1, \nu \rangle \nu |_{\R^m} \, ,\,\epsilon  | p^2- \frac{1}{2} \langle p^1, \nu \rangle \langle \mathcal{B}p^1, \nu \rangle |_{\R^n}^{1/2} \biggl\} \, \leq - \beta \langle p^1, \nu \rangle ,
\end{equation} 
there exist $\eta _1, \dots , \eta _n \in \R^m$ such that for all $s=1,\dots , n$
\begin{equation}\label{mv2.7p22}
\begin{aligned}
& \langle \eta _s, \nu \rangle \leq -\sqrt{1-k^2}\,  |\eta _s|_{\R^m} \qquad \langle p^1- \eta _s, \nu \rangle \leq -\sqrt{1-k^2} \, |  p^1- \eta _s|_{\R^m} \\
& p^2= (\langle \B ^{(1)} \eta _1, p^1 \rangle , \dots , \langle \B ^{(n)} \eta _n, p^1 \rangle)
\end{aligned}
\end{equation} 
\end{lem}

\begin{proof}
We split the proof of this lemma in two steps.

$\mathbf{Step \, 1.}$ First we show the case $m=2$ and $n=1$. Without loss of generality, we can assume that $\nu =(1,0)\in \mathbb{S}^{1}$. For $h>0$ and $z=(z_1,z_2)\in \R^2$ such that $|z_2|\leq -h z_1$ consider the set
\begin{equation*}
R_z(h)= \Big\{ \eta=(\eta_1, \eta_2 ) \in \R^2 \, : \, |\eta_2 |\leq -h\eta _1, \, \, |z_2-\eta_2 |\leq -h (z_1-\eta _1) \Big\}.
\end{equation*} 
For $z=(z_1,z_2)\ne 0$, the set $R_z(h)$ is a parallelogram with vertices $0$, $z$, 
\begin{equation*}
\hat z_1 =\biggl( \frac{z_2 +hz_1}{2h},  \frac{z_2 +hz_1}{2} \biggl) \qquad \mbox{and} \qquad \hat z_2 =\biggl( \frac{hz_1-z_2}{2h}, \frac{z_2 -hz_1 }{2} \biggl).
\end{equation*} 
The smooth function $\vp _z : R_z(h)\to \R$ defined as  $$\vp _z(\eta):= \langle \B^{(1)}\eta , z  \rangle =  b_{12}(\eta_2 z_1 - \eta _1 z_2)$$ is linear and attains the maximum and the minimum on $\partial R_z (h)$. If we assume that $ b_{12}>0$ (for the other case $ b_{12}<0$ it is
sufficient to exchange the roles of $\hat z_1$ and $\hat z_2$), then
\begin{equation*}
\max _{R_z(h)} \vp _z = \vp (\hat z_1)=  b_{12}  \frac{h^2z_1^2-z_2^2}{2h}  \qquad \mbox{and} \qquad  \min _{R_z(h)} \vp _z = \vp (\hat z_2)=  b_{12}  \frac{z_2^2-h^2z_1^2}{2h} 
\end{equation*} 
Moreover if we consider the set 
\begin{equation*}
D_h = \biggl\{ (z,y) \in \R^3 \, :\, |z_2|\leq -h z_1, \, \, \frac{b_{12}}{2} \, \big(z_2^2-h^2z_1^2 \big)  \leq yh\leq \frac{b_{12}}{2}  \, \big(h^2z_1^2-z_2^2\big) \,\biggl\}
\end{equation*} 
by continuity, $\vp _z$ attains all the values between the maximum and the minimum. Hence for each $(z_1,z_2,y)\in D_h$ there exists $\eta \in R_z (h)$ such that $y=  b_{12}(\eta_2 z_1 - \eta _1 z_2)$. 

Now let $\beta >0$ be a number satisfying 
\begin{equation}\label{mv2.8p2} 
\beta \biggl(\frac{\beta}{\epsilon ^2} - \frac{b_{12}}{2} \biggl) \leq  \frac{3b_{12} h }{8} \quad \mbox{with}\quad h=\sqrt{\frac{k^2}{2-k^2}}
\end{equation} 
\begin{equation}\label{mv2.9p2}
\beta ^2 \leq \frac{k^2}{2-2k^2}
\end{equation} 
From now on, $h$ is fixed  and it depends on $k$ as in $\eqref{mv2.8p2} $.

Now we would like to show that each point $P\in \R^3$ satisfying $\eqref{mv2.6p22}$ (with $m=2$, $n=1$ and $\nu =(1,0)$) belongs to  $D_h$.

Indeed let $p=(p_1,p_2,p_3)\in \R^3$ such that $\eqref{mv2.6p22}$ is true, i.e.
\begin{equation*}\label{mv2.6*}
\| \nu ^\perp (p)\|=\max \{ |p_2 | \, ,\, \epsilon | p_3- \frac{1}{2} b_{12} p_1 p_2 |^{1/2} \} \, \leq -\beta p_1 
\end{equation*} 
we have that  $$|p_2|\leq  - \beta p_1 \leq -  \frac{h}{2} p_1$$ and 
\begin{equation*}\label{serve}
-\frac{\beta ^2 p_1^2 }{\epsilon ^2}+\frac{b_{12}}{2}p_1p_2 \leq y\leq \frac{\beta ^2 p_1^2  }{\epsilon ^2} +\frac{b_{12}}{2}p_1p_2.
\end{equation*} 
Consequently using $\eqref{mv2.8p2}$ we deduce $p_3  \leq \frac{3}{8}b_{12} p_1^2  h $ and 
\begin{equation*}
hp_3+\frac{b_{12}}{2} p_2^2 \leq \frac{b_{12}}{2}p_1^2 h^2.
\end{equation*} 
The other estimate $hp_3 - \frac{b_{12}}{2}p_2^2 \geq -\frac{b_{12}}{2} p_1^2 h^2$ is obtained in the similar way. Hence $p\in D_h$ and so  there exists $\eta 
\in R_{p^1} (h)$ such that $p_3= \langle \B^{(1)}\eta , p^1 \rangle $, where $p^1=(p_1,p_2)$. 

Moreover by  $\eqref{mv2.8p2}$, $\eqref{mv2.9p2}$ and $\eqref{mv2.6p22}$ we obtain that
\begin{equation}\label{mv2.12}
\begin{aligned}
 \langle \eta  , \nu \rangle & \leq -\frac{1}{ \sqrt{1+\beta^2}}\,  |\eta |\leq -\sqrt{1-\frac{k^2}{2}}\,  |\eta |\\  \langle p^1- \eta  , \nu \rangle &\leq -\sqrt{1-\frac{k^2}{2}} \, |  p^1- \eta  |,
\end{aligned}
\end{equation} 
 i.e. \eqref{mv2.7p22} follows.
 
 $\mathbf{Step \, 2.}$ We prove the proposition in the general case.
 
 Let $ q^1, (q')^1 \in \R^m -\{0\}$ with $q^1\ne  a(q')^1$ for all $a \in \R - \{ 0 \}$. We denote $p^1:=q^1+(q')^1$ and $ \mathcal{A}= $ span$ \{ q^1, (q')^1 \} $ 
 and we consider the orthogonal projection of $\nu $ onto $\mathcal{A}$ by
 \begin{equation*}
 \pi _{ p^1} \nu = \frac{\langle p^1, \nu \rangle p^1}{|p^1|_{\R^m}^2}
\end{equation*} 
If we put
 \begin{equation*}
\hat \nu =\frac{ \pi _{ p^1} \nu }{| \pi _{ p^1} \nu |} \qquad \mbox{and} \qquad  \xi =| \pi _{ p^1} \nu | 
\end{equation*} 
then  $\xi = \frac{|\langle p^1,\nu \rangle|}{|p^1|_{\R^m}}$ and by $\eqref{mv2.6p22}$ we conclude that
 \begin{equation}\label{mv2.13p2.2}
\frac{1}{\sqrt{1+\beta ^2}}\leq  \xi  \leq  1.
\end{equation} 
Now we would like to show that if $(p^1,p^2)$ satisfies $\eqref{mv2.6p22}$ relatively to $\nu $, then $(p^1,\hat p^2)$ with $ \hat p^2 = p^2/\xi^2$ satisfies $\eqref{mv2.6p22}$ relatively to $\hat \nu$ with the same $\beta $. In fact, it is clear that 
\begin{equation}\label{mv2.15p2.2}
\epsilon |p^2 - \frac{1}{2} \langle p^1, \nu \rangle \langle \B ^{(s)}   p^1, \nu \rangle|_{\R^n}^{1/2} \leq -\beta \langle p^1, \nu \rangle  \qquad \Longleftrightarrow \qquad \epsilon | \hat p^2 - \frac{1}{2} \langle p^1, \hat \nu \rangle \langle \B ^{(s)}   p^1, \hat \nu \rangle|_{\R^n}^{1/2} \leq -\beta \langle p^1, \hat \nu \rangle .
\end{equation}
Moreover, if $ |p^1|_{\R^m}\leq -\sqrt{1+\beta ^2}\langle p^1,\nu \rangle $ then 
 \begin{equation}\label{mv2.14p2.2}
  |p^1| _{\R^m}\leq -\xi \sqrt{1+\beta ^2}\langle p^1,\hat \nu \rangle \leq -\sqrt{1+\beta ^2}\langle p^1,\hat \nu \rangle
\end{equation} 
and consequently $|p^1-\langle p^1,\hat \nu \rangle \hat \nu|_{\R^m} \leq -\beta \langle p^1,\hat \nu \rangle $.  By Step 1. we know that there exists $\hat \eta _s \in  \mathcal{A}$ such that if $p^2=(p_{m+1},\dots , p_{m+n})$, then   $p_{m+s} =  \langle \B^{(s)} \hat \eta _s, p^1 \rangle$ where $\eta _s = \xi^2 \hat \eta _s$ solves $p_{m+s}=  \langle \B^{(s)} \eta _s,  p^1 \rangle$ for $s=1,\dots , n$. Moreover 
\begin{equation}\label{mv2.16p2.2}
 \langle \hat \eta _s,  \hat \nu \rangle \leq -\sqrt{1-\frac{k^2}{2}}\,  |  \hat \eta _s|_{\R^m} \qquad \mbox{and} \qquad  \langle  p^1 -  \hat \eta _s,  \hat \nu \rangle \leq -\sqrt{1-\frac{k^2}{2}} \, |  p^1-  \hat \eta _s|_{\R^m}
\end{equation} 
i.e. \eqref{mv2.12} follows.
Taking into account $\eqref{mv2.9p2}$, $\eqref{mv2.13p2.2}$ and $\eqref{mv2.16p2.2}$ we get the first inequality in $\eqref{mv2.7p22}$, indeed
\begin{equation*}
\langle  \eta _s, \nu \rangle  = \xi  \langle  \eta _s,\hat  \nu \rangle  \leq \xi ^3  \langle \hat \eta _s,\hat  \nu \rangle \leq -\xi |\eta _s|_{\R^m} \sqrt{1-\frac{k^2}{2}} \,  \leq - |\eta _s |_{\R^m}\frac{\sqrt{1-\frac{k^2}{2}}}{\sqrt{1+\beta ^2}} \, \leq -|\eta _s|_{\R^m} \sqrt{1-k^2}
\end{equation*}
for every $s=1,\dots ,n$. Finally using $\eqref{mv2.16p2.2}$ 
\begin{equation*}
\begin{aligned}
\langle  p^1- \eta _s, \nu \rangle & = \langle  p^1- \xi ^2 \hat \eta _s, \nu \rangle = \xi ^2  \langle  p^1- \hat \eta _s, \nu \rangle + (1-\xi ^2) \langle  p^1, \nu \rangle = \xi ^3  \langle  p^1- \hat \eta _s, \hat \nu \rangle + \xi (1-\xi ^2) \langle p^1, \hat \nu \rangle \\
& \leq  \xi ^3  \biggl( - \sqrt{1-\frac{k^2}{2} } \, \, | p^1 - \hat \eta _s|_{\R^m} \biggl) + \xi(1-\xi ^2) \langle  p^1,\hat \nu \rangle  
\end{aligned}
\end{equation*} 
and by $\eqref{mv2.13p2.2}$ and $\eqref{mv2.14p2.2}$
\begin{equation*}
\begin{aligned}
\langle  p^1- \eta _s, \nu \rangle & \leq  \xi   \biggl( - \sqrt{1-\frac{k^2}{2} } \,\,  |\xi ^2  p^1- \eta _s|_{\R^m} \biggl) + (1-\xi ^2) \frac{-| p^1 |_{\R^m}}{\sqrt{1+\beta^2}} \\
& \leq - \sqrt{1-k^2} \, \, |\xi ^2  p^1 - \eta _s|_{\R^m}  -  \frac{|(1-\xi ^2) p^1 |_{\R^m}}{\sqrt{1+\beta ^2}} \\
& \leq - \sqrt{1-k^2} \, \, |\xi ^2  p^1- \eta _s| _{\R^m} -  |(1-\xi ^2) p^1|_{\R^m}\sqrt{1-k^2} \\
& \leq -\sqrt{1-k^2} \, |   p^1-  \hat \eta |_{\R^m}
\end{aligned}
\end{equation*} 
i.e. the second inequality in $\eqref{mv2.7p22}$ follows.
\end{proof}

Now we are able to show the proof of Theorem \ref{teo11montivittone}. The following proof is based on the following observation: if we start from a point of $E \cap \partial \mathcal U(0,r)$ with positive lower density and we move for a short time along a horizontal direction near $\nu $, then we remain in the set of positive lower density of $E$. We can then show that for each point of $E \cap \partial \mathcal U(0,r)$ there is a truncated lateral cone with fixed opening that is contained in $E$. We use the similar technique exploited in Theorem 1.1 in \cite{biblioMV} in the context of Heisenberg groups.

\begin{proof}[Proof of Theorem $\ref{teo11montivittone}$.]
Possibly modifying $E$ in a $\mathcal{L}^{m+n}$-negligible set, we can assume that $E$ coincides with the set of points where $E$ has positive lower density. Precisely
\begin{equation*}\label{mvt15}
E= \biggl\{ p\in \G \, :\, \liminf_{\delta \to 0} \frac{\mathcal{L}^{m+n}(E\cap \mathcal U_e(p,\delta))}{\mathcal{L}^{m+n}(\mathcal U_e(p,\delta ))} >0 \biggl\}
\end{equation*}
where $\mathcal U_e(p,\delta )$ is the Euclidean ball centered at $p$ having radius $\delta>0$. 

Let $\beta =\beta (k)>0$ as in Lemma $\ref{mvprop2.2}$. We would like to show that for every $p\in E$ \begin{equation}\label{Kmvt14}
\{ q\in \mathcal U(0,r) \, :\, \| \nu ^\perp (p^{-1} q) \| <-\beta \langle p^{-1} q, \nu \rangle  \} \subset E
\end{equation} 

First we define 
\begin{equation*}
\begin{aligned}
\mathcal{A}_1(p)& :=\{ \exp t Z_\mu (p) \in \mathcal U(0,r) \, :\, t\geq 0, \mu \in \mathbb{S}^{m-1}_k \}\\
\end{aligned}
\end{equation*}
where 
$ \mathbb{S}^{m-1}_k:=\{ \mu \in  \mathbb{S}^{m-1} : \langle \mu , \nu  \rangle \leq - \sqrt{1-k^2} \, \} $ and $Z_\mu$ is the left invariant vector field
\begin{equation*}
Z_\mu =\mu _1 X_1 + \dots +\mu _{m} X_m, \quad \mbox{ for }  \mu =(\mu _1, \dots, \mu _m)\in \mathbb{S}^{m-1}_k.
\end{equation*} 
For any $\xi \in \C^1_c (\mathcal U(0,r) , \R)$ such that $\xi \geq 0$ and for all $\mu \in \mathbb{S}^{m-1}_k$ it follows
\begin{equation*}
\int _E Z_\mu \xi \, d\mathcal{L}^{m+n}=- \int _{\mathcal U(0,r)} \xi \langle \mu , \nu _E \rangle \, d |\partial E|_\G \leq 0.
\end{equation*} 
Indeed we know that $\langle \mu , \nu _E (p)\rangle \geq 0$ for $|\partial E |_\G$-a.e. $p\in \mathcal U(0,r)$  because $\langle \mu , \nu \rangle \leq - \sqrt{1-k^2}$ and $\langle \nu_E , \nu \rangle \leq -k$. 
Then using Lemma $\ref{mvlemma2.1}$ we conclude that if $p\in E\cap \mathcal U(0,r)$, $t>0$ is such that $\exp tZ_\mu (p)\in \mathcal U(0,r)$ and $\delta >0$ is small enough, then 
\begin{equation*}
\mathcal{L}^{m+n} (E\cap \exp tZ_\mu (\mathcal U_e(p,\delta ))) \geq \mathcal{L}^{m+n} (E\cap  \mathcal U_e(p,\delta )).
\end{equation*}
Moreover by $\mathcal{L}^{m+n} ( \exp tZ_\mu (\mathcal U_e(p,\delta ))) = \mathcal{L}^{m+n} (  \mathcal U_e(p,\delta ))$ we deduce 
\begin{equation*}
\liminf_{\delta \to 0} \frac{\mathcal{L}^{m+n}(E\cap \exp tZ_\mu (\mathcal U_e(p,\delta )))}{\mathcal{L}^{m+n}(\exp tZ_\mu (\mathcal U_e(p,\delta )))} \geq  \liminf_{\hat r \to 0} \frac{\mathcal{L}^{m+n}(E\cap \mathcal U_e(p,\delta))}{\mathcal{L}^{m+n}(\mathcal U_e(p,\delta ))} >0
\end{equation*}
and consequently the point $q=\exp tZ_\mu (p)$ satisfies
\begin{equation*}
\liminf_{\delta \to 0} \frac{\mathcal{L}^{m+n}(E\cap \mathcal U_e(q,\delta ))}{\mathcal{L}^{m+n}(\mathcal U_e(q,\delta ))} >0.
\end{equation*}
This implies that $\exp tZ_\mu (p) \in E$ and $\mathcal{A}_1(p) \subset E$. Now if $p=0 \in E$, then
\begin{equation*}
\begin{aligned}
\mathcal{A}_1(0)& =\{ \exp t Z_\mu (0) \in \mathcal U(0,r) \, :\, t\geq 0, \mu \in \mathbb{S}^{m-1}_k \}\\
& = \{ (\eta , 0 )\in \G \, :\, \eta \in \R^m, \, \langle \eta , \nu  \rangle \leq -|\eta |\sqrt{1-k^2}, \, |\eta |<r \}.
\end{aligned}
\end{equation*}
Moreover if we consider the conditions $\eqref{mv2.7p22}$ of Lemma $\ref{mvprop2.2}$ and we define 
\begin{equation*}
\begin{aligned}
\mathcal{A}_2 & :=\{ \exp t Z_\mu (\eta , 0) \in \mathcal U(0,r) \, :\, t\geq 0, \mu \in \mathbb{S}^{m-1}_k, \,  (\eta, 0)\in \mathcal{A}_1(0) \}\\
\end{aligned}
\end{equation*}
then
\begin{equation*}
\begin{aligned}
\mathcal{A}_2 & = \{ (p^1,p^2 )\in \G \, : \mbox{ there are } \eta_1, \dots ,\eta_n  \in \R^m, \, |\eta _s|_{\R^m}<r \mbox{ such that } \eqref{mv2.7p22} \mbox{ holds }   \}
\end{aligned}
\end{equation*}
and $\mathcal{A}_2\subset E$.  Hence by Proposition $\ref{mvprop2.2}$ we obtain 
\begin{equation*}
\{ q\in \mathcal U(0,r) \, :\, \| \nu ^\perp (q) \| <-\beta \langle q, \nu \rangle  \} \subset \mathcal{A}_2 \subset E
\end{equation*}
i.e. \eqref{mvt14k} is true in the case $p=0$. Precisely  \eqref{mvt14k} follows  for all for each $p\in E$ from the case $p=0$ by a left translation.

Now we consider $\G-E$ where $\nu _{\G-E }=-\nu _E $ in $\mathcal U(0,r)$. We can repeat the previous argument and we obtain for each $P$ where $\G-E$ has positive lower density,
\begin{equation}\label{Kmvt15}
\{ q\in \mathcal U(0,r) \, :\, \| \nu ^\perp (p^{-1} q) \| <\beta  \langle p^{-1} q, \nu \rangle  \} \subset \G -E
\end{equation}
Precisely \eqref{Kmvt15} holds for all $p\in \mathcal U(0,r)-E$ because $\G -E$ has density $1$ at such $p$.

Approximating a point  $p\in \partial E \cap \mathcal U(0,r) $ with a sequence of points in $E\cap \mathcal U(0,r)$, from \eqref{Kmvt14} we obtain \eqref{mvt14k}. Moreover approximating a point  $p $ with a sequence of points in $\mathcal U(0,r) - E$, using \eqref{Kmvt15}, \eqref{mvt15k} holds. Possibly we have to take a smaller $\beta $. 
\end{proof}

Finally we are able to show the proof of Theorem \ref{teo11montivittoneTO}.


\begin{proof}[Proof of Theorem $\ref{teo11montivittoneTO}$.]
Let $\mathbf{P}_{ \nu ^\perp } :\G \to \nu ^\perp$ be the projection map. By \eqref{Kmvt14} we have that $\mathbf{P}_{ \nu ^\perp } (E\cap \mathcal U(0,r))$ is open set in $\mathbf{P}_{ \nu ^\perp } ( \mathcal U(0,r))$ and relatively open in $ \nu ^\perp $. Let 
\begin{equation*}
\hat{ \mathcal O} := \{ p\in  \mathbf{P}_{ \nu ^\perp } (E\cap \mathcal U(0,r)) \,:\, \mbox{there is } t\in \R \mbox{ such that } \exp tZ_\nu (p) \in  \mathcal U(0,r) -E\}.
\end{equation*}
By \eqref{Kmvt14} and \eqref{Kmvt15} we deduce that $\hat{ \mathcal O}$ is relatively open in $\mathbf{P}_{ \nu ^\perp } ( \mathcal U(0,r))$ and so in $ \nu ^\perp $. Then from Theorem \ref{teo11montivittone}, the function $\hat \phi = \phi \nu :\hat{ \mathcal O} \to \G$  defined as
\begin{equation*}
\phi (p):=\sup \{t\in \R \, :\, \exp tZ_\nu (p) \in  \mathcal U(0,r) \mbox{ and } \chi_E ( \exp tZ_\nu (p)) =1 \},
\end{equation*}
is intrinsic Lipschitz map because $\graph{\hat  \phi } \cap \{ Q\in \mathcal U(0,r) \, :\, \| \nu ^\perp (p^{-1} q) \| <\beta  \langle p^{-1} q, \nu \rangle  \} = \emptyset$ (see Definition \ref{FMSdefi2.3.3}) and 
\begin{equation*}
\partial E\cap \mathcal U(0,r)=\{ p\cdot \hat \phi (p) \,:\, p\in \mathcal O\},
\end{equation*}
i.e. the thesis is true.
\end{proof}

\section{area formula in  carnot groups of step 2}
In this section we prove an area formula beside the spherical Hausdorff measure for the graph of an intrinsic Lipschitz function (see Theorem \ref{citteo16OK}) and that the
pointwise gradient coincides with the weak one (see Proposition \ref{theo4.7}). We observe that this fact is not elementary at all in our situation, since the $\nabla^{\hat \phi} $ is not well defined when $\hat \phi$ is an intrinsic Lipschitz function. 

In this section we examine a Carnot group $\G $ of step 2 and $\V$, $\W$ are the complementary subgroups defined in \eqref{5.2.0}.

Let $\hat \phi :\hat{ \mathcal O} \to \V$ be intrinsic Lipschitz function, where $\hat{ \mathcal O}$ is an open in $\W$ and $\phi:\mathcal O\to \R$ is the map associated to $\phi$ as in $\eqref{phipsi}$. We recall that $\Phi :\mathcal O \to \G$ is the map graph of $\hat \phi $ defined as $$\Phi (a):=i(a) \cdot \hat \phi (i(a)).$$
Moreover we define the intrinsic subgraph of $\hat \phi$ as 
\begin{equation}\label{defSubgraph}
\E=\E_{\hat \phi} := \{i(a)\cdot (t,0,\dots, 0 )\in  \hat{ \mathcal O} \cdot \V \,:\, t < \phi(a)  \}.
\end{equation}

We begin with a result about the intrinsic generalized inward normal $\nu _{\E}$ to the intrinsic subgraph:

\begin{lem}\label{coro4.2}
Let $\G := (\R^{m+n}, \cdot  , \delta_\lambda )$ be a Carnot group of step $2$ and $\V$, $\W$ the complementary subgroups defined in \eqref{5.2.0}.  Let $\hat \phi :\hat{ \mathcal O} \to \V$ be an intrinsic Lipschitz function, where $\hat{ \mathcal O}$ is an open subset of $\W$ and $\phi :\mathcal O  \to \R$ is the map associated to $\hat \phi$ as in $\eqref{phipsi}$.

 Then the intrinsic generalized inward normal $\nu _{\E}$ to the intrinsic subgraph $\E$ in $\G$ has the following representation
\begin{equation}\label{63cit}
\nu_{\E}\left(\Phi (a)\right)=\left( - \frac{1}{\sqrt{1+|\nabla^{\hat \phi} \hat\phi (i(a)) |_{\R^{m-1}}^2}} ,\frac{\nabla^{\hat \phi} \hat\phi (i(a))}{\sqrt{1+|\nabla^{\hat \phi} \hat\phi (i(a))|_{\R^{m-1}}^2}} \right)
\end{equation}
for a.e. $a\in \mathcal O$. 
\end{lem}

\begin{proof}
By Theorem \ref{Theorem 4.3.5fms}, Theorem \ref{teo433} and Theorem \ref{teo3.2.8} we have that $\hat \phi$ is intrinsic differentiable a.e. in $\hat{ \mathcal O}$ and for all $i(a)\in \hat{ \mathcal O} $ point of intrinsic differentiability of $\hat \phi $, there exists a unique  $d\hat \phi _{i(a)}:\W\to \V$ intrinsic differential of $\hat \phi $ at $i(a)$ such that
\begin{equation*}
\graph {d\hat \phi _{i(a)}}=\Big\{ (p_1,\dots , p_{m+n}) \in \G \, : \,  \sum_{j=1}^{m} \nu _{\E}^{(j)} (\Phi (a)) p_j  =0 \Big\}
\end{equation*}
where $\nu _{\E}^{(1)},\dots , \nu _{\E}^{(m)}$ are the components of $\nu _{\E}$. Then we obtain
\begin{equation*}
\begin{aligned}
\left\{ b\cdot d\hat \phi _{i(a)}(b) \in \G \,  :\, b\in \W \right\} =\Big\{ (p_1,\dots , p_{m+n}) \in \G \, : \,  \sum_{j=1}^{m} \nu _{\E}^{(j)} (\Phi (a)) p_j =0 \Big\}. 
\end{aligned}
\end{equation*}
By $\eqref{DISSUdifferential}$ there is $( D_2^{ \phi} \phi (a), \dots , D_m^{ \phi} \phi (a) )\in \R^{m-1}$ associated to $d\hat \phi _{i(a)}$ such that
\[
d\hat \phi _{i(a)}(b)= \Big(  \sum_{j=2}^{m} D_j^{\phi} \phi (a) x_j ,0,\dots , 0\Big), 
\]
for all $b=(0,x_2,\dots , x_m, y_1,\dots , y_n)\in \W $, and consequently recalling that $(b\cdot  d\hat \phi _{i(a)}(b))^1=(  \sum_{j=2}^{m} D_j^{\phi} \phi (a) x_j, x_2,\dots ,x_m )$ we deduce 
\begin{equation*}
\nu _{\E}^{(1)} (\Phi (a))  \sum_{j=2}^{m} D_j^{\phi} \phi (a) x_j  + \sum_{j=2}^{m} \nu _{\E}^{(j)} (\Phi (a)) x_j =0
\end{equation*} 
for all $b=(0,x_2,\dots , x_m, y_1,\dots , y_n)\in \W $. The thesis follows choosing $b=(0,\dots, 0 ,1,0,\dots, 0)$ where $j$-th element is $1$ for $j=2,\dots ,m$  and recalling that $| \nu _{\E}(\Phi (a))|_{\R^{m}}=1$ a.e. in $\mathcal O $.
\end{proof}

\begin{lem}[\cite{biblio21}, Lemma 4.2.10]\label{coro4.3}
 Let $\hat \phi :\hat{ \mathcal O} \to \V$ be an intrinsic Lipschitz function. Then there exists $C(\W, \V)>0$ such that
\begin{equation*}
(\Phi )_* (\mathcal{L}^{m+n-1} \res \W)=- C(\W, \V) \nu _{\E}^{(1)} |\partial \E |_\G
\end{equation*}
where $(\Phi )_* (\mathcal{L}^{m+n-1} \res \W )$ denotes the image of $\mathcal{L}^{m+n-1} \res \W$ under the map $\Phi $ and  $\nu _{\E}^{(1)}$ is the first component of the intrinsic generalized inward normal to  the intrinsic subgraph $\E$. 
\end{lem}

\begin{prop}\label{prop4.6}
Let $\G := (\R^{m+n}, \cdot  , \delta_\lambda )$ be a Carnot group of step $2$ and $\V$, $\W$ the complementary subgroups defined in \eqref{5.2.0}.  Let $\hat \phi :\hat{ \mathcal O} \to \V$ be an intrinsic Lipschitz function, where $\hat{ \mathcal O}$ is an open subset of $\W$ and $\phi :\mathcal O  \to \R$ is the map associated to $\hat \phi$ as in $\eqref{phipsi}$. 

 Then for every $\xi =(\xi _1,\dots , \xi _m) \in \C^1_c (\hat{ \mathcal O} \cdot \V , \R^m)$ we have
\begin{equation*}\label{div70}
\int_{\E } \mbox{div} _\G \xi \, d\mathcal{L}^{m+n} =  \frac{1}{C(\W,\V)} \int_\mathcal O \xi _1 \circ \Phi - \sum_{j=2}^{m} D_j^{\phi} \phi  ( \xi  \circ \Phi )_j \, d\mathcal{L}^{m+n-1}
\end{equation*}
where $C(\W,\V)>0$ is given by Lemma $\ref{coro4.3}$ and $\Phi :\mathcal O \to \G$ is the map graph of $\hat \phi $.
\end{prop}

\begin{proof}
  Thanks to Proposition \ref{Theorem 4.2.9fms} we know that $\E$ is a set of locally finite perimeter in $\G$ and consequently  by Structure  Theorem of BV$_\G$ functions (see Theorem \ref{structure theorem BV}) there exists a unique $|\partial \E|_\G$-measurable function $\nu _\E:\hat{ \mathcal O}  \cdot \V \to \R^m$ such that $|\nu _\E|_{\R^{m}}=1$ $|\partial \E|_\G$-a.e. in $\hat{ \mathcal O} \cdot \V $ and 
\begin{equation}\label{eqDIV}
\int_{\E } \mbox{div} _\G \xi \, d\mathcal{L}^{m+n} =  - \int_{\hat{ \mathcal O} \cdot \V } \langle \xi ,\nu _\E\rangle\, d|\partial \E|_\G
\end{equation}
for all $\xi \in \C^1_c (\hat{ \mathcal O} \cdot \V , \R^m)$ with $|\xi |_{\R^{m}}\leq 1$.
Now using the Lemma $\ref{coro4.2}$ we have that the first component $\nu _\E^{(1)}$ of $\nu _\E$ is such that $\nu _\E^{(1)}\ne 0$ $|\partial \E|_\G$-a.e. in $\hat{ \mathcal O} \cdot \V $. Moreover from Lemma  $\ref{coro4.3}$ there exists $C(\W,\V)>0$ such that
\begin{equation*}
-\int_{\hat{ \mathcal O} \cdot \V } \langle \xi ,\nu _\E\rangle\, d|\partial \E|_\G =- \int_{\hat{ \mathcal O} \cdot \V } \frac{\langle \xi ,\nu _\E \rangle}{\nu _\E^{(1)}}\, \nu _\E^{(1)} \, d|\partial \E|_\G =  \frac{1}{C(\W,\V)}  \int_{\hat{ \mathcal O} \cdot \V } \frac{\langle \xi ,\nu _\E \rangle}{\nu _\E^{(1)}}\, d\Phi_* (\mathcal{L}^{m+n-1} \res \W).
\end{equation*}
Finally since a change of variables we conclude that
\begin{equation*}
 \frac{1}{C(\W,\V)}  \int_{\hat{ \mathcal O} \cdot \V } \frac{ \langle \xi ,\nu _\E\rangle}{\nu _\E^{(1)}}\, d\Phi_* (\mathcal{L}^{m+n-1} \res \W) =  \frac{1}{C(\W,\V)}  \int_\mathcal O  \frac{\langle \nu _\E \circ \Phi , \xi \circ \Phi \rangle}{\nu _\E^{(1)} \circ \Phi } \, d\mathcal{L}^{m+n-1}
\end{equation*} 
and consequently, by \eqref{63cit} for all $\xi =(\xi_1,\dots , \xi_m)\in \C^1_c (\hat{ \mathcal O} \cdot \V , \R^m)$ with $|\xi |_{\R^{m}}\leq 1$
\begin{equation}\label{lousodopo}
\begin{aligned}
\int_{\E } \mbox{div} _\G \xi \, d\mathcal{L}^{m+n} & =   -\int_{\hat{ \mathcal O} \cdot \V } \langle \xi ,\nu _\E \rangle\, d|\partial \E|_\G\\
 & =  \frac{1}{C(\W,\V)}  \int_\mathcal O \frac{\langle \nu _\E \circ \Phi , \xi \circ \Phi \rangle}{\nu _\E^{(1)} \circ \Phi } \, d\mathcal{L}^{m+n-1}\\
& =  \frac{1}{C(\W,\V)}  \int_\mathcal O \biggl( \xi_1 \circ \Phi + \sum_{j=2}^{m} \frac{(\nu _\E \circ \Phi )_j}{\nu _\E^{(1)} \circ \Phi } \,  (\xi \circ \Phi )_j \biggl)\, d\mathcal{L}^{m+n-1}\\
& =   \frac{1}{C(\W,\V)}  \int_\mathcal O \biggl( \xi_1 \circ \Phi - \sum_{j=2}^{m} D_j^{\phi} \phi  ( \xi  \circ \Phi )_j \biggl)\, d\mathcal{L}^{m+n-1}
\end{aligned}
\end{equation} 
Hence putting together the last equality and \eqref{eqDIV} we obtain the thesis.
\end{proof}

\begin{prop}\label{theo4.7}
Let $\G := (\R^{m+n}, \cdot  , \delta_\lambda )$ be a Carnot group of step $2$ and $\V$, $\W$ the complementary subgroups defined in \eqref{5.2.0}.  Let $\hat \phi :\hat{ \mathcal O} \to \V$ be a locally intrinsic Lipschitz function, where $\hat{ \mathcal O}$ is an open subset of $\W$ and $\phi :\mathcal O  \to \R$ is the map associated to $\hat \phi$ as in $\eqref{phipsi}$. 
 Then the intrinsic gradient $(D_2^\phi \phi ,\dots ,D_m^\phi \phi  )$ is also distributional, i.e. 
\begin{equation*} 
 \int _\mathcal O \phi \left( \,X_j \zeta +\phi  \sum_{s=1}^n b_ {j1 }^{(s)} Y_s \zeta \right)  \, d\mathcal{L}^{m+n-1} =- \int _\mathcal O D^\phi _j \phi \, \zeta \, d \mathcal{L}^{m+n-1}
\end{equation*} 
for every $\zeta \in \C^1_c (\mathcal O , \R)$ and $j=2,\dots , m$.
\end{prop}

\begin{proof}
 By standard considerations there is a sequence $(\phi _h)_{h\in \N} \subset \C^\infty _c (\mathcal O , \R)$ converging uniformly to $\phi $ on each $\mathcal O ' \Subset  \mathcal O$. We would like to prove that the sequence $(D_2^{\phi _h} \phi _h, \dots , D_m^{\phi _h} \phi _h)_{h\in \N}$ converges to $(D^\phi _2 \phi, \dots , D^\phi _m\phi )$ in the sense of distributions on each $\mathcal O ' \Subset  \mathcal O$. Indeed, we start to show that
\begin{equation}\label{72cit}
\int _\mathcal O \sum_{j=2}^{m} D_j^{\phi} \phi \, \eta _j   \, d  \mathcal{L}^{m+n-1} =\lim_{h\to \infty } \int _\mathcal O \sum_{j=2}^{m} D_j ^{\phi _h} \phi _h \, \eta _j  \, d \mathcal{L}^{m+n-1}  \qquad \forall \, \eta \in \C^1_c (\mathcal O , \R^{m-1})
\end{equation} 
 We denote by $\Phi _h :\mathcal O \to \G$ the graph map of $\hat \phi _h=(\phi_h,0,\dots ,0)$ defined as $\Phi _h(a):=i(a)\hat \phi _h(i(a))$ and $\E_h$ the intrinsic subgraph of $\hat \phi _h$. Therefore, by Proposition $\ref{prop4.6}$ we know that for each $\xi =(\xi_1, \dots ,\xi_m)\in \C^1_c (\hat{ \mathcal O} \cdot \V , \R^m)$ 
\begin{equation*}
\int_{\E} \mbox{div}_\G \xi \, d \mathcal{L}^{m+n} =  \frac{1}{C(\W,\V)} \int_\mathcal O \xi _1 \circ \Phi - \sum_{j=2}^{m} D_j^{\phi} \phi ( \xi  \circ \Phi )_j \, d  \mathcal{L}^{m+n-1}
\end{equation*}
where $C(\W,\V)>0$ is given by Lemma $\ref{coro4.3}$.
Moreover by the uniform convergence of $\phi _h$ to $\phi $ we conclude that
\begin{equation*}
\int_{\E} \mbox{div}_\G \xi \, d \mathcal{L}^{m+n} = \lim_{h\to \infty }\int_{\E_h} \mbox{div}_\G \xi \, d \mathcal{L}^{m+n}
\end{equation*}
Now we recall that $\C^1$ functions are uniformly intrinsic differentiable maps (see Theorem \ref{propC1implicauid}) and consequently they are also locally intrinsic Lipschitz maps (see Proposition \ref{lip1512DDD}). Hence we can apply Proposition $\ref{prop4.6}$ for every $\phi _h$ and we obtain that
\begin{equation*}
\int_{\E_h} \mbox{div} _\G \xi \, d \mathcal{L}^{m+n} =  \frac{1}{C(\W,\V)} \int_\mathcal O \xi _1 \circ \Phi _h-\sum_{j=2}^{m} D_j^{\phi _h} \phi _h ( \xi  \circ \Phi _h)_j \, d \mathcal{L}^{m+n-1}
\end{equation*}
for every $\xi \in \C^1_c (\hat{ \mathcal O} \cdot \V , \R^m)$. Finally putting together the last three equalities we have that
\begin{equation*}
\frac{1}{C(\W,\V)} \int_\mathcal O \xi _1 \circ \Phi -\sum_{j=2}^{m} D_j^{\phi} \phi  ( \xi  \circ \Phi )_j \, d \mathcal{L}^{m+n-1}=  \lim_{h\to \infty } \frac{1}{C(\W,\V)} \int_\mathcal O \xi _1 \circ \Phi _h-\sum_{j=2}^{m} D_j^{\phi _h} \phi_h ( \xi  \circ \Phi _h)_j  \, d \mathcal{L}^{m+n-1}
\end{equation*}
for all $\xi \in \C^1_c (\hat{ \mathcal O} \cdot \V , \R^m)$. 

Clearly if we choose $\xi _1=0$, then 
\begin{equation}\label{eqADHOC}
\int_\mathcal O  \sum_{j=2}^{m} D_j^{\phi} \phi  ( \xi  \circ \Phi )_j \, d \mathcal{L}^{m+n-1}=  \lim_{h\to \infty } \int_\mathcal O  \sum_{j=2}^{m} D_j^{\phi _h} \phi_h ( \xi  \circ \Phi _h)_j  \, d \mathcal{L}^{m+n-1}.
\end{equation}
If  we consider $$\xi _j(i(a)\cdot (t,0,\dots , 0)):=\eta _j(a)\rho (t)$$ with $\eta =(\eta_2, \dots , \eta _m) \in \C^1_c (\mathcal O ,\R^{m-1})$ and $\rho \in \C^1_c (\R)$ such that $\rho (t)=1$ for all $t\in [-\| \phi \|_{\mathcal{L}^\infty (\mbox{spt}(\eta) )}-1,\| \phi \|_{\mathcal{L}^\infty (\mbox{spt}(\eta) )}+1]$, then $\xi =(0, \xi _2,\dots , \xi _m) \in \C^1_c (\hat{ \mathcal O} \cdot \V , \R^m)$ and by \eqref{eqADHOC} we deduce that
\begin{equation*}
\int_\mathcal O   \sum_{j=2}^{m} D_j^{\phi} \phi (a) \eta _j(a)\rho (\phi (a))  \, d \mathcal{L}^{m+n-1}(a)=  \lim_{h\to \infty } \int_\mathcal O  \sum_{j=2}^{m} D_j^{\phi _h} \phi_h (a) \eta _j(a)\rho (\phi _h(a))  \, d \mathcal{L}^{m+n-1}(a).
\end{equation*}
Hence since $\phi _h$ converges uniformly to $\phi $, there is $\bar h \in \N$ such that for all $h\geq \bar h$ and for all $a\in $ spt$(\eta)$ we have
\begin{enumerate}
\item $\phi _h (a)\in [-M-1 ,M+1]$ for $M:=\| \phi \|_{\mathcal{L}^\infty (\mbox{spt}(\eta) , \R),}$
\item $\rho (\phi _h(a)) =1$;
\end{enumerate}
and consequently $\eqref{72cit}$ follows.

Now using again the uniformly convergence of $\phi _h$ to $\phi$ and $\eqref{72cit}$ we conclude that for all $\eta =(\eta _2,\dots , \eta _m) \in \C^1_c (\mathcal O , \R^{m-1})$ and for all $j=2,\dots , m$
\begin{equation*} 
\begin{aligned}
\int _\mathcal O D^\phi _j \phi \, \eta _j\, d \mathcal{L}^{m+n-1} & = \lim_{h\to \infty }  \int _\mathcal O D^{\phi _h}_j \phi _h\, \eta_j \, d \mathcal{L}^{m+n-1} \\
& = - \lim_{h\to \infty }  \int _\mathcal O \phi_h \left( \,X_j \eta _j +\phi _h \sum_{s=1}^n b_ {j1 }^{(s)} Y_s \eta _j\right)  \, d \mathcal{L}^{m+n-1}\\
& = -   \int _\mathcal O \phi \left( \,X_j \eta _j +\phi  \sum_{s=1}^n b_ {j1 }^{(s)} Y_s \eta _j\right)  \, d \mathcal{L}^{m+n-1}\\
\end{aligned}
\end{equation*} 
Then the thesis follows with $\zeta =\eta _j \in \C^1_c (\mathcal O, \R)$ for each $j=2,\dots , m$.
\end{proof}


\begin{theorem}\label{citteo16OK}
Let $\G := (\R^{m+n}, \cdot  , \delta_\lambda )$ be a Carnot group of step $2$ and $\V$, $\W$ the complementary subgroups defined in \eqref{5.2.0}.  Let $\hat \phi :\hat{ \mathcal O} \to \V$ be a locally intrinsic Lipschitz function, where $\hat{ \mathcal O} \subset \W$ is an open  set and $\phi :\mathcal O  \to \R$ is the map associated to $\hat \phi$ as in $\eqref{phipsi}$. 

 Then there exist $c=c(\G)>0$ and $C(\W,\V)>0$ such that the following area formula holds
\begin{equation*}\label{citarea}
|\partial \E |_\G (\hat{ \mathcal O} \cdot \V )= c \, \mathcal{S}^{\mathfrak q -1} (\Phi (\mathcal O )) =\frac{1}{C(\W,\V)}  \int _\mathcal O \sqrt{1+|\nabla^{\hat \phi} \hat\phi |_{\R^{m-1}}^2}\, d \mathcal{L}^{m+n-1}
\end{equation*} 
\end{theorem}

\begin{proof}
Thanks to Proposition \ref{Theorem 4.2.9fms} and Theorem \ref{structure theorem BV} we know that $|\partial \E|_\G$ is a Radon measure. Then a classical approximation result ensures the existence of sequence $(\xi _h)_{h\in \N}= (\xi _{h,1},\dots , \xi _{h,m})_{h\in \N}\subset \C^1_c (\hat{ \mathcal O} \cdot \V, \R^m)$ with $|\xi _h|_{\R^{m}}\leq 1$ such that 
\begin{equation}\label{convanu} 
\xi_h \to \nu _\E \quad |\partial \E|_\G \mbox{ -a.e. in } \hat{ \mathcal O} \cdot \V
\end{equation} 
and by Lemma $\ref{coro4.2}$ and Lemma $\ref{coro4.3}$ we have that $\xi _h \circ \Phi \to \nu _\E \circ \Phi $   a.e. in $ \mathcal O $. 

Moreover using \eqref{lousodopo} in Proposition \ref{prop4.6} we get that for all $h\in \N$
\begin{equation*}
- \int_{\hat{ \mathcal O} \cdot \V} \langle \xi _h , \nu _\E\rangle \, d|\partial \E|_\G = \frac{1}{C(\W,\V)}  \int _\mathcal O   \xi _{h,1 } \circ \Phi  -\sum_{j=2}^{m} D_j^{\phi} \phi ( \xi _h \circ \Phi )_j  \, d \mathcal{L}^{m+n-1}
\end{equation*}
Hence taking the limit as $h\to \infty $ in the last equality and by \eqref{convanu} and Lemma $\ref{coro4.2}$ we conclude  
\begin{equation*}\label{cit75}
|\partial \E|_\G (\hat{ \mathcal O} \cdot \V )= \frac{1}{C(\W,\V)}  \int_\mathcal O \sqrt{1+| \nabla^{\hat \phi} \hat\phi |_{\R^{m-1}}^2}\, d \mathcal{L}^{m+n-1}.
\end{equation*}
Finally, $|\partial \E |_\G (\hat{ \mathcal O} \cdot \V )= c \, \mathcal{S}^{\mathfrak q -1} (\Phi (\mathcal O ))$ is a direct consequence of results of Theorem \ref{Theorem 4.18fssc}  and Theorem \ref{teo433}.
\end{proof}

\section{Smooth approximation for intrinsic lipschitz maps}
In this section we characterize intrinsic Lipschitz functions as maps which can be approximated by a sequence of smooth maps, with pointwise convergent intrinsic gradient. 
The classical ap\-proxi\-ma\-tion by convolution does not apply because of the nonlinearity of $\nabla ^{\hat \phi } \hat \phi$. Here we use the similar technique in  \cite{biblio26} in the context of Heisenberg groups. The strategy is somehow indirect, indeed we ap\-proxi\-ma\-te in  Carnot groups of step 2 the intrinsic subgraph of intrinsic Lipschitz function $\hat \phi$, rather than the function itself. The key point is that the intrinsic subgraph of $\hat \phi$ is a set of locally finite $\G$-perimeter (see Proposition \ref{Theorem 4.2.9fms}). Precisely, we prove the following Theorem
\begin{theorem}\label{cmpssteo17}
Let $\G := (\R^{m+n}, \cdot  , \delta_\lambda )$ be a Carnot group of step $2$ and $\V$, $\W$ the complementary subgroups defined in \eqref{5.2.0}. Let $\hat \phi :\hat{ \mathcal O} \to \V$ be a locally intrinsic Lipschitz function, where $\hat{ \mathcal O}$ is an open subset of $\W$ and $\phi :\mathcal O  \to \R$ is the map associated to $\hat \phi$ as in $\eqref{phipsi}$. 

 Then if $\nabla^{\hat \phi} \hat\phi \in \mathcal{L}^\infty _{loc}(\hat{ \mathcal O} ,\R^{m-1})$ there exists a sequence $(\hat \phi _h)_h \subset \C^\infty (\hat{ \mathcal O} ,\V)$ such that for all $\hat{ \mathcal O'} \Subset \hat{ \mathcal O}$ there is $C=C(\mathcal O')>0$ such that
\begin{enumerate}
\item  $\hat \phi _h$ uniformly converges to $\hat \phi $ in $\hat{ \mathcal O'} $
\item $|\nabla^{\hat \phi _h} \hat\phi _h(a) | \leq C$ $\forall a\in \hat{ \mathcal O'} $, $h\in \N$
\item $\nabla^{\hat \phi_h} \hat\phi _h (a)\to \nabla^{\hat \phi} \hat\phi   (a)$ a.e. $a\in \hat{ \mathcal O'}$
\end{enumerate}
\end{theorem}

Before stating the approximation theorem we need to recall the following result.
%
%
%

\begin{theorem}[\cite{biblioV}]\label{teoVisintin}
 Let $f:\R^{m+n} \to \R$ be a strictly convex function and let $(g_h)_h$ and $g$ be in $\mathcal{L}^1(\Omega , \R^{m+n})$ where $\Omega \subset \R^{m+n}$. If
\begin{enumerate}
\item $g_h \to g \, \, $ weakly in $\mathcal{L}^1(\Omega , \R^{m+n})$
\item $\int _\Omega f \circ g_h $ d$ \mathcal{L}^{m+n} \to \int _\Omega f \circ g $ d$ \mathcal{L}^{m+n}$
\end{enumerate}
then $g_h \to g \,$ strongly in $\mathcal{L}^1(\Omega , \R^{m+n})$.
\end{theorem}

\begin{proof}[Proof of Theorem $\ref{cmpssteo17}$.]

Let $\mathcal O' \Subset \mathcal O$ and $M:=\| \phi \|_{\mathcal{L}^\infty (\mathcal O', \R)} <+\infty $.

We split the proof in several steps. First we use Friedrichs' mollifier to regularize $\chi _{E }$ with a family of functions $(f_\alpha )_{\alpha >0}$ which satisfies Theorem \ref{teo4.1}. Then we consider the level set of $(f_\alpha )_{\alpha >0}$ which defines a family of functions $(\phi _\alpha )_{\alpha >0}$. Finally we prove that $(\phi _\alpha )_{\alpha >0}$ is the family of functions, up to subsequence, that we were looking for.

$\mathbf{Step \, 1.}$  Let $\E:=\E_{\hat \phi}$ the intrinsic subgraph of $\hat \phi |_{\mathcal O'  }$ defined as \eqref{defSubgraph}. For every $\alpha >0$ we consider $f_\alpha :\G\to \R$ given by
\begin{equation}\label{conv} 
\begin{aligned} 
f_\alpha (p):= (\rho _\alpha \ast \chi _{\E })(p)& =\int _\G \rho _\alpha (pq^{-1}) \chi_{\E }(q)\, d \mathcal{L}^{m+n}(q)\\
& =\int _\G \rho _\alpha (q) \chi_{\E}(q^{-1} p)\, d \mathcal{L}^{m+n}(q)
\end{aligned} 
\end{equation} 
where $\rho _\alpha (p):=\alpha ^{m+2n}\rho (\delta _{1/\alpha }(p))$ and $\rho \in \C^\infty _c (\mathcal U(0,1), \R)$ is a smooth mollifier with $\rho (p^{-1})=\rho (-p)=\rho (p)$ for every $p\in \G$. 

By properties of convolution in $\G$ introduced in \cite{biblio15}, we know that  $f_\alpha \in \C^\infty _c (\G , \R)$ and spt$(f_\alpha )\subset \mathcal U(0,\alpha )\cdot $ spt$(\chi_{\E})$ for every $\alpha >0$. Precisely, for all $\alpha >0$
\begin{equation*}
 f_\alpha (p)\in [0,1] \quad   \mbox{ for all }  p\in \G
\end{equation*}
and  for all sufficiently small $\alpha >0$
\begin{equation}\label{cit76}
f_\alpha (p)=1 \quad \mbox{ for all } p=(p_1,\dots , p_{m+n})\in \G  \, \mbox{ such that }  p_1 \leq -2M
\end{equation}
Moreover by definition of $\E $ we know that $\E $ is an open set in $\G $ and
\begin{equation*}
 \mbox{ spt} (\chi _{\E }) =  \mbox{clos} ( \E ) \subset  \{a\cdot (t,0,\dots, 0 )\in  \mbox{clos} (\hat{ \mathcal O'} ) \cdot \V \,:\,  t \leq M  \}.
\end{equation*} 
As a consequence
\begin{equation*}
 \mbox{ spt} (f_\alpha )  \subset \overline{\mathcal U(0,\alpha )}\cdot   \mbox{ spt} (\chi_{\E})\subset  \{a\cdot (t,0,\dots, 0 )\in  \hat{ \mathcal O'} \cdot \V \,:\,  t < 2M  \}   \quad \mbox{ for }  \alpha <M
\end{equation*} 
and 
\begin{equation}\label{cit79}
f_\alpha (p)=0 \quad \mbox{ for all } p=(p_1,\dots , p_{m+n})\in \G \,   \mbox{ such that }  p_1 \geq 2M
\end{equation}
i.e. $f_\alpha $ is constant far from the graph of $\hat \phi $, so $\eqref{conv}$ is extended only in a neighborhood of the graphs itself.\\
$\mathbf{Step \, 2.}$ For every $\alpha >0$ and $c\in (0,1)$, let
\begin{equation*}\label{cit85}
S_\alpha:= \left\{ p=(p_1,\dots , p_{m+n})\in \hat{ \mathcal O'} \cdot \V   \, : \, p_1 \in (-2M,2M),\, f_\alpha (p)=c  \right\}.
\end{equation*}
We note that if
\begin{equation*}\label{citRANGO}
\mbox{rank} \nabla _\G f_\alpha (p)=1, \quad \mbox{for all } p\in S_\alpha
\end{equation*}
then $S_\alpha $ is the level set of a map $f_\alpha \in \C^\infty(\G , \R)$ and so $f_\alpha \in \C^1_\G (\G , \R)$ such that $\mbox{rank} \nabla _\G f_\alpha$ is maximum, i.e. $S_\alpha $ is $\G$-regular hypersurface. More precisely, we show that
\begin{equation}\label{cit85}
X_1f_\alpha (p) <0 \quad \mbox{for all } p\in S_\alpha
\end{equation}
where $X_1f_\alpha (p)$ is the first component of horizontal gradient $\nabla _\G f_\alpha $ of $ f_\alpha $ in $p$. Indeed let $\mathcal A ':=\{p= (p_1,\dots , p_{m+n})\in \hat{ \mathcal O'} \cdot \V \, :\, p_1\in    (-3M,3M) \}$ and $\vp \in \C^\infty _c(\mathcal A ' , \R)$, then
\begin{equation*}
\begin{aligned}
 \langle X_1f_\alpha , \vp \rangle &=-\int_{ \mathcal A ' } f_\alpha (q)X_1 \vp (q) \, d\mathcal L^{m+n} (q)\\
 &=-\int_{ \overline{\mathcal U(0,\alpha )} }  \rho_\alpha (p)  \, d\mathcal L^{m+n} (p)  \int_{ \mathcal A ' } \chi _\E (p^{-1}q) X_1 \vp (q)\, d\mathcal L^{m+n} (q)\\
  &=-\int_{ \overline{\mathcal U(0,\alpha )} }  \rho_\alpha (p)  \, d\mathcal L^{m+n} (p)  \int_{\tau_{P^{-1}} (\mathcal A ') } \chi _\E (q) X_1 \vp (pq)\, d\mathcal L^{m+n} (q)\\
\end{aligned}
\end{equation*}
With the notation $\vp _p(q):=\vp (pq)$ we have $X_1(\vp (pq))= X_1 \vp_p (q)$ because $X_1$ is left-invariant; moreover $\vp _p \in \C^\infty _c (\tau _{p^{-1}} \mathcal A ' , \R)$ and by integration by parts, we obtain
\begin{equation*}
\begin{aligned}
  \int_{\tau_{P^{-1}} (\mathcal A ') } \chi _\E (q) X_1 \vp _P(q)\, d\mathcal L^{m+n} (q) = - \int_{\tau_{p^{-1}} (\mathcal A ') } \nu^{(1)}_{\E} (q) \vp _P(q)\, d |\partial \E|_\G(q)
\end{aligned}
\end{equation*}
Since $\mbox{spt}(\vp_P) \Subset \tau_{P^{-1}} (\mathcal A ')$ and $P\in \overline{\mathcal U(0,\alpha )}$ if $\alpha$ is small enough, we can replace $\tau_{p^{-1}} (\mathcal A ')$ by $\mathcal A '$. Thus by Fubini-Tonelli Theorem and a change of variable, we get
\begin{equation*}
\begin{aligned}
 \langle X_1f_\alpha , \vp \rangle &=\int_{\mathcal A ' } \nu^{(1)}_{\E} (q) \, d |\partial \E|_\G(q) \left(  \int_{\G} \rho _\alpha (p) \vp (pq) \, d\mathcal L^{m+n} (p) \right)\\
\end{aligned}
\end{equation*}

Then for all $p\in \mathcal A '$ and for all $\alpha >0$ small enough
\begin{equation}\label{83cit}
\begin{aligned}
X_1f_\alpha (p) & = \int_{ \mathcal A ' }\rho _\alpha (p q^{-1}) \nu^{(1)}_{\E} (q)\, d |\partial \E|_\G(q)\\
&= \int_{\mathcal U(0,\alpha )\cdot P }\rho _\alpha (p q^{-1}) \nu^{(1)}_{\E } (q)\, d |\partial \E |_\G(q)
\end{aligned}
\end{equation}
In particular we immediately deduce from $\eqref{83cit}$ the following assertion: For all couple $(\hat{ \mathcal O'} ,\hat{ \mathcal O} _0')$ of open and bounded subsets of $\W$ with $\hat{ \mathcal O'} \Subset  \hat{ \mathcal O} _0'$ there is $\bar \alpha =\alpha (\hat{ \mathcal O} _0')>0$ such that for every $0<\alpha <\bar \alpha$
\begin{equation*}
\int_{   \mathcal A }|\nabla _\G f_\alpha|\, d \mathcal{L}^{m+n} \leq |\partial \E |_\G(  \mathcal A_0 ).
\end{equation*}
where $ \mathcal A:=\{p= (p_1,\dots , p_{m+n})\in \hat{ \mathcal O'} \cdot \V \, :\, p_1\in    (-2M,2M) \}$ and $ \mathcal A_0 :=\{p= (p_1,\dots , p_{m+n})\in \hat{ \mathcal O}_0' \cdot \V \, :\, p_1\in    (-2M,2M) \}$. Moreover since Lemma $\ref{coro4.2}$ we get
\begin{equation*}
\nu^{(1)}_{\E} \circ \Phi =-\frac{1}{\sqrt{1+|\nabla^{\hat \phi} \hat\phi |_{ \R^{m-1}}^2}}\quad \mbox{in } \mathcal O'
\end{equation*}
and if we put 
\begin{equation*}
I_\alpha (p):= \int_{ \mathcal U(0,\alpha )\cdot p}\rho _\alpha (pq^{-1})\, d |\partial \E|_\G(q)
\end{equation*}
 then  using $\eqref{83cit}$ we have
\begin{equation}\label{cit87}
X_1f_\alpha (p) \leq -\frac{1}{\sqrt{1+\|\nabla^{\hat \phi} \hat\phi \|^2_{\mathcal L^\infty (\hat{ \mathcal O'} , \R^{m-1})}}}I_\alpha (p) \quad  \mbox{ for all } p \in   \mathcal A' .\end{equation}
Moreover for every fixed $c\in (0,1)$ and for all $p\in  \mathcal A$ such that $f_\alpha (p)=c$, then
\begin{equation}\label{cit88}
 \mathcal{L}^{m+n}((\mathcal U(0,\alpha )\cdot p) \cap \E )>0 \quad  \mbox{and} \quad  \mathcal{L}^{m+n}((\mathcal U(0,\alpha )\cdot p)  \cap \E^c )>0.
\end{equation}
Indeed, by contradiction, if we assume $ \mathcal{L}^{m+n}((\mathcal U(0,\alpha )\cdot p) \cap \E)=0$. Then because $\E$ is open, we can assume $(\mathcal U(0,\alpha )\cdot p) \cap \E = \emptyset$. Hence by the definition of convolution, we have that $f_\alpha (p)=0 \ne c$ and then a contradiction. In the similar way it follows that $f_\alpha (p)=1$ if we suppose by contradiction $ \mathcal{L}^{m+n}((\mathcal U(0,\alpha )\cdot p) \cap \E^c)=0$.

Now by $\eqref{cit88}$ and by Theorem \ref{isoperimetric inequalitieC}, we have that $|\partial \E |_\G((\mathcal U(0,\alpha )\cdot p )>0$ and $I_\alpha (p )>0 $ for  $p\in \mathcal A $ with  $f_\alpha (p)=c$. As a consequence using $\eqref{cit87}$,  $\eqref{cit85}$ holds and $S_\alpha $ is a $\G$-regular hypersurface.

$\mathbf{Step \, 3.}$ Now we show that  $(S_\alpha)_\alpha$ implicitly defines a sequence $(\hat \phi_\alpha )_{\alpha}$ with $\hat \phi_\alpha=(\phi_\alpha ,0,\dots , 0)$ and ,up to subsequence, it is the family of functions which we were looking for.

Fix $\alpha $ and $c\in (0,1)$. Because  $f_\alpha$ is continuous map such that  
\begin{equation*}
\begin{aligned}
f_\alpha ( i(a) \cdot (-2M,0,\dots , 0)) & =1>c \\
f_\alpha (i(a)\cdot (2M,0,\dots , 0)) & = 0<c
\end{aligned}
\end{equation*}
for all $a\in \mathcal O' $,  there is $t_a\in (-2M,2M)$ such that $f_\alpha ( i(a) \cdot (t_a,0 , \dots , 0))=c$.  In particular since 
\begin{equation*} 
 X_1f_\alpha ( i(a) \cdot (t,0,\dots , 0) )\leq 0 \quad \forall \, t<t_a \, \qquad \mbox{ and } \qquad   X_1f_\alpha ( i(a) \cdot (t_a,0,\dots , 0)) < 0
\end{equation*}  
we have that $t_a$ is unique and $ \{ t\in (-2M,2M) \, :\, f_\alpha ( i(a) \cdot (t,0,\dots , 0))>c \}=(-2M,t_a)$. 

In the similar way in Lemma 4.8 in \cite{biblioSCV}, we define $\hat \phi _\alpha (i(a)):=(\phi _\alpha (a),0,\dots , 0)$ with $\phi _\alpha (a):=t_a$ and 
\begin{equation}\label{90cit}
\E_\alpha =\E_{\alpha , c}:= \{ (i(a) \cdot (t,0,\dots , 0)) \in \hat{ \mathcal O} \cdot \V  \, :\, f_\alpha (i(a)\cdot (t,0,\dots , 0))>c\}.
\end{equation} Then $\phi _\alpha :\mathcal O'  \to[-2M, 2M]$ and
\begin{equation}\label{91cit}
\E_\alpha \cap \mathcal A'' =\E_{\hat \phi _\alpha } \cap  \mathcal A '',
\end{equation}
where $ \mathcal A'' :=\{p\in \hat{ \mathcal O'} \cdot \V \, :\, p_1\in    [-2M, 2M]  \}$. Moreover recalling that $c\in (0,1)$, from $\eqref{cit76}$ and $\eqref{cit79}$ we get
\begin{equation*}
S_\alpha = \graph{\hat \phi _\alpha} =\partial \E_\alpha \cap ( \hat{ \mathcal O'} \cdot \V)
\end{equation*}
and  by Theorem \ref{teo4.1} $\hat \phi _\alpha $  is uniformly intrinsic differentiable in $\hat{ \mathcal O'} $. 

Let $\tilde f _\alpha :\R^{m+n} \to \R$ defined by $$ \tilde f_\alpha  (t,a ):=f_\alpha  (i(a)\cdot (t,0,\dots , 0)).$$ By $f_\alpha \in \C^\infty (\G , \R)$ and \eqref{cit85}, $ \tilde f_\alpha$ is $\C^\infty$ map such that $\frac{ \partial \tilde f_\alpha}{\partial p_1}(p) = X_1 f_\alpha (p) \ne 0$ for all $p\in S_\alpha $. Hence we apply the classical Implicit Function Theorem to $\tilde f _\alpha$ and we conclude that also $\hat \phi _\alpha$ is $\C^\infty $.

$\mathbf{Step \, 4.}$ Firstly we prove that $(\hat \phi _\alpha )_\alpha $, up to subsequence, converges uniformly on $\hat{ \mathcal O'}$, i.e. the condition (1) of the thesis follows. 

In particular, we would like to show that there is a constant $L>0$ dependent on $ \|\nabla^{\hat \phi} \hat\phi \|_{\mathcal L^\infty (\hat{ \mathcal O'} , \R^{m-1})} $, $\B^{(1)}, \dots ,\B^{(n)}$ such that for all $\alpha >0$ sufficiently small and all $A\in \mathcal O'$, it holds 
\begin{equation}\label{tesi1}
|\phi _\alpha (a)-\phi (a)|\leq L\alpha .
\end{equation}
Therefore $(\phi _\alpha )_\alpha $, up to subsequence, converges uniformly on $\mathcal O'$. This is a direct consequence of Theorem \ref{teo11montivittone}. Here $\nu ^\perp (p) =\mathbf{P}_\W(p)$, $\nu  (p) =\mathbf{P}_\V(p)$ and $E:=\E _{\hat \phi _\alpha } $ is the intrinsic subgraph of $\hat \phi _\alpha$. Because $\hat \phi _\alpha$ is locally intrinsic Lipschitz,  $\E _{\hat \phi _\alpha } $ is also a finite $\G$-perimeter set (see Theorem \ref{Theorem 4.2.9fms}).

For $\beta$ fixed as in Theorem  \ref{teo11montivittone} there is $L=L(\beta , \B_M)>0$, where $\mathcal{B}_{M} = \max \{ b_{ij}^{(s)} \, | \,  i,j=1,\dots , m \, , s=1,\dots , n \}$ such that
\begin{equation*}
\mathcal U((-L,0,\dots,0),1) \subset \{ q\, :\, \| \mathbf{P}_\W q\|<-\beta \mathbf{P}_\V q \} \quad \mbox{and} \quad \mathcal U((L,0,\dots,0),1) \subset \{ q\, :\, \| \mathbf{P}_\W q\|<\beta \mathbf{P}_\V q  \}
\end{equation*}
Let $p_a:= i(a) \cdot \hat \phi (i(a))$, $p'_a:= i(a) \cdot (\phi (a)-\alpha L ,0,\dots , 0)$ and $p''_a:= i(a) \cdot (\phi (a)+\alpha L ,0,\dots , 0)$ with $a \in \mathcal O' $ and $\alpha \in (0,1]$. Notice that  if we put  $r_0:= M+L+\max_{(x,y)\in \mathcal O' }|x|_{\R^{m-1}}+ \max_{(x,y)\in \mathcal O' }|y|_{\R^n}^{1/2} +\frac1 2 \B_M^{1/2} \max_{(x,y)\in \mathcal O' }|x|_{\R^{m-1}}^{1/2} (M+L)^{1/2}< +\infty $ then\begin{equation*}
\| p'_a\|  \leq r_0,\quad \| p''_a\| \leq r_0,
\end{equation*}
for all $a\in  \mathcal O'$. Moreover by standard considerations (see \cite{biblioGROMOV}) we know that for every $q\in \mathcal U(0,r_0)$ there exists $c(r_0)>0$ such that
 \begin{equation*}
\mathcal U(0,r)\cdot q \subset \mathcal U(q,c(r_0)\sqrt r ), \quad \forall \, r \in (0,1).
\end{equation*} 
So if we consider $a\in  \mathcal O'$ and $\alpha \in (0,1]$, we conclude that
\begin{equation*}
\begin{aligned}
\mathcal U(0,\alpha )\cdot p'_a & \subset \mathcal U(p'_a,c(r_0)\sqrt \alpha  ) \subset  \{ q\, :\, \| \mathbf{P}_\W (p_a ^{-1} Q)\|<-\beta \|\mathbf{P}_\V (p_a ^{-1} q) \| \} \subset \E _{\hat \phi _\alpha }\\
\mathcal U(0,\alpha )\cdot p''_a& \subset \mathcal U(p''_a,c(r_0)\sqrt \alpha  ) \subset  \{ q\, :\, \| \mathbf{P}_\W (p_a ^{-1} Q)\|<\beta \|\mathbf{P}_\V (p_a^{-1} q)\|  \} \subset \G - \E _{\hat \phi _\alpha }
\end{aligned} 
\end{equation*} 
 More precisely, by definition of $f_\alpha $ we have
\begin{equation*}
f_\alpha (p'_a)=1,  \quad \quad f_\alpha \left( i(a) \cdot \hat \phi _\alpha (i(a)) \right)=c,  \quad  \quad  f_\alpha (p''_a)=0 
\end{equation*} 
and so by $\eqref{90cit}$, $\eqref{91cit}$ and $\eqref{cit85}$ we deduce
\begin{equation*}
\phi (a)-\alpha L \leq \phi _\alpha (a) \leq \phi (a)+\alpha L\quad  \mbox{ for all } \alpha \in (0,1], \, a\in   \mathcal O'.
\end{equation*} 
Hence \eqref{tesi1} holds and consequently the condition (1) of the thesis is true. \\
$\mathbf{Step \, 5.}$ We prove that $(\hat \phi_\alpha )_{\alpha}$ has a subsequence $(\hat \phi _k)_k$ such that $|\nabla^{\hat \phi _k} \hat\phi _k |_{\R^{m-1}} \leq \| \nabla^{\hat \phi} \hat\phi \|_{\mathcal L^\infty (\hat{ \mathcal O'} ,\R^{m-1})}$ for each $k\in \N$ on $\mathcal O '$, i.e. the condition $(2)$ of the thesis holds. 

Let $\hat \nabla _\G f_\alpha :=(X_2f_\alpha , \dots ,X_mf_\alpha )$ and $\hat \nu_{\E } =(\nu^{(2)}_{\E},\dots , \nu^{(m)}_{\E })$. Arguing as in Step 2 and by \eqref{DPHI2}, it follows
\begin{equation*}
\begin{aligned}
 |  \nabla^{\hat \phi _\alpha} \hat\phi _\alpha  (p)|_{\R^{m-1}} &= \frac{|\hat \nabla _\G f_\alpha (p)|_{\R^{m-1}}}{|X_1f_\alpha (p)|} \leq \frac{1}{|X_1f_\alpha (p)|} \int _{ \mathcal U(0,\alpha )\cdot p} |\hat \nu_{\E } (q)| |\rho _\alpha (p\cdot q^{-1})| \, d |\partial \E |(q)\\
 & \leq I_\alpha (p) \frac{ \| \nabla^{\hat \phi} \hat\phi \|_{\mathcal L^\infty (\hat{ \mathcal O'},\R^{m-1})}  }{|X_1f_\alpha (p)| \sqrt{1+ \| \nabla^{\hat \phi} \hat\phi \|^2_{\mathcal L^\infty (\hat{ \mathcal O'} ,\R^{m-1})} }} \\
\end{aligned}
\end{equation*}
and consequently by $\eqref{cit87}$
\begin{equation}\label{cit93}
 |  \nabla^{\hat \phi _\alpha} \hat\phi _\alpha  (p)|_{\R^{m-1}}  \leq \| \nabla^{\hat \phi} \hat\phi \|_{ \mathcal L^\infty (\hat{ \mathcal O'} ,\R^{m-1})},
\end{equation}
for all $p\in \hat{ \mathcal O'}$, as desired. \\
$\mathbf{Step \, 6.}$ Now we show that $(\hat \phi_\alpha )_\alpha $, up to a subsequence, satisfying the condition $(3)$ of the thesis, i.e. there exists a sequence $(\alpha _h)_h \subset (0,+\infty )$ such that let $\hat \phi _h \equiv \hat \phi _{\alpha _h}$ we have
\begin{equation}\label{citt97}
\nabla^{\hat \phi _h} \hat\phi _h  \to \nabla^{\hat \phi} \hat\phi  \quad \mbox{ a.e. in } \hat{ \mathcal O'}.
\end{equation}
It is sufficient to show that there is a positive sequence $(\alpha _h)_h$ converging to $0$ such that there exists
\begin{equation}\label{citt98}
\lim_{h\to \infty } \int_{ \mathcal O'} \sqrt{ 1+|\nabla^{\hat \phi _h} \hat\phi _h|_{\R^{m-1}} ^2} \, d \mathcal{L}^{m+n-1} = \int_{ \mathcal O'} \sqrt{ 1+|\nabla^{\hat \phi} \hat\phi |_{\R^{m-1}} ^2} \, d  \mathcal{L}^{m+n-1}
\end{equation}
 In fact, up a subsequence, using $\eqref{cit93}$ and Proposition $\ref{theo4.7}$ we can assume that the sequence in $\eqref{citt98}$ also satisfies 
\begin{equation*}
\nabla^{\hat \phi _h} \hat\phi _h  \to \nabla^{\hat \phi} \hat\phi  \quad \mbox{ weakly in  }\, \mathcal{L}^1(  \hat{ \mathcal O'}, \R^{m-1} ).
\end{equation*} 
Consequently by Theorem \ref{teoVisintin}
\begin{equation*}
\nabla^{\hat \phi _h} \hat\phi _h  \to \nabla^{\hat \phi} \hat\phi   \quad \mbox{ strongly in  }\, \mathcal{L}^1(  \hat{ \mathcal O'}, \R^{m-1} )
\end{equation*} 
and so, up a subsequence, $\eqref{citt97}$ holds.

In particular in order to show $\eqref{citt98}$ we need to prove that there exist $\hat c \in (0,1)$ and a positive sequence $(\alpha _h)_h$ converging to $0$ satisfying 
\begin{equation}\label{cit101}
\exists \, \lim _{h\to \infty  } |\partial \E_{\alpha _h, \hat c  }|_\G ( \mathcal A) =  |\partial \E|_\G (\mathcal A) 
\end{equation} 
where $\mathcal A$ is defined as in Step 2. Since the semicontinuity of $\G$-perimeter measure and Step 3 we obtain
\begin{equation}\label{cit106} 
|\partial \E|_\G (\mathcal A ) \leq \liminf _{\alpha \to 0^+} |\partial \E_{\alpha , c }|_\G (\mathcal A )
\end{equation}
for every $c\in (0,1)$. More precisely by the last inequality and the coarea formula we get
\begin{equation*}
\begin{aligned}
|\partial \E|_\G (\mathcal A) & \leq \int _0^1 \liminf _{\alpha \to 0^+} |\partial \E_{\alpha , c }|_\G (\mathcal A )\, dc \leq  \liminf _{\alpha \to 0^+} \int _0^1 |\partial \E_{\alpha , c }|_\G (\mathcal A )\, dc \\ & \quad  = \liminf _{\alpha \to 0^+} \int _{\mathcal A } | \nabla _\G f_\alpha |\, d \mathcal{L}^{m+n} =: I(\hat{ \mathcal O'} , c ).
\end{aligned} 
\end{equation*} 
Now using Step 2 we know that for every $\hat{ \mathcal O}_0'$ such that $\hat{ \mathcal O'}\Subset  \hat{ \mathcal O} _0' $ open and bounded there exists $(\alpha _h)_h \subset (0,+ \infty )$ which converges to $0$ and $ \bar h (\hat{ \mathcal O}_0')>0$ such that for all $h\leq \bar h (\hat{ \mathcal O}_0')$ we have
\begin{equation*}
 \int_{\mathcal A } | \nabla _\G f_{\alpha _h}| \, d \mathcal{L}^{m+n} \leq |\partial \E |_\G ( \mathcal  A_0)
\end{equation*}
where $\mathcal  A, \mathcal  A_0$ are defined as in Step 2. Consequently
\begin{equation}\label{cit109}
  I(\hat{ \mathcal O'}, c ) \leq |\partial \E |_\G (\mathcal A_0)
\end{equation}
for all $c\in (0,1)$ and $\hat{ \mathcal O}_0'$ such that $\hat{ \mathcal O'} \Subset \hat{ \mathcal O}_0' $ open and bounded. Moreover by $|\partial \E |_\G$ is a Radon measure then by standard approximation argument we can rewrite $\eqref{cit109}$ with $\hat{ \mathcal O'}$ instead of $\hat{ \mathcal O}_0'$. In particular using $\eqref{cit106}$ we have that a.e. $c\in (0,1)$
\begin{equation*}
 \liminf _{\alpha \to 0^+} |\partial \E_{\alpha , c }|_\G (\mathcal A )=  |\partial \E|_\G (\mathcal A ).
\end{equation*} 
Hence there exists $\hat c \in (0,1)$ and a positive sequence $(\alpha _h)_h$ converging to $0$ such that $\eqref{cit101}$ follows.

Finally because $\| \phi \|_{\mathcal L^\infty (\mathcal O ' , \R ) } = M$ we get
\begin{equation*}
\partial \E \cap \mathcal A _1  =  \partial \E \cap  \mathcal A _2  = \emptyset 
\end{equation*} 
where $ \mathcal  A_1:=\{p \in \hat{ \mathcal O'} \cdot \V \, :\, p_1 \leq -2M \}$ and $\mathcal A_2:=\{p\in \hat{ \mathcal O'} \cdot \V \, :\, p_1 \geq 2M \}$.

Hence since Proposition $\ref{citteo16OK}$ and well-know $\G$-perimeter properties it follows that
\begin{equation*}
\begin{aligned}
\frac{1}{C(\W,\V)}   \int_{ \mathcal O'}  \sqrt{1+| \nabla^{\hat \phi} \hat\phi  |_{\R^{m-1}} ^2} \, d\mathcal{L}^{m+n-1} & =|\partial \E|_\G ( \hat{ \mathcal O'} \cdot \V) \\
& = |\partial \E|_\G (\mathcal A _1 )+ |\partial \E|_\G (\mathcal A  )+ |\partial \E| ( \mathcal A_2 )\\
& = |\partial \E|_\G ( \partial \E \cap \mathcal A_1) + |\partial \E|_\G (\mathcal A ) + |\partial \E|_\G ( \partial \E \cap \mathcal A_2 )\\
&= |\partial \E|_\G ( \mathcal A ).
\end{aligned} 
\end{equation*} 
In the same way, using also $\eqref{cit76}$, $\eqref{cit79}$ and $\eqref{90cit}$
\begin{equation*}
 |\partial \E_{\alpha _h , c} |_\G ( \mathcal A )= |\partial \E_{\hat \phi _h }|_\G (  \mathcal A )
\end{equation*} 
and 
\begin{equation*}
 |\partial \E_{\hat \phi _h }|_\G (  \hat{ \mathcal O'} \cdot \V)= \frac{1}{C(\W,\V)}  \int_{ \mathcal O'}  \sqrt{1+|\nabla^{\hat \phi_h} \hat\phi _h |_{\R^{m-1}} ^2} \, \, d \mathcal{L}^{m+n-1} 
 \end{equation*} 
with $\hat \phi _h = \hat \phi _{\alpha _h }$. Now by the last three equalities and $\eqref{cit101}$, $\eqref{citt98}$ holds. Consequently $\eqref{citt97}$ is true and the proof of theorem is complete.
\end{proof}

\section{Intrinsic Lipschitz Function and Distributional Solution}

The main results in this section are Theorem \ref{lemma5.6} and Theorem \ref{lemma5.4bcsc}. We show that a map $\hat \phi =(\phi , 0 \dots , 0): \hat{ \mathcal O} \to \V$ is a locally intrinsic Lipschitz function if and only if $\phi$ is  locally $1/2 $-H\"older continuous and it is also distributional solution of system $$\left(D^{\phi}_2\phi,\dots, D^{\phi}_m\phi\right)  =w \quad \mbox{in } \mathcal O.$$ 
More precisely, we generalize Theorem 1.1 on \cite{biblio17} valid in Heisenberg groups. 

\begin{theorem}\label{lemma5.6}
Let $\G := (\R^{m+n}, \cdot  , \delta_\lambda )$ be a Carnot group of step $2$ and $\V$, $\W$ the complementary subgroups defined in \eqref{5.2.0}. Let $\hat \phi :\hat{ \mathcal O} \to \V$ be a locally intrinsic Lipschitz function, where $\hat{ \mathcal O}$ is an open subset of $\W$ and $\phi :\mathcal O  \to \R$ is the map associated to $\hat \phi$ as in $\eqref{phipsi}$. 

Then if $(D^{\phi}_2\phi,\dots, D^{\phi}_m\phi ) \in \mathcal{L}_{loc}^\infty (\mathcal O ,\R^{m-1})$, we have that $\phi$ is a distributional solution of \\ $(D^{\phi}_2\phi,\dots, D^{\phi}_m\phi )  =w$ for a suitable $w\in \mathcal{L}^\infty _{loc} (\mathcal O , \R^{m-1})$ such that $w(a)=(D^{\phi}_2\phi (a),\dots, D^{\phi}_m\phi (a) )$ $\mathcal{L}^{m+n-1} $-a.e. in $\mathcal O$.
\end{theorem}

\begin{proof}
By Theorem $\ref{cmpssteo17}$ we know that there exist $(\phi _h)_{h\in \N} \subset \C^\infty (\mathcal O ,\R)$ such that for all $\mathcal O' \Subset \mathcal O$ we have
\begin{enumerate}
\item  $ \phi _h$ uniformly converges to $ \phi $ in $\mathcal O' $
\item $|(D^{ \phi _h}_2 \phi _h(a), \dots , D^{ \phi _h}_m \phi _h(a)) |_{\R^{m-1}} \leq \| (D^{ \phi }_2 \phi (a), \dots , D^{ \phi }_m \phi (a))  \|_{\mathcal{L}^\infty (\mathcal O',\R^{m-1})}$ $\forall a\in \mathcal O '$, $h\in \N$
\item for each $j=2,\dots , m$, we have $D^{ \phi _h}_j \phi _h (a)\to D^{ \phi }_j \phi (a)$ a.e. $a\in \mathcal O'$
\end{enumerate}

 Let $w_h =(w_{2,h}, \dots , w_{m,h}):=(D_2^{\phi _h} \phi _h, \dots ,D_m^{\phi _h} \phi _h)$. Observe that for each $h\in \N$ and $\zeta \in \C^\infty _c (\mathcal O ,\R)$
\[
 \int _\mathcal O \phi_h \left( \,X_j \zeta +\phi _h \sum_{s=1}^n b_ {j1 }^{(s)} Y_s \zeta \right)  \, d\mathcal{L}^{m+n-1}  = - \int _\mathcal O w_{j,h} \zeta \, d \mathcal{L}^{m+n-1}
\]
for all $j=2,\dots , m$. Getting to the limit for $h\to \infty $ we have
\[
\int _\mathcal O \phi \left( \,X_j \zeta +\phi  \sum_{s=1}^n b_ {j1 }^{(s)} Y_s \zeta \right)  \, d\mathcal{L}^{m+n-1} = - \int _\mathcal O w_{j} \zeta \, d \mathcal{L}^{m+n-1}.
\]
Hence $\phi $ is a distributional solution of $(D^{\phi}_2\phi,\dots, D^{\phi}_m\phi ) =w$ in $\mathcal O$.
\end{proof}

\begin{theorem}\label{lemma5.4bcsc} 
Let $\G := (\R^{m+n}, \cdot  , \delta_\lambda )$ be a Carnot group of step $2$ and $\V$, $\W$ the complementary subgroups defined in \eqref{5.2.0}. Let $\hat \phi :\hat{\mathcal O} \to \V $ be a continuous map  where $\hat{\mathcal O} $ is open in $\W$ and $\phi :\mathcal O  \to \R$ is the map associated to $\hat \phi$ as in $\eqref{phipsi}$.  We also assume that 
\begin{enumerate}
\item $\phi$ is a continuous distributional solution of $(D^{\phi}_2\phi,\dots, D^{\phi}_m\phi )  =w$ in $\mathcal O $ with $w\in \mathcal{L}^\infty _{loc}(\mathcal O ,\R^{m-1})$
\item $\phi$ is locally $1/2 $-H\"older continuous along the vertical components with H\"older's constant $C_h>0$
\end{enumerate}
Then $\hat \phi $ is locally intrinsic Lipschitz function.
\end{theorem} 
The proof of Theorem \ref{lemma5.4bcsc}  relies on a preliminary result about the distributional solutions of the problem $(D^{\phi}_2\phi,\dots, D^{\phi}_m\phi )  =w$ in $\mathcal O $.

It is convenient to introduce the following notation: for $j=2,\dots ,m $ and \\
 $\hat x_j=(\hat x_2, \dots , \hat x_{j-1} ,\hat x_{j+1}, \dots , \hat x_{m}) \in \R^{m-2}$ fixed, we denote by $(t,\hat x_j) :=(\hat x_2, \dots , \hat x_{j-1}, t ,\hat x_{j+1}, \dots , \hat x_{m} )$


\begin{lem}\label{lemma5.1bcsc} 
Under the same assumptions of Theorem $\ref{lemma5.4bcsc}$, for all $j=2,\dots , m$ there is $C=C(\| w_j\|_{\mathcal{L}^\infty (\mathcal O, \R )}, C_h, b_ {j1 }^{(1)},\dots , b_ {j1 }^{(n)})>0$ such that $\phi$ is a C-Lipschitz map along any characteristic line $\gamma _{j}=(\gamma _{j1}, \dots , \gamma _{jn}):[-T, T ]\to \R^n$ satisfying
\begin{equation*}\label{carattere}
\dot \gamma _{js} (t)= b_ {j1 }^{(s)} \phi (t,\hat x_j, \gamma _j(t))+ \frac{1}{2}\sum_{l=2}^{m}b^{(s)}_{jl}(\hat x_j )_l \quad \mbox{ for all } s=1,\dots, n
\end{equation*}
with $t\in [-T , T ]$ and $\hat x_j\in \R^{m-2}$ fixed. 
\end{lem} 

\begin{proof}
Fix $j=2,\dots ,m$ and $\hat x_j \in \R^{m-2}$. 

Let $t_1,t_2 \in (-T,T)$ with $t_1<t_2$ and let $\epsilon >0$. In the similar way in \cite{biblioDAF} and in \cite{biblio17}, Lemma 5.1 we consider the functions $\xi _1(x,y_1), \dots, \xi_n(x,y_n)$ and $h(x)$ defined for small $\delta >0$ by
\begin{equation*}
\xi _s (x,y_s)= \left\{
\begin{array}{l}
\quad \quad 0\hspace{3,2 cm} -T<x<T, \, \,  -r_s <y_s\leq \gamma _s(x) -\epsilon -\delta 
\\
\frac{1}{\delta }(y_s -\gamma _s(x)+\epsilon +\delta ) \hspace{0,65 cm} -T<x<T, \, \,  \gamma_s (x) -\epsilon -\delta<y_s\leq \gamma_s (x) -\epsilon  
\\
\quad \quad 1 \hspace{3,2 cm} -T<x<T, \, \,   \gamma _s(x) -\epsilon <y_s\leq \gamma _s(x) 
\\
\frac{1}{\delta }(-y_s +\gamma _s(x) +\delta ) \hspace{1 cm} -T<x<T, \, \,  \gamma_s (x) <y_s\leq \gamma _s(x) +\delta  
\\
\quad \quad 0 \hspace{3,2 cm} -T<x<T, \, \,  \gamma _s(x) +\delta <y_s< r_s 
\end{array} 
\right.
\end{equation*}
for $s=1,\dots ,n$ and 
\begin{equation*}
h(x) = \left\{
\begin{array}{l}
\quad \quad 0\hspace{2,5 cm}  -T< x \leq t_1 -\delta 
\\
\frac{1}{\delta }(x-t_1 +\delta ) \hspace{1,3 cm} t_1 -\delta <x\leq t_1
\\
\quad \quad 1 \hspace{2,55 cm} t_1 <x\leq t_2
\\
\frac{1}{\delta }(-x+t_2 +\delta ) \hspace{1,05 cm} t_2<x \leq t_2+\delta 
\\
\quad \quad 0 \hspace{2,6 cm}t_2+\delta <x<T
\end{array} 
\right.
\end{equation*}

If $\phi$ is distributional solution of $(D_2^\phi \phi , \dots , D_m^\phi \phi )=w$ in $\mathcal O $, for all test function $\vp$ we have
\begin{equation*}
\begin{aligned}
\int_{-T}^T \int_{-r_1}^{r_1} \dots \int_{-r_n}^{r_n} (\phi (x,\hat x_j,y)\vp_x(x,y) & +f_1(\phi (x,\hat x_j,y)) \vp_{y_1}(x,y)+\dots +f_n(\phi (x,y)) \vp_{y_n}(x,y)\\
& +w_j(x,\hat x_j,y)\vp (x,y))\, dx dy =0
\end{aligned}
\end{equation*}
where 
\begin{equation}\label{f1f2}
\begin{aligned}
f_s(\phi (x_j,\hat x_j,y))& = \frac{1}{2} \biggl( b_{j1}^{(s)}\phi ^2(x_j,\hat x_j,y)+\phi (x_j,\hat x_j,y)\sum_{l=2}^{m}b_{jl}^{(s)}(\hat x_j)_l   \biggl)
\end{aligned}
\end{equation}
for all $(x_j,\hat x_j,y)\in \mathcal O $ and $s=1,\dots ,n$.

We choose the test function $\vp (x,y_1,\dots ,y_n )= \xi_1 (x,y_1) \dots \xi_n (x,y_n) h(x)$ and compute the limit $\delta \to 0^+$:
\begin{equation*}
\begin{aligned}
\lim_{\delta \to 0^+} \int_{-T}^T \int_{-r_1}^{r_1} \dots \int_{-r_n}^{r_n} (\phi \vp_x+f_1(\phi) ( \xi _1)_{y_1} \xi _2 \dots \xi_n h+\dots +f_n(\phi ) \xi _1 \dots (\xi _n)_{y_n}h+w_j\xi _1 \dots \xi _nh)\, dx dy
\end{aligned}
\end{equation*}

For simplicity we consider the case $n=2$. We compute the following limit:
\begin{equation*}
\begin{aligned}
\lim_{\delta \to 0^+} & \int_{-T}^T \int_{-r_1}^{r_1} \int_{-r_2}^{r_2} \biggl(\phi \biggl((\xi_1)_x\xi _2 h+\xi_1(\xi _2)_x h+\xi_1\xi _2 h_x\biggl)+f_1(\phi) ( \xi _1)_{y_1} \xi _2 h \\
& \quad +f_2(\phi ) \xi _1(\xi _2)_{y_2}h + w_j\xi _1 \xi _2h\biggl)\, dx dy_1dy_2=: \lim_{\delta \to 0^+} I_1+I_2+I_3+I_4+I_5+I_6.
\end{aligned}
\end{equation*}
Then
\begin{equation*}
\begin{aligned}
\lim_{\delta \to 0^+}  I_1& = \lim_{\delta \to 0^+} \int_{-T}^T \int_{-r_1}^{r_1} \int_{-r_2}^{r_2} \phi (\xi_1)_x\xi _2 h\, dx dy_1dy_2\\
& =\int_a^b \int_{\gamma _2(x)-\epsilon }^{\gamma _2(x)} \dot \gamma _1(x) \Bigl(\phi (x,\hat x_j,\gamma _1(x), y_2) - \phi (x,\hat x_j,\gamma_1 (x)-\epsilon , y_2)\Bigl)\, dxdy_2\\
\lim_{\delta \to 0^+}  I_2& = \lim_{\delta \to 0^+} \int_{-T}^T \int_{-r_1}^{r_1} \int_{-r_2}^{r_2} \phi \xi_1(\xi _2)_x h\, dx dy_1dy_2\\
& =\int_{t_1}^{t_2} \int_{\gamma _1(x)-\epsilon }^{\gamma _1(x)} \dot \gamma _2(x) \Bigl(\phi (x,\hat x_j,y_1, \gamma _2(x)) - \phi (x,\hat x_j,y_1, \gamma_2 (x)-\epsilon )\Bigl)\, dxdy_1\\
\end{aligned}
\end{equation*}
\begin{equation*}
\begin{aligned}
\lim_{\delta \to 0^+}  I_3& = \lim_{\delta \to 0^+} \int_{-T}^T \int_{-r_1}^{r_1} \int_{-r_2}^{r_2} \phi \xi_1\xi _2 h_x\, dx dy_1dy_2\\
& = \int_{\gamma _1(t_1)-\epsilon }^{\gamma _1(t_1)}  \int_{\gamma _2(t_1)-\epsilon }^{\gamma _2(t_1)}  \phi (t_1,\hat x_j,y_1, y_2) \, dy_1dy_2 -\int_{\gamma _1(t_2)-\epsilon }^{\gamma _1(t_2)}  \int_{\gamma _2(t_2)-\epsilon }^{\gamma _2(t_2)}  \phi (t_2,\hat x_j,y_1, y_2) \, dy_1dy_2\\
\lim_{\delta \to 0^+}  I_4& = \lim_{\delta \to 0^+} \int_{-T}^T \int_{-r_1}^{r_1} \int_{-r_2}^{r_2} f_1(\phi) ( \xi _1)_{y_1} \xi _2 h\, dx dy_1dy_2\\
& = \int_{t_1}^{t_2} \int_{\gamma _2(x)-\epsilon }^{\gamma _2(x)}  \Bigl( f_1(\phi (x,\hat x_j,\gamma _1 (x)- \epsilon ,y_2))-f_1(\phi (x,\hat x_j,\gamma_1 (x) ,y_2))  \Bigl)\, dxdy_2\\
\lim_{\delta \to 0^+}  I_5& = \lim_{\delta \to 0^+} \int_{-T}^T \int_{-r_1}^{r_1} \int_{-r_2}^{r_2} f_2(\phi)\xi _1 ( \xi _2)_{y_2}  h\, dx dy_1dy_2\\
& = \int_{t_1}^{t_2} \int_{\gamma _1(x)-\epsilon }^{\gamma _1(x)}  \Bigl( f_2(\phi (x,y_1, \gamma _2 (x)- \epsilon ))-f_2(\phi (x,y_1, \gamma_2 (x) ))  \Bigl)\, dxdy_1\\
\lim_{\delta \to 0^+}  I_6 &= \lim_{\delta \to 0^+} \int_{-T}^T \int_{-r_1}^{r_1} \int_{-r_2}^{r_2}w_j\xi _1 \xi _2h\, dx dy_1dy_2 = \int_{t_1}^{t_2} \int_{\gamma _1(x)-\epsilon }^{\gamma _1(x)}\int_{\gamma _2(x)-\epsilon }^{\gamma _2(x)}  w_j\, dxdy_1dy_2
\end{aligned}
\end{equation*}

Hence we obtain
\begin{equation}\label{dise01}
\begin{aligned}
& \int_{\gamma _1(t_2)-\epsilon }^{\gamma _1(t_2)}  \int_{\gamma _2(t_2)-\epsilon }^{\gamma _2(t_2)}  \phi (t_2,\hat x_j,y_1, y_2) \, dy_1dy_2 -  \int_{\gamma _1(t_1)-\epsilon }^{\gamma _1(t_1)}  \int_{\gamma _2(t_1)-\epsilon }^{\gamma _2(t_1)}  \phi (t_1,\hat x_j,y_1, y_2) \, dy_1dy_2\\
& \quad - \int_{t_1}^{t_2} \int_{\gamma _1(x)-\epsilon }^{\gamma _1(x)}\int_{\gamma _2(x)-\epsilon }^{\gamma _2(x)}  w_j (x,\hat x_j,y_1, y_2) \, dxdy_1dy_2\\
&  =  \int_{t_1}^{t_2} \int_{\gamma _2(x)-\epsilon }^{\gamma _2(x)}  f_1(\phi (x,\hat x_j,\gamma _1 (x)- \epsilon ,y_2))-f_1(\phi (x,\hat x_j,\gamma_1 (x) ,y_2))\\
& \quad + \dot \gamma _1(x) \Bigl(\phi (x,\hat x_j,\gamma _1(x), y_2) - \phi (x,\hat x_j,\gamma_1 (x)-\epsilon , y_2)\Bigl)\, dx dy_2 \\
& \quad + \int_{t_1}^{t_2} \int_{\gamma _1(x)-\epsilon }^{\gamma _1(x)}   f_2(\phi (x,\hat x_j,y_1, \gamma _2 (x)- \epsilon ))-f_2(\phi (x,\hat x_j,y_1, \gamma_2 (x) )) \\
&\quad + \dot \gamma _2(x) \Bigl(\phi (x,\hat x_j,y_1, \gamma _2(x)) - \phi (x,\hat x_j,y_1, \gamma_2 (x)-\epsilon )\Bigl)\,  dxdy_1 
\end{aligned}
\end{equation}

From \eqref{f1f2} with $n=2$ we deduce that
\begin{equation}\label{lipottobre}
\begin{aligned}
& \int_{\gamma _1(t_2)-\epsilon }^{\gamma _1(t_2)}  \int_{\gamma _2(t_2)-\epsilon }^{\gamma _2(t_2)}  \phi (t_2,\hat x_j,y_1, y_2) \, dy_1dy_2 -  \int_{\gamma _1(t_1)-\epsilon }^{\gamma _1(t_1)}  \int_{\gamma _2(t_1)-\epsilon }^{\gamma _2(t_1)}  \phi (t_1,\hat x_j,y_1, y_2) \, dy_1dy_2\\
& \quad - \int_{t_1}^{t_2} \int_{\gamma _1(x)-\epsilon }^{\gamma _1(x)}\int_{\gamma _2(x)-\epsilon }^{\gamma _2(x)}  w_j (x,\hat x_j,y_1, y_2) \, dxdy_1dy_2\\
 & = \int_{t_1}^{t_2} \int_{\gamma _2(x)-\epsilon }^{\gamma _2(x)}  \frac{1}{2}b_{j1}^{(1)}\biggl( \phi (x,\hat x_j,\gamma _1(x)-\epsilon ,y_2)- \phi (x,\hat x_j,\gamma _1(x) ,y_2)  \biggl) \biggl( \phi (x,\hat x_j,\hat x_j,\gamma _1(x),y_2)+\\
 & + \phi (x,\hat x_j,\hat x_j,\gamma _1(x)-\epsilon ,y_2) -2\phi (x,\hat x_j,\hat x_j,\gamma _1(x),\gamma _2(x))     \biggl) \, dxdy_2 +  \int_{t_1}^{t_2}  \int_{\gamma _1(x)-\epsilon }^{\gamma _1(x)}  \frac{1}{2}b_{j1}^{(2)}\biggl( \phi (x,\hat x_j,y_1,\gamma _2(x)-\epsilon )\\  &  - \phi (x,y_1,\gamma _2(x) )  \biggl) \biggl( \phi (x,y_1,\gamma _2(x)) + \phi (x,y_1, \gamma _2(x)-\epsilon ) -2\phi (x,\hat x_j,\gamma _1(x),\gamma _2(x))     \biggl) \, dxdy_1\\
\end{aligned}
\end{equation}

If $b_{j1}^{(1)}, b_{j1}^{(2)} = 0$ the integral \eqref{lipottobre} is zero. On the other hand if $b_{j1}^{(1)} , b_{j1}^{(2)}\ne 0$, because $\phi$ is locally $1/2 $-H\"older continuous along the vertical components  with H\"older's constant $C_h>0$
\begin{equation}\label{Chholder}
\begin{aligned}
&\frac {|\phi (x,\hat x_j,\gamma _1(x),y_2)- \phi (x,\hat x_j,\gamma _1(x)-\epsilon ,y_2)| }{\sqrt \epsilon}  \leq C_h\\
& \frac {|\phi (x,\hat x_j,\gamma _1(x),y_2) -\phi (x,\hat x_j,\gamma _1(x),\gamma _2(x)) |}{\sqrt \epsilon}  \leq  C_h\\
& \frac { |\phi (x,\hat x_j,\gamma _1(x)-\epsilon ,y_2) -\phi (x,\hat x_j,\gamma _1(x),\gamma _2(x)) |}{\sqrt {2\epsilon}}  \leq  C_h
\end{aligned}
\end{equation}
for $x\in [t_1,t_2]$ and $y_2\in (\gamma _2(x)-\epsilon  , \gamma _2(x))$. Hence dividing by $\epsilon ^2$ in $\eqref{lipottobre}$ and getting to the limit to $\epsilon \to 0$, we obtain
\begin{equation}\label{dise02}
\begin{aligned}
 &\lim_{\epsilon \to 0 }  \frac{1}{\epsilon ^2} \int_{t_1}^{t_2} \int_{\gamma _2(x)-\epsilon }^{\gamma _2(x)}  \frac{1}{2}b_{j1}^{(1)}\biggl( \phi (x,\hat x_j,\gamma _1(x)-\epsilon ,y_2)- \phi (x,\hat x_j,\gamma _1(x) ,y_2)  \biggl) \biggl( \phi (x,\hat x_j,\gamma _1(x),y_2)+\\
 & \quad + \phi (x,\hat x_j,\gamma _1(x)-\epsilon ,y_2) -2\phi (x,\hat x_j,\gamma _1(x),\gamma _2(x))     \biggl) \, dxdy_2 \\
 & \leq |b_{j1}^{(1)}|   \frac{(1+\sqrt 2)C_h^2}{2} (t_2-t_1).
\end{aligned}
\end{equation}
Moreover
\begin{equation}\label{dise03}
\begin{aligned}
 &\lim_{\epsilon \to 0 } \frac{1}{\epsilon ^2} \int_{t_1}^{t_2} \int_{\gamma _1(x)-\epsilon }^{\gamma _1(x)}  \frac{1}{2}b_{j1}^{(2)}\biggl( \phi (x,\hat x_j,y_1,\gamma _2(x)-\epsilon )- \phi (x,\hat x_j,y_1,\gamma _2(x) )  \biggl) \biggl( \phi (x,\hat x_j,y_1,\gamma _2(x))+\\
 & \quad + \phi (x,\hat x_j,y_1, \gamma _2(x)-\epsilon ) -2\phi (x,\hat x_j,\gamma _1(x),\gamma _2(x))     \biggl) \, dxdy_1\\
 & \leq |b_{j1}^{(2)} |  \frac{(1+\sqrt 2)C_h^2}{2} (t_2-t_1),
\end{aligned}
\end{equation}
where in the last inequality we have used the fact 
\begin{equation}\label{Chholder1.0}
\begin{aligned}
&\frac {|\phi (x,\hat x_j,y_1,\gamma _2(x)-\epsilon)- \phi (x,\hat x_j,y_1,\gamma _2(x)) |}{\sqrt \epsilon}  \leq C_h\\
& \frac {|\phi (x,\hat x_j,y_1,\gamma _2(x)) -\phi (x,\hat x_j,\gamma _1(x),\gamma _2(x)) |}{\sqrt \epsilon}  \leq  C_h\\
& \frac { |\phi (x,\hat x_j,y_1,\gamma _2(x)-\epsilon ) -\phi (x,\hat x_j,\gamma _1(x),\gamma _2(x)) |}{\sqrt {2\epsilon}}  \leq  C_h
\end{aligned}
\end{equation}
for $x\in [t_1,t_2]$ and $y_1\in (\gamma _1(x)-\epsilon , \gamma _1(x))$.

Now putting together $\eqref{dise01}$, $\eqref{dise02}$ and $\eqref{dise03}$, dividing by $\epsilon ^2$ and getting to the limit to $\epsilon \to 0$, we obtain 
\begin{equation}\label{pervederelipsch}
\phi (t_2,\hat x_j , \gamma _j(t_2)) -\phi (t_1,\hat x_j , \gamma _j(t_1 )) \leq  \biggl(\|w_j\|_{\mathcal{L}^\infty (\mathcal O ,\R)} + \frac{(1+\sqrt 2)C_h^2}{2} \Big(|b_{j1}^{(1)} |+ |b_{j1}^{(2)}|\Big) \biggl) (t_2-t_1)
\end{equation}
for $t_1,t_2 \in  (-T,T)$ with $t_1<t_2$. 

In the similar way, using again \eqref{Chholder} and \eqref{Chholder1.0} we have
\begin{equation}\label{pervederelipsch2}
\phi (t_2,\hat x_j , \gamma _j(t_2)) - \phi (t_1,\hat x_j , \gamma _j(t_1)) \geq - \biggl(\|w_j\|_{\mathcal{L}^\infty (\mathcal O,\R )} + \frac{(1+\sqrt 2)C_h^2}{2} \Big(|b_{j1}^{(1)} |+ |b_{j1}^{(2)}|\Big) \biggl) (t_2-t_1)
\end{equation}
for $t_1,t_2 \in  (-T,T)$ with $t_1<t_2$. 

Hence combining \eqref{pervederelipsch} and \eqref{pervederelipsch2} we get 
\begin{equation*}
|\phi (t_2,\hat x_j , \gamma _j(t_2 )) - \phi (t_1,\hat x_j , \gamma _j(t_1)) | \leq  \biggl(\|w_j\|_{\mathcal{L}^\infty (\mathcal O,\R )} + \frac{(1+\sqrt 2)C_h^2}{2} \Big(|b_{j1}^{(1)} |+ |b_{j1}^{(2)}|\Big) \biggl) (t_2-t_1),
\end{equation*}
for $t_1,t_2 \in  (-T , T)$ with $t_1<t_2$, i.e. $\phi$ is a Lipschitz map along any characteristic line $\gamma _{j}$ as desired.

\end{proof}

\begin{exa} 
%
%
We consider free step 2 group $\mathbb{F}_{m,2}$. As we said in Remark \ref{rem3.1free}, the composition law $\eqref{1}$ is given by  $\mathcal B^{(s)}\equiv \mathcal{B}^{(l,h)}$ where $1\leq h< l\leq m$ and $\mathcal{B}^{(l,h)}$ has entries $-1$ in  position $(l,h)$, $1$ in  position $(h,l)$ and $0$ everywhere else.

For each $j=2,\dots ,m$ and $\hat x_j\in \R^{m-2}$ fixed the characteristic line $\gamma _{j}=(\gamma _{j1}, \dots , \gamma _{jn}):[-\delta , \delta ]\to \R^n$ has the following form:
 \begin{equation*}
\begin{aligned}
\dot \gamma _{js_1} (t)&= b_ {j1 }^{(s_1)} \phi (t,\hat x_j, \gamma _j(t))= - \phi (t,\hat x_j, \gamma _j(t)) \\
\dot \gamma _{js_2} (t)&= \frac{1}{2}b^{(s_2)}_{jh_2}(\hat x_j )_{h_2}= - \frac{1}{2}(\hat x_j )_{h_2}\\
\dot \gamma _{js_3} (t)&= 0
\end{aligned}
\end{equation*}
where  $\mathcal B^{(s_i)}\equiv \mathcal{B}^{(l_i,h_i)}$ is such that
\begin{itemize}
\item $s_1$ is the unique index $(l_1,h_1)=(j,1)$.
\item $s_2$ is such that $h_2<l_2=j$ and $h_2\ne 1$.
\item $s_3$ represents the other cases. 
\end{itemize}

\end{exa}

Now we are able to show the proof of Theorem \ref{lemma5.4bcsc}. 
\begin{proof}
Fix $b\in \mathcal{O}$. According to Remark \ref{lip0}, we would like to prove that there exists $\hat C>0$ such that
\begin{equation}\label{lipsch}
|\phi (a')-\phi (a)| \leq \hat C \|\hat \phi(i(a))^{-1}i(a)^{-1}i(a')\hat \phi(i(a))  \|  \qquad \mbox{for all } \, a, a' \in  \mathcal U(b,\delta )
\end{equation}
with $\delta >0$. Let $a=( x,y),a'= ( x',y')$ points of $\mathcal{O}$ be sufficiently close to $b$, and let $\bar {D}_j$ be the vector fields given by $\bar {D}_j=(x'_j-x_j)D^\phi _j$ for $j\in \{2,\dots ,m\}$. We define 
\begin{equation*}
\begin{aligned}
a_1 & :=a\\
a_2 & :=\exp(\bar D _2)(a_1)\\
a_3 & :=\exp(\bar D _3)(a_2)\\
  \vdots & \\
a_m & :=\exp(\bar D_m)(a_{m-1}).
\end{aligned}
\end{equation*} 
More precisely for $j\in \{2,\dots ,m\}$
\begin{equation*}
a_j=(x'_2,\dots ,x'_j,x_{j+1},\dots ,x_m,y^{a_j})
\end{equation*} 
with
\begin{equation*}
\begin{aligned}
y^{a_j}_s & = y_s+\sum_{l=2}^{j} \left( b_{l1}^{(s)} \int_0^{x'_l-x_l} \phi \Big(\exp(rD^\phi _l (a_{l-1}))\Big)\, dr +\frac{1}{2}(x'_l-x_l)\Big(\sum_{ h=2 }^{l} x'_h b_{lh}^{(s)}+\sum_{ h=l+1}^{m} x_h b_{lh}^{(s)}\Big)  \right)\\
& = y^{a_{j-1}}_s +\left( b_{j1}^{(s)} \int_0^{x'_j-x_j} \phi \Big(\exp(rD^\phi _j (a_{j-1}))\Big)\, dr +\frac{1}{2}(x'_j-x_j)\Big(\sum_{ h=2 }^{j} x'_h b_{jh}^{(s)}+\sum_{ h=j+1}^{m} x_hb_{jh}^{(s)}\Big)  \right)
\end{aligned}
\end{equation*} 
for $s=1,\dots ,n$. By Lemma \ref{lemma5.1bcsc}  $(-\delta ,\delta )\ni r\mapsto  \phi \Big(\exp(rD^\phi _j (a_{j-1}))\Big)$ is Lipschitz for all $j=2,\dots , m$ and so $a_2,\dots ,a_m $ are well defined if $a,a'\in  \mathcal U(b,\delta )$ for a sufficiently small $\delta $.
 
We observe that
\begin{equation}\label{rimanente5}
|\phi (a')-\phi (a)| \leq |\phi (a')-\phi (a_m)| + \sum_{l=2}^{m} |\phi (a_l)-\phi (a_{l-1} )| 
\end{equation} 
and by Lemma $\ref{lemma5.1bcsc} $ there is $C>0$ such that
\begin{equation}\label{rimanente4}
 \sum_{l=2}^{m} |\phi (a_l)-\phi (a_{l-1} )| \leq  \sum_{l=2}^{m} C  | x'_l - x_l |_{\R^{m-1}} \leq  C \|\hat \phi(i(a))^{-1}i(a)^{-1}i(a')\hat \phi(i(a))  \|.
\end{equation} 
Hence in order to establish $\eqref{lipsch}$ we show that there exists $C_1>0$ such that
\begin{equation}\label{rimanenteok1}
|\phi (a')-\phi (a_m)|\leq C_1 \|\hat \phi(i(a))^{-1}i(a)^{-1}i(a')\hat \phi(i(a))  \|.
\end{equation} 
Recalling that $a_m=(x',y^{a_m})$ and using $\phi$ is locally $1/2 $-H\"older con\-ti\-nuous along the vertical components  with H\"older's constant $C_h>0$, it follows
\begin{equation*}
|\phi (a')-\phi (a_m)|\leq C_h |y'-y^{a_m}|^{1/2}_{\R^{n}} 
\end{equation*} 
Moreover arguing as in the proof of the implication $4. \Rightarrow 2.$ in Theorem 5.7  in \cite{biblioDDD}, we have 

\begin{equation*}
\begin{aligned}
|y'-y^{a_m}|_{\R^{n}}    & \leq \sum_{ s=1}^{n} \biggl|y'_s-y_s+\sum_{l=2}^{m} \biggl( b_{1l}^{(s)} \int_0^{x'_l-x_l} \phi \Big(\exp(rD^\phi _l (a_{l-1}))\Big)\, dr +\\
&\quad -\frac{1}{2}(x'_l-x_l)\Big(\sum_{ i=2 }^{l} x'_i b_{li}^{(s)}+\sum_{ i=l+1}^{m} x_i b_{li}^{(s)}\Big)  \biggl) \biggl| \\
& \leq \sum_{ s=1}^{n} \biggl|y'_s- y_s +\phi (a)\sum_{ l=2 }^{m} (x'_l-x_l) b_{1l}^{(s)} -\frac{1}{2}\langle \mathcal{B}^{(s)} x, x'- x \rangle \biggl|\\
& \quad + \sum_{ s=1}^{n} \biggl| -\frac{1}{2} \sum_{l=2}^{m}(x'_l-x_l)\Big(\sum_{ i=2 }^{l} x'_i b_{li}^{(s)}+\sum_{ i=l+1}^{m} x_i b_{li}^{(s)}\Big) + \frac{1}{2} \langle \mathcal{B}^{(s)}  x,x'- x \rangle \biggl| \\
& \quad+\sum_{ s=1}^{n} \biggl| -\phi (a)\sum_{ l=2 }^{m} (x'_l-x_l) b_{1l}^{(s)} +\sum_{l=2}^{m}  b_{1l}^{(s)} \int_0^{x'_l-x_l} \phi \Big(\exp(r D^\phi _l (a_{l-1}))\Big)\, dr \biggl| \\
&  \leq  c_1 \|\hat \phi(i(a))^{-1}i(a)^{-1}i(a')\hat \phi(i(a))  \|^2 + \frac{1}{2}n\mathcal{B}_M|x'-x|^2_{\R^{m-1}} +\\
&  \quad +\sum_{ s=1}^{n} \biggl| -\phi (a)\sum_{ l=2 }^{m} (x'_l-x_l) b_{1l}^{(s)}  +\sum_{l=2}^{m}  b_{1l}^{(s)} \int_0^{x'_l-x_l} \phi \Big(\exp(r D^\phi_l (a_{l-1}))\Big)\, dr \biggl| \\
\end{aligned}
\end{equation*}
where $c_1 $ is given by \eqref{deps} and $\mathcal{B}_{M} = \max \{ b_{ij}^{(s)} \, | \,  i,j=1,\dots , m \, , s=1,\dots , n \}$.  
Note that we have used
\begin{equation*}
\begin{aligned}
\frac{1}{2} \langle \mathcal{B}^{(s)} x,  x' & - x\rangle -\frac{1}{2} \sum_{l=2}^{m} (x'_l -x_l)\Big(\sum_{ i=2 }^{l} x'_i b_{li}^{(s)} +\sum_{ i=l+1}^{m} x_i b_{li}^{(s)}\Big)  \\
&= -\frac{1}{2} \sum_{l=2}^{m} (x'_l-x_l)\Big(\sum_{ i=2 }^{l} x'_i b_{li}^{(s)}+\sum_{ i=l+1}^{m} x_i b_{li}^{(s)} -\sum_{ i=2}^{m} x_i b_{li}^{(s)} \Big) \\
& \leq \frac{1}{2}n\mathcal{B}_M|x'-x|^2_{\R^{m-1}} .
\end{aligned}
\end{equation*}
Finally, the last term
\begin{equation*}
\begin{aligned}
&  \sum_{ s=1}^{n} \biggl| -\phi (a)\sum_{ l=2 }^{m} (x'_l-x_l) b_{1l}^{(s)} 
 +\sum_{l=2}^{m}  b_{1l}^{(s)} \int_0^{x'_l-x_l} \phi \Big(\exp(r D^\phi_l (B_{l-1}))\Big)\, dr \biggl| \\
 &\qquad \leq R_1(a,a')+R_2(a,a')
\end{aligned}
\end{equation*}
where
\begin{equation*}
\begin{aligned}
R_1(a,a')& :=  \sum_{ s=1}^{n} \sum_{l=2}^{m}\biggl|   b_{1l}^{(s)} \int_0^{x'_l-x_l} \phi \Big(\exp(rD^\phi _l (a_{l-1}))\Big)\, dr  -  b_{1l}^{(s)}\phi (a_{l-1})(x'_l-x_l)   \biggl|\\
R_2(a,a') &:= \sum_{ s=1}^{n} \biggl| \sum_{l=2}^{m} b_{1l}^{(s)}(x'_l-x_l) \Big( \phi (a_{l-1})-\phi (a)\Big)  \biggl| \\
\end{aligned}
\end{equation*}

We would like to show that there exist $C_1, C_2 >0$ such that
\begin{equation}\label{r1}
R_1(a,a') \leq C_1 |x'-x|^2_{\R^{m-1}} 
\end{equation}
\begin{equation}\label{r2}
R_2(a,a') \leq C_2 |x'-x|^2_{\R^{m-1}} 
\end{equation}
for all $a,a'\in  \mathcal U(b,\delta )$ and consequently
\begin{equation*}
|y'-y^{a_m}|_{\R^{n}}   \leq  c_1 \|\hat \phi(i(a))^{-1}i(a)^{-1}i(a')\hat \phi(i(a))  \|^2 + \left( \frac{1}{2}n\mathcal{B}_M   +C_1+C_2 \right) |x'-x|^2_{\R^{m-1}}  
\end{equation*}
Hence there is $ C_3>0$ such that
\begin{equation}\label{rimanente3}
|y'-y^{a_m}|^{1/2} _{\R^{n}}  \leq  C_3   \|\hat \phi(i(a))^{-1}i(a)^{-1}i(a')\hat \phi(i(a))  \|,
\end{equation}
i.e. \eqref{rimanenteok1} is true. 
We start to consider $R_1(a,a')$. Fix $l=2,\dots , m$. For $t\in [-\delta , \delta ]$ 
we define 
\begin{equation*}
g_l (t):=\sum_{s=1}^{n} b_{1l}^{(s)} \biggl(\int_0^t \phi (\exp (rD^\phi_l)(a_{l-1}))\, dr - t \phi (a_{l-1}) \biggl)
\end{equation*}
Observe that
\begin{equation*}
\begin{aligned}
b_{1l}^{(s)}\int_0^t \left(  \phi (\exp (rD^\phi_l)(a_{l-1})) -  \phi (a_{l-1})  \right) \, dr  = O(t^2)
\end{aligned}
\end{equation*}
and so there is $C_l>0$ such that
\begin{equation*}
|g_l (t)| \leq C_l t^2, \hspace{0,5 cm} \forall t\in [-\delta , \delta ]
\end{equation*}
Hence set $t=x'_l-x_l$ we get
\begin{equation*}
|g_l (x'_l-x_l)| \leq C_l (x'_l-x_l)^2
\end{equation*}
and consequently $\eqref{r1}$ follows from 
\begin{equation*}
\sum_{l=2}^{m} |g_l (x'_l-x_l )| \leq \sum_{l=2}^{m} C_l (x'_l-x_l)^2 \leq   C_1|x'-x|_{\R^{m-1}} ^2.
\end{equation*}

Now we consider $R_2(a,a')$. Observe that
\begin{equation*}
\begin{aligned}
 \sum_{ s=1}^{n} \biggl| \sum_{l=2}^{m} b_{1l}^{(s)}(x'_l& -x_l) \Big( \phi (a_{l-1})-\phi (a)\Big)  \biggl|  \leq n \mathcal{B}_M \sum_{l=2}^{m} |x'_l-x_l| \left| \phi (a_{l-1})-\phi (a)  \right| \\
& \leq  n \mathcal{B}_M \sum_{l=2}^{m} |x'_l-x_l| \Big(\sum_{h=2}^{l-1}|\phi (a_h)-\phi (a_{h-1})| \Big)\\
& \leq  n \mathcal{B}_M \sum_{l=2}^{m} |x'_l-x_l| \Big( \sum_{h=2}^{l-1}\Big| \int_0^1 (\bar D_h  \phi )(\exp (r \bar D_h (a_{h-1})))\, dr \Big|\Big) \\
& \leq  n \mathcal{B}_M \sum_{l=2}^{m} |x'_l-x_l|  \Big(\sum_{h=2}^{l-1} \Big|(x'_h-x_h) \left(D^\phi _h \phi ( a_{h-1} )+o(1) \right)  \Big| \Big)\\
& \leq  n \mathcal{B}_M C |x'-x|^2_{\R^{m-1}} 
\end{aligned}
\end{equation*}
Then $\eqref{r2}$ follows with $ C_2:=n \mathcal{B}_M C $ and \eqref{rimanente3} is true. 
 
 Finally putting together $\eqref{rimanente5}$, $\eqref{rimanente4}$ and $\eqref{rimanenteok1}$, $\eqref{lipsch}$ holds and the proof is complete.
\end{proof}

\section{A characterization of $\G$-regular  hypersurfaces} 

In this section we will prove a new characterization of $\G$-regular  hypersurfaces (see Definition \ref{Gregularsurfaces}). Here we consider 
the non linear first order system
\begin{equation}\label{sistema}
\left(D^{\phi}_2\phi,\dots, D^{\phi}_m\phi\right)= w 
\end{equation}
where $w:\mathcal O\to \R^{m-1}$ is a given \emph{continuous function} and \textit{not} measurable as opposed to the previous sections.

In  \cite{biblioDDD}, we give some equivalent conditions about $\G$-regular hypersurface inside the Carnot groups of step 2 (see Theorem 5.7 in  \cite{biblioDDD}). In this section we give another one. More precisely, we show that if a continuous map $\phi:\mathcal O\subset \R^{m+n-1}\to \R$ is locally little $1/2 $-H\"older continuous  along the vertical components, then $\phi$  is a  distributional solution in an open set $\mathcal O$ of the non linear first order system $\eqref{sistema}$ if and only if $\hat\phi: \hat{\mathcal O}\to \V$ is uniformly intrinsic differentiable in $\hat{\mathcal O}$ and consequently its graph is a $\G$-regular hypersurface. The main equivalence result is contained in Theorem \ref{CharacterizationRegularSurfaces}.
 
In order to state the equivalence result we have to be precise about the meaning of being a solution of \eqref{sistema}.  To this aim we recall a  notion of generalized solutions of systems of this kind. These generalized solutions, denoted
\emph{broad* solutions} were introduced and studied for the system \eqref{sistema} inside Heisenberg groups in \cite{biblio1, biblio27}. For a more complete bibliography we refer to the bibliography in \cite{biblio1}. Our strategy will be to prove that each continuous distributional solution of the system \eqref{sistema} is a broad* solution and vice versa (see Proposition \ref{propCharacterizationRegularSurfaces}).
 
\begin{defi}\label{defbroad*}
Let $\mcal O \subset  \R^{m+n-1}$ be open and $w:=\left(w_2,\dots,w_m\right):\mcal O \to \R^{m-1}$  a continuous function. 
With the notations of Definition \ref{defintder}   we say that $\phi\in \C(\mcal O,\R)$ is a \emph{broad* solution} in $\mcal O$ of the system
\[
\left(D^{\phi}_2\phi,\dots, D^{\phi}_m\phi\right)= w \]
 if for every $a_0\in \mcal O$ there are $0< \delta_2 < \delta_1$ and $m-1$ maps $\exp _{a_0}(\cdot D^\phi _j)(\cdot)$
\begin{equation*}
\begin{split}
\exp _{a_0}(\cdot D^\phi _j)(\cdot):[-\delta _2,\delta _2]&\times  \mathcal I(a_0,\delta_2)\to \mathcal I(a_0, \delta_1)\\
(t,&a)    \mapsto  \exp _{a_0}(tD^\phi _j)(a )
\end{split}
\end{equation*}
for $j=2,\dots , m$, where  $ \mathcal I(a_0, \delta):= \mathcal U(a_0,\delta)\cap \W$ and $\mathcal I(a_0, \delta_1)\subset\mathcal O$. Moreover these maps, called \emph{exponential maps} of the vector fields $D^\phi _2,\dots,D^\phi _m$, are required to have the following properties 
\[t\mapsto \gamma^j_a (t):=\exp_{a_0} (tD^\phi _j)(B) \in \C^1 ([-\delta _2,\delta _2], \R^{m+n-1})\]
for all $a\in  \mathcal I(a_0,\delta_2)$ and 
\[\left\{
\begin{array}{l}
\dot \gamma ^j_a =D^\phi _j \circ \gamma ^j_a\\
\gamma ^j_a (0)=a
\end{array}
\right.
\]
\[
\phi (\gamma ^j_a (t))-\phi (\gamma ^j_a (0))=\int_0^t w_j (\gamma ^j_a (r))\, dr
\]
once more 
for all $a\in  \mathcal I(a_0,\delta_2)$.
\end{defi}

\begin{prop}\label{propCharacterizationRegularSurfaces}
Let $\G := (\R^{m+n}, \cdot  , \delta_\lambda )$ be a Carnot group of step $2$ and $\V$, $\W$ the complementary subgroups defined in \eqref{5.2.0}. Let $\hat \phi :\hat{ \mathcal O} \to \V$ be a continuous map, where $\hat{ \mathcal O}$ is an open subset of $\W$ and $\phi :\mathcal O  \to \R$ is the map associated to $\hat \phi$ as in $\eqref{phipsi}$.  Then the following conditions are equivalent:
\begin{enumerate}
\item $\phi  $ is locally little $1/2 $-H\"older continuous  along the vertical components, i.e. that is
$\phi\in \C(\mathcal{O},\R )$ and for all $\mathcal O' \Subset \mathcal O$ and $(x,y), (x,y')\in \mathcal{O}'$
\begin{equation*}
\lim_{r\to 0^+} \sup_{0<|y'-y|_{\R^{n}} <r}  \, \frac{|\phi (x,y')-\phi (x,y)|}{|y'-y |_{\R^{n}}^{1/2}}  =0.
\end{equation*}
 and there exists $w\in \C(\mcal O , \R ^{m-1})$ such that $\phi $ is a broad* solution of 
\[
\left(D^{\phi}_2\phi,\dots, D^{\phi}_m\phi\right)= w, \quad \mbox{ in } \mcal O  
\]
 \item $\phi $ is locally little $1/2 $-H\"older continuous  along the vertical components  and there exists $w\in \C(\mcal O , \R ^{m-1})$ such that $\phi $ is distributional solution of
 \[
\left(D^{\phi}_2\phi,\dots, D^{\phi}_m\phi\right)= w, \quad \mbox{ in } \mcal O  
\]
\end{enumerate}
\end{prop}

\begin{proof}
$\mathbf{(1)  \implies  (2)}$. 
By Theorem $5.7$ in \cite{biblioDDD}, we know that there is a family of functions $\phi_\eps \in \mathbb C^1(\mathcal O, \R)$ such that for all $\mathcal O'\Subset \mathcal O$, 
\begin{equation}\label{convfinale02}
\phi _\epsilon \to \phi  \quad  and \quad  D_j^{{\phi _\epsilon}} {\phi _\epsilon} \to D_j^{{\phi}} {\phi}
\end{equation}
for $j=2,\dots , m$ uniformly on $\mcal O '$  as $\epsilon \to 0^+$. Then  for each $j=2,\dots , m$, $\epsilon>0$ and $\zeta \in \C^1 _c (\mcal O, \R )$

\begin{equation*} 
 \int _\mathcal O \phi_\epsilon  \left( \,X_j \zeta +\phi _\epsilon \sum_{s=1}^n b_ {j1 }^{(s)} Y_s \zeta \right)  \, d\mathcal{L}^{m+n-1} =- \int _\mathcal O D^{\phi _\epsilon }_j \phi_\epsilon \, \zeta \, d \mathcal{L}^{m+n-1}
\end{equation*}
Using \eqref{convfinale02} and getting to the limit for $\epsilon \to 0^+$ we have that 
\begin{equation*} 
 \int _\mathcal O \phi \left( \,X_j \zeta +\phi  \sum_{s=1}^n b_ {j1 }^{(s)} Y_s \zeta \right)  \, d\mathcal{L}^{m+n-1} =- \int _\mathcal O D^\phi _j \phi \, \zeta \, d \mathcal{L}^{m+n-1}
\end{equation*}
Hence $\phi $ is distributional solution of $\left(D^{\phi}_2\phi,\dots, D^{\phi}_m\phi\right)= w$ in $\mcal O $.

\medskip

$\mathbf{(2)  \implies  (1)}$.
Let $a_0\in \mcal O $, $\delta _1>0$ and $\mcal I:= \mcal U (a_0,\delta _1) \cap \W$ and $\mathcal I(a_0, \delta_1)\subset\mathcal O$. Denote 
$$K:= \sup_{(x,y)\in \mathcal I} \sum_{l=2}^{m}|x_l|,\quad M:=\|\phi \|_{\mathcal{L}^\infty (\mathcal I ,\R)}, \quad \delta _2  <\frac{\delta _1}{ 2+\frac{1}{2} K\mathcal{B}_M+M\mathcal{B}_M},$$ 
where $\mathcal{B}_{M} = \max \{ b_{jl}^{(s)} \, | \,  l,j=1,\dots , m \, , s=1,\dots , n \}$. 
 Peano's Theorem yields that for all $a=(x_j,\hat x_j,y)\in \overline{ \mathcal I (a_0,\delta _2)}$ there exists $\C^1$ function $ (\gamma _{j1},\dots , \gamma _{jn}) :[-\delta _2,\delta _2]\to \R ^n$ such that
\begin{equation*}
\gamma ^j_{a}(t)=(x_j+t,\hat x_j , \gamma _{j1} (t),\dots , \gamma _{jn}(t))\in \mathcal I \quad \mbox{for } t\in [-\delta _2,\delta _2],
\end{equation*}
and $ \gamma _{js}(t)$ is a solution of the Cauchy problem
 \begin{equation*}\label{prcauchy02}
\left\{
\begin{array}{l}
\dot \gamma _{js}(t)=\frac{1}{2} \sum_{l=2}^{m}(\hat x_j)_l b_{jl}^{(s)}+b_{j1}^{(s)} \phi (\gamma _a^j (t)), \quad \mbox { for  } t\in [-\delta _2,\delta _2]\\
\\
\gamma _{js}(0)=y_s, \\
\end{array}
\right.
\end{equation*}
for $s=1,\dots ,n$.  It is clear that if $b_ {j1 }^{(s)}= 0$, then $ \gamma _{js}(t)=y_s+\frac{1}{2}t \sum_{l=2}^{m}(\hat x_j)_l b_{jl}^{(s)}$ for  $t\in [-\delta _2,\delta _2]$. On the other hand, if $b_ {j1 }^{(s)}\ne 0$, then the map $\phi  (\gamma _a^j (\cdot ))$ satisfies the following ODE
 \begin{equation*}\label{odeserveora}
\frac{d}{dt} \left(b_ {j1 }^{(s)} \phi  (\gamma _a^j (t)) \right) = b_ {j1 }^{(s)} w_{j}  (\gamma _a^j (t)) \\
\end{equation*}
 with $t\in [-\delta , \delta ]$ for some $\delta >0$. Indeed, we can repeat verbatim the proof of Lemma \ref{lemma5.1bcsc} and so we obtain \eqref{lipottobre} with $b_{j1}^{(s)}\ne 0$. Moreover because $\phi$ is locally little $1/2 $-H\"older continuous along the vertical components, upon dividing  \eqref{lipottobre} by $\epsilon ^2$ and getting to the limit to $\epsilon \to 0$, we have that
\begin{equation*}
\begin{aligned}
&\phi  (\gamma _a^j (t_2)) - \phi  (\gamma _a^j (t_1)) - \int_{t_1}^{t_2}   w_j  (\gamma _a^j (t)) \, dt = 0 \\
\end{aligned}
\end{equation*}
for $t_1,t_2 \in [-\delta ,\delta ]$ with $\delta >0$ and $t_1<t_2$.

In particular $\phi (\gamma _a^j (\cdot)) $ and $\dot \gamma _{js} (\cdot )$ are $\C^1([-\delta _3,\delta _3] , \R)$ for $s=1,\dots , n$ where $\delta _3:=\min\{ \delta , \delta _2\} $. Therefore the curve $\gamma ^j_{a}:[-\delta _3,\delta _3] \to \mcal I$ satisfies  the conditions of the Definition $\ref{defbroad*}$ for each $a\in  \mcal I (a_0,\delta _3) := \mcal U (a_0,\delta _3) \cap \W$ and $\mathcal I(a_0, \delta_3)\subset\mathcal O$.

 Then, for each $j=2,\dots , m$, $\exp _{a_0}(\cdot D^\phi _j)(\cdot ):[-\delta_3,\delta _3]\times \mcal I (a_0,\delta _3) \to \mcal I$ defined as $\exp _{a_0}(t D^\phi _{j})(a):= \gamma ^j_{a}(t)$ is a family of exponential maps at $a_0$ which we were looking for.  This completes the proof of the implication $(4)\implies (5)$.

 \end{proof}

By Proposition \ref{propCharacterizationRegularSurfaces} and Theorem $5.7$ in \cite{biblioDDD}, it immediately follows
\begin{theorem}\label{CharacterizationRegularSurfaces}
Let $\G := (\R^{m+n}, \cdot  , \delta_\lambda )$ be a Carnot group of step $2$ and $\V$, $\W$ the complementary subgroups defined in \eqref{5.2.0}. Let $\hat \phi :\hat{ \mathcal O} \to \V$ be a continuous map, where $\hat{ \mathcal O}$ is an open subset of $\W$ and $\phi :\mathcal O  \to \R$ is the map associated to $\hat \phi$ as in $\eqref{phipsi}$.  Then the following conditions are equivalent:
\begin{enumerate}
\item $\graph {\hat\phi} $ is a $\G$-regular hypersurface and for all $a\in \graph{\hat \phi}$ there is $r=r(a)>0$ and $f\in \mathbb C_{\G}^1(\mathcal U(a,r)),\R)$ with $X_1f>0$ such that
\[
\graph{\hat \phi}\cap \mathcal U(a,r)=\{ p:f(p)=0\}
\]
\item  $\hat\phi $ is u.i.d. in $ \hat {\mcal O}$.
\item $D_j^{\phi}\phi$  interpreted in distributional sense is a continuos function in $\mathcal O$ and for $0<\eps<1$ there is a family of functions $\phi_\eps\in \mathbb C^1(\mathcal O, \R)$ such that for all $\mathcal O'\Subset \mathcal O$, 
\begin{equation*}
\phi _\epsilon \to \phi  \quad  and \quad  D_j^{{\phi _\epsilon}} {\phi _\epsilon} \to D_j^{{\phi}} {\phi}
\end{equation*}
for $j=2,\dots, m$, uniformly on $\mcal O '$  as $\epsilon \to 0^+$.
\item $\phi  $ is locally little $1/2 $-H\"older continuous  along the vertical components and there exists $w\in \C(\mcal O , \R ^{m-1})$ such that $\phi $ is a broad* solution of 
\[
\left(D^{\phi}_2\phi,\dots, D^{\phi}_m\phi\right)= w, \quad \mbox{ in } \mcal O  
\]
 \item $\phi $ is locally little $1/2 $-H\"older continuous  along the vertical components and there exists $w\in \C(\mcal O , \R ^{m-1})$ such that $\phi $ is distributional solution of
 \[
\left(D^{\phi}_2\phi,\dots, D^{\phi}_m\phi\right)= w, \quad \mbox{ in } \mcal O  
\]
\end{enumerate}
\end{theorem}

\begin{rem}
We recall that in Theorem $5.7$ in \cite{biblioDDD}, we consider a locally little $1/2 $-H\"older continuous  map  but in the proof we just use that $\phi  $ is locally little $1/2 $-H\"older continuous  along the vertical components.
\end{rem}

\end{document}